\documentclass{amsart}
\usepackage[utf8]{inputenc}  
\usepackage[T1]{fontenc}

\usepackage{comment}
\usepackage{amsthm}
\usepackage{amssymb}
\usepackage{pifont}
\usepackage{amsmath}
\usepackage{mathrsfs}
\usepackage{dsfont}
\usepackage{yhmath}
\usepackage{listings}
\usepackage{enumitem}
\usepackage{float}
\usepackage{graphicx}
\usepackage{changepage}
\usepackage{listings}
\usepackage{calrsfs}

\usepackage{pgf,tikz}
\usetikzlibrary{arrows}
\usetikzlibrary{patterns}

\usepackage[all]{xy}

\usepackage{subcaption}
\usepackage{geometry}
\usepackage{ifthen}

\usepackage[colorlinks=true,linkcolor=blue,citecolor=green]{hyperref}

\newcommand{\R}{\mathbb{R}}
\newcommand{\N}{\mathbb{N}}
\newcommand{\Z}{\mathbb{Z}}
\newcommand{\Q}{\mathbb{Q}}
\newcommand{\C}{\mathbb{C}}

\renewcommand{\P}{\mathbb{P}}

\newcommand{\GL}{\mathrm{{GL}}}
\newcommand{\SL}{\mathrm{{SL}}}
\newcommand{\PGL}{\mathrm{{PGL}}}

\newcommand{\rd}{\mathrm{{rd}}\,}

\def\bsys{\left\{\begin{array}}
\def\esys{\end{array}\right.}

\newcommand{\eps}{\varepsilon}

\renewcommand{\emptyset}{\varnothing}

\newcommand{\de}{\,\mathrm{d}}

\newcommand{\id}{\mathrm{id}}

\newcommand{\Aut}{\mathrm{Aut}}
\newcommand{\PsAut}{\mathrm{PsAut}}
\newcommand{\Nef}{\mathrm{Nef}}
\newcommand{\Mov}{\mathrm{Mov}}

\theoremstyle{plain}
\newtheorem{theorem}{Theorem}[section]
\newtheorem{lemma}[theorem]{Lemma}
\newtheorem{minilemma}[theorem]{Mini-lemma}
\newtheorem{proposition}[theorem]{Proposition}
\newtheorem{corollary}[theorem]{Corollary}

\theoremstyle{definition}
\newtheorem{definition}[theorem]{Definition}
\newtheorem{notation}[theorem]{Notation}

\newtheorem{remark}[theorem]{Remark}
\newtheorem{example}[theorem]{Example}
\newtheorem{question}[theorem]{Question}
\newtheorem{conjecture}[theorem]{Conjecture}

\newcommand{\lem}[2][None]
{\begin{lemma}\ifthenelse{\equal{#1}{None}}{}{\label{#1}}
#2
\end{lemma}}

\newcommand{\minlem}[2][None]
{\begin{minilemma}\ifthenelse{\equal{#1}{None}}{}{\label{#1}}
#2
\end{minilemma}}

\newcommand{\prop}[2][None]
{\begin{proposition}\ifthenelse{\equal{#1}{None}}{}{\label{#1}}
#2
\end{proposition}}

\newcommand{\cor}[2][None]
{\begin{corollary}\ifthenelse{\equal{#1}{None}}{}{\label{#1}}
#2
\end{corollary}}

\newcommand{\theo}[2][None]
{\begin{theorem}\ifthenelse{\equal{#1}{None}}{}{\label{#1}}
#2
\end{theorem}}

\newcommand{\defi}[2][None]
{\begin{definition}\ifthenelse{\equal{#1}{None}}{}{\label{#1}}
#2
\end{definition}}

\newcommand{\nota}[2][None]
{\begin{notation}\ifthenelse{\equal{#1}{None}}{}{\label{#1}}
#2
\end{notation}}

\newcommand{\rem}[2][None]
{\begin{remark}\ifthenelse{\equal{#1}{None}}{}{\label{#1}}
#2
\end{remark}}

\newcommand{\expl}[2][None]
{\begin{example}\ifthenelse{\equal{#1}{None}}{}{\label{#1}}
#2
\end{example}}

\title[On the cone conjecture for finite quotients]{Well-clipped cones under finite quotients and applications to the cone conjecture}
\author{C\'ecile Gachet}
\address{Ruhr-Universit\"at Bochum, Universit\"atsstr. 150, 
44801 Bochum, Germany
}
\email{cecile.gachet@ruhr-uni-bochum.de}

\begin{document}

\begin{abstract}
We introduce a property of convex cones, being ``well-clipped'', that is inspired by the work of several complex algebraic geometers on the Morrison-Kawamata cone conjecture. That property is satisfied by movable cones of divisors on various complex projective varieties of Calabi--Yau type, such as abelian varieties and projective hyperkähler manifolds. The property of being well-clipped has the advantage to descend under taking invariants by a finite group action, and to be stable under direct sums. In the class of well-clipped cones, we also provide a simple characterization of those cones that admit a rational polyhedral fundamental domain under some natural group action. 

We use this framework to prove the movable Morrison--Kawamata cone conjecture for finite quotients of various projective varieties of Calabi--Yau type, notably products of complex projective primitive symplectic varieties, abelian varieties, and smooth rational surfaces underlying klt Calabi--Yau pairs. This entails Enriques manifolds in the sense of Oguiso--Schröer. We also provide Galois descent statements implying the movable Morrison--Kawamata cone conjecture for abelian varieties over arbitrary perfect fields.
\end{abstract}

\maketitle

\section{Introduction}

\subsection{Motivation} Birational geometry rarely behaves well under finite covers. However, finite covers offer meaningful ways to relate certain varieties to one another, and appear in numerous geometric constructions and classifications: Fundamental groups and uniformization results, the increasingly popular framework of orbifold pairs, the Enriques--Kodaira classification of complex projective surfaces, and the celebrated Beauville--Bogomolov decomposition theorem \cite{Beauv84} all involve finite covers in some way, at the very least in the form of quotients by finite group actions.
Modern tools have evolved to tackle this ubiquitous challenge, most notably the equivariant minimal model program (MMP) developed in works by Manin, Iskovskikh, Prokhorov (see notably \cite{Pro21}). From a naive perspective, it is not surprising to be able to detect birational properties of a finite quotient $X/G$ from the initial variety $X$: After all, a lot of birational information is entailed in the nef and movable cones of a variety, which behave well under finite quotients as
 $$\Nef(X/G) = \Nef(X) \cap N^1(X) ^G
 ,\quad
 \overline{\Mov}(X/G) = \overline{\Mov}(X) \cap N^1(X) ^G,$$ 
 where $N^1(X) ^G$ denotes the $G$-invariant subspace of the real Néron-Severi space.
 But this naive perspective soon reaches its limits: By Mori's cone theorem, relevant birational information is typically provided by extremal faces of the dual cone of the nef cone that are $K_X$-negative. This elicits two questions: What about faces that are not $K_X$-negative? More importantly, how to relate extremal faces of a given cone with extremal faces of its $G$-invariant slice?
 
\medskip


The Morrison--Kawamata cone conjecture was made by Morrison \cite{Mor}, notably studied by Kawamata \cite{Kaw97}, and reformulated by Totaro \cite{To10}, and offers a prediction anologous to Mori's cone theorem for projective varieties of Calabi--Yau type (which can be thought of as varieties $X$ with $K_X$ non-positive). It can be summarized as follows.

\begin{definition}\label{def-pair}
A {\it pair} $(X,\Delta)$ is the data of a normal $\Q$-factorial complex projective variety $X$ and of an effective $\Q$-divisor $\Delta$ on $X$ with coefficients in $[0,1]$.
A {\it Calabi--Yau pair} is a pair $(X,\Delta)$ such that the $\Q$-Cartier divisor $K_X+\Delta$ is numerically trivial.
A variety $X$ is called {\it of Calabi--Yau type} if there exists a klt Calabi--Yau pair $(X,\Delta)$. (For the definition of a {\it klt} pair, we refer the reader to \cite[Definition 2.34]{KollarMori}.)
\end{definition}

\begin{conjecture}[Morrison--Kawamata cone conjecture]\label{conj-cc}
Let $(X,\Delta)$ be a klt Calabi--Yau pair. There exist rational polyhedral fundamental domains both for
\begin{enumerate}
\item (movable cone conjecture) the group $\PsAut^*(X,\Delta)$ acting on the positive movable cone $\Mov^+(X)$;
\item (nef cone conjecture) the group $\Aut^*(X,\Delta)$ acting on the positive nef cone $\Nef^+(X)$.
\end{enumerate}
Here, for a cone $\mathcal{C}$ in a vector space $V$ with a rational structure $V_{\Q}$, we define the positive cone
$$\mathcal{C}^+:={\rm Conv}_{\R}(\overline{\mathcal{C}}\cap V_{\Q}).$$
We denote by $\PsAut(X)$ the group of birational automorphisms of $X$ that are isomorphisms in codimension one, by $\mathrm{(Ps)Aut}(X,\Delta)$ the subgroup of $\mathrm{(Ps)Aut}(X)$ that preserves the support of $\Delta$, and by $H^*$ the image of a subgroup $H$ of $\PsAut(X)$ by the representation
$$\rho:g\in \PsAut(X)\mapsto g^*\in \mathrm{GL}(N^1(X)).$$
\end{conjecture}

This paper concerns the descent of the movable Morrison--Kawamata cone conjecture under finite quotients, as expressed in the following question.

\begin{question}\label{qu-ccquotients}
Let $(X,\Delta)$ be a pair. For a finite subgroup $G$ of $\Aut(X,\Delta)$, consider the {\it quotient pair} $(X/G,\Delta_G)$, where $\Delta_G$ is the $\Q$-divisor such that $K_X+\Delta$ is the pullback of $K_{X/G}+\Delta_G$ by the quotient map. 
Does the movable cone conjecture for $(X,\Delta)$ imply the movable cone conjecture for $(X/G,\Delta_G)$, for every finite subgroup $G$?
\end{question}

It is worth noting that if $(X,\Delta)$ is a klt Calabi--Yau pair, then so is $(X/G,\Delta_G)$.

\subsection{Main results} We give a positive answer to Question \ref{qu-ccquotients} for a large class of pairs $(X,\Delta)$ for which the movable cone conjecture is currently known to be satisfied.

\begin{theorem}\label{thm-ccforexplicitvars}
Let $(X,\Delta)$ be a klt Calabi--Yau pair defined over $\C$ that decomposes as
$$X = A \times \prod_{i=1}^r Y_i \times \prod_{j=1}^s S_j,\quad 
\Delta = \sum_{j=1}^s p_j^*\Delta_j,$$
where $A$ is an abelian variety, each $Y_i$ is a primitive symplectic variety with canonical singularities and with $b_2(Y_i)\ge 5$ or $\dim Y_i = 2$, and each $S_j$ is a smooth rational surface underlying a klt Calabi--Yau pair $(S_j,\Delta_j)$. Then, the movable cone conjecture holds for every quotient pair $(X/G,\Delta_G)$ of the pair $(X,\Delta)$ by a finite subgroup $G$ of $\Aut(X,\Delta)$. Moreover, there are finitely many isomorphism classes of pairs $(Y,\Delta_Y)$ obtained by small $\Q$-factorial modifications of $(X/G,\Delta_G)$, and the nef cone conjecture holds for each of them.
\end{theorem}

Primitive symplectic varieties are introduced in \cite[Definition 3.1]{BL18}, and generalize hyperkähler manifolds in singular settings. Note that Theorem \ref{thm-ccforexplicitvars} also holds under the weaker assumption that each $Y_i$ is a primitive symplectic variety with canonical singularities that admits a terminalization $\hat{Y_i}$ of second Betti number $b_2(\hat{Y_i})\ge 5$ or that has dimension $2$. Theorem \ref{thm-ccforexplicitvars} draws its motivation from a large body of work on descent results for the cone conjecture under finite group actions, which we present in Section \ref{sec-lit}.

Our proof of Theorem \ref{thm-ccforexplicitvars} starts with an observation: In most instances of pairs for which the cone conjecture has been proved, much more structure was in fact discovered on the cone under study than one may expect: Movable cones of abelian varieties are self-dual homogeneous cones by \cite{PS12a}, and so are movable cones of their smooth finite quotients by \cite{MQ}; Nef cones of K3 surfaces are cut out by the orthogonal hyperplanes to $(-2)$-curve classes in the hyperbolic cone given by the self-intersection form by \cite{Sterk85}; Movable cones of projective hyperkähler manifolds are described very precisely by Markman in \cite{Markman13}, and so are movable cones of projective primitive symplectic varieties with terminal $\Q$-factorial singularities by \cite{LMP24}. 

With all these results in mind, we define a new class of cones, the {\it well-clipped} cones, meant to subsume all the afore-mentioned algebro-geometric examples, to be stable under direct sum, and to behave better than the Morrison--Kawamata cone conjecture under finite quotients.

\begin{definition}[ = Definition \ref{def-nicelycutout} later]
Let $V=V_{\Z}\otimes \R$ be a finite dimensional real vector space with a preferred lattice. A full-dimensional convex cone $\mathcal{C}$ is {\it well-clipped} if there are a self-dual homgeneous cone $\mathcal{A}$ in $V$ and a set of hyperplanes $({H_i})_{i\in I}$ of $V$ such that
$$\overset{\circ}{\mathcal{C}}=\mathcal{A}\cap \bigcap_{i\in I} {H_{i,+}},$$
where $H_{i,+}$ denotes a connected component of $V\setminus H_i$, and the following three assumptions are satisfied:
\begin{enumerate}
\item[(i)] Decomposing $\mathcal{A} = \bigoplus_{j\in J} \mathcal{A}_j$ into $\R$-indecomposable summands, every hyperplane $H_i$ is of the form
$$H_i = H_i\cap {\rm Span}_{\R}\, \mathcal{A}_{j(i)}
\oplus
\bigoplus_{k\neq j(i)}{\rm Span}_{\R}\, \mathcal{A}_k,$$
with the remaining cone $\mathcal{A}_{j(i)}$ of hyperbolic type and defined over $V_{\Q}$.
\end{enumerate}
Fixing an $\mathrm{Aut}(\mathcal{A},V_{\Z})$-invariant and $V_{\Z}$-integral quadratic form $q$ that is a direct sum of hyperbolic forms on the linear spans of the $(\mathcal{A}_{j(i)})_{i\in I}$ and of a positive definite quadratic form on the other summands' spans, 
\begin{enumerate}
\item[(ii)] For every $i\in I$, the $q$-orthogonal reflection
$\sigma_i$ fixing the hyperplane $H_i$ preserves the lattice $V_{\Z}$.
\item[(iii)] For every $i,k\in I$ such that $H_i\neq H_k$, any elements $e_i,e_k$ perpendicular to $H_i,H_k$ and on their negative sides satisfy $q(e_i,e_k) \ge 0$.
\end{enumerate}
\end{definition}

With this terminology, Theorem \ref{thm-ccforexplicitvars} essentially follows from the following result applied over the field $k=\overline{k}=\C$.

\begin{theorem}\label{thm-main}
Let $(X,\Delta)$ be a pair defined over a perfect field $k$. Let $(\overline{X},\overline{\Delta})$ denote their base change to the algebraic closure $\overline{k}$ of $k$. Suppose that the movable cone $\Mov(\overline{X})$ is well-clipped inside a self-dual homogeneous $\PsAut(\overline{X},\overline{\Delta})$-invariant cone. Suppose also that the pair $(\overline{X},\overline{\Delta})$ satisfies the movable cone conjecture. Then the movable cone conjecture holds for every quotient pair $(X/G,\Delta_G)$ of $(X,\Delta)$ by a finite subgroup $G$ of $\Aut_k(X,\Delta)$. 
\end{theorem}

Note that Theorem \ref{thm-main} encompasses both geometric, arithmetic, and mixed equivariant descent results. We make note of the following consequence, used notably by \cite[Corollary 8.5]{Boundedness} with $k$ a function field:

\begin{corollary}\label{cor-boundedness}
Let $X$ be a torsor over an abelian variety over a perfect field $k$. Then the movable cone conjecture holds for $X$.
\end{corollary}

\medskip

A key result involved in the proof of Theorem \ref{thm-main} can be summarized under the motto ``Well-clipped cones behave themselves under finite quotients, and so does the cone conjecture for them.'' Let us state it more formally. Here, we denote by $C_{\Gamma}(G)$ the centralizer of a subgroup $G$ in a group $\Gamma$, and for a representation $G< \mathrm{GL}(V)$, we denote by
$$\rho^G: \phi\in C_{\mathrm{GL}(V)}(G) 
\mapsto \phi|_{V^G} \in \mathrm{GL}(V^G)$$
the representation induced by restriction to the invariant subspace $V^G$.

\begin{theorem}\label{thm-convexmain}
Let $\mathcal{C}$ be a well-clipped cone in a self-dual homogeneous cone $\mathcal{A}$ in $V=V_{\Z}\otimes \R$. Let $\Gamma$ denote the group $\Aut(\mathcal{C}^{\circ},\mathcal{A},V_{\Z})$. Let $G$ be a subgroup of $\Gamma$ whose action on $V$ has finite orbits. Then, the invariant cone $\mathcal{C}^G$ is well-clipped. 
Furthermore, if there exists a rational polyhedral fundamental domain for the action of $\Gamma$ on $\mathcal{C}^+$, then there exists a rational polyhedral fundamental domain for the action of the group $\rho^G(C_{\Gamma}(G))$ on ${\mathcal{C}^G}^+$.
\end{theorem}

A small technical drawback of our definition of a well-clipped cone is that, although the invariant cone $\mathcal{A}^G$ always is self-dual homogeneous, the cone $\mathcal{C}^G$ often is not well-clipped in $\mathcal{A}^G$, but in a smaller self-dual homogeneous cone.

To prove Theorem \ref{thm-convexmain}, we use a mix of convex geometry, computations with Coxeter groups generated by hyperplane reflections, and, surprisingly yet unavoidably, the Koecher--Vinberg equivalence of categories between self-dual homogeneous cones and formally real Jordan algebras. Indeed, it seems difficult to show that an invariant subcone $\mathcal{A}^G$ of a self-dual homogeneous cone $\mathcal{A}$ remains homogeneous directly, without using either the Koecher--Vinberg equivalence or the classification of self-dual homogeneous cones (whose proof relies on the Koecher--Vinberg equivalence anyways).

\subsection{Relationship with previous work}\label{sec-lit}

The literature on descending the cone conjecture under finite quotients is abundant. It can be hard to navigate, since some landmark results are proven independently by different authors, several results are claimed with incomplete proofs, and many seemingly related papers use in fact very distinct sets of assumptions on the variety $X$ and finite group $G$, thereby obtaining starkly different results. 
Let us situate Theorem \ref{thm-ccforexplicitvars} within this context. We focus on the following setting: $X$ is a smooth complex projective variety with trivial canonical bundle and $G<\Aut(X)$ is a finite group.

The first known descent results concern the case when $X$ is a K3 surface and $G$ is generated by an antisymplectic involution without fixed points on $X$, i.e., the quotient $X/G$ is a smooth Enriques surface.
The cone conjecture was proven for Enriques surfaces independently by Horikawa and Namikawa \cite{Hori,Nam85}.
Note that admitting an antisymplectic involution without fixed points is a rather restrictive condition on the surface $X$. For instance, it forces its Picard number $\rho(X)$ to be at least $10$; a stronger lattice-theoretic constraint on $N^1(X)_{\Z}$ is given in \cite[Section 2]{Nam85}.

A more general setup is studied by Oguiso--Sakurai: There, $X$ is a K3 surface and $G<\Aut(X)$ is any finite group. Oguiso--Sakurai establish an equivariant Torelli theorem in \cite[Section 1]{OS01}, which can easily be shown to imply the cone conjecture for the quotient $X/G$. Notably, their result allows to descend the cone conjecture under certain index-one covers, a fact that plays a role in Totaro's proof of the cone conjecture for klt Calabi--Yau surface pairs \cite[Theorem 3.3]{To10}.

\medskip

In higher dimension, the literature focuses on finite groups acting freely on smooth varieties. The freeness assumption allows to derive restrictions on the possible varieties $X$ and groups $G$, and sometimes even to classify them. For instance, Oguiso--Sakurai prove the cone conjecture for finite étale quotients of abelian threefolds by classification and explicit geometric considerations in \cite[Theorems 0.2]{OS01}. By the same methods, they also prove the cone conjecture for finite étale quotients of threefolds of the form $S\times E$, when $S$ is a K3 surface and $E$ an elliptic curve (\cite[Theorem 2.23]{OS01}). Unfortunately, this approach does not generalize well to non-free group actions.

Generalizing this work, Monti--Quedo \cite{MQ} prove the cone conjecture for finite étale quotients of abelian varieties in any dimension.

In \cite{PacSarti}, Pacienza--Sarti study the descent of the cone conjecture under finite étale quotients of hyperkähler manifolds. They prove that the cone conjecture descends under a free finite group action $G$ on a hyperkähler manifold $X$ under the assumption that $G$ acts trivially on the Picard group of $X$; see \cite[Section 3.2]{PacSarti}.
They further claim that the cone conjecture descends under a free action of prime order on any hyperkähler manifold \cite[Theorem 1.2]{PacSarti}; however, their proof is incomplete. Let us outline the gap. First, the authors correctly establish the cone conjecture for a very general Enriques manifold of prime index. Then, they claim to derive the cone conjecture for any prime Enriques manifold by a specialization argument in \cite[Theorem 3.5]{PacSarti}. Controlling the behaviour of the nef cone and automorphism group of an Enriques manifold under specialization is notoriously difficult, even in dimension $2$; see the account of Barth--Peters \cite[Page 395]{BP83} and Brandhorst--Shimada \cite[Introduction]{BS20} on Enriques surfaces. Yet the proof of \cite[Theorem 3.5]{PacSarti} does not address this known challenge: it identifies the cone ${\rm Nef}(X_0)$ to a subcone of ${\rm Nef}(X_t)$ and the group ${\rm Aut}^*(X_t)$ to a subgroup of ${\rm Aut}^*(X_0)$ and stops there, failing to quantify how these two inclusions ``work against each other'' (\cite[Page 395]{BP83}). This gap also impacts the proofs of \cite[Theorem 1.2, Proposition 4.1, Proposition 4.2]{PacSarti}.



The descent of the nef cone conjecture under free finite group actions on smooth $K$-trivial varieties is claimed in Skauli's Master Thesis \cite[Theorem 7.2.4]{Skauli}. However, the proof relies on a flawed lemma \cite[Lemma 7.2.3]{Skauli}. Indeed, the diagram presented in the proof of \cite[Lemma 7.2.3]{Skauli} does not commute as is: It should in principle feature an additional (non-trivial) automorphism of $\P^N$. Another way to notice the issue is to note that the proof of \cite[Lemma 7.2.3]{Skauli} does not use the facts that $A$ is $K$-trivial and $G$ acts freely. Yet, without these assumptions, the lemma fails for $A = \P^1$ with any non-trivial finite group action $G$: a general $\alpha\in\PGL(2,\C) = \Aut(A)$ does not commute to $G$. The claimed descent result remains open.

\medskip

In this rich context, Theorem \ref{thm-ccforexplicitvars} provides a complete proof of the cone conjecture for Enriques manifolds, avoiding the incomplete specialization argument of \cite{PacSarti}. Theorem \ref{thm-ccforexplicitvars} further applies to non-free finite quotients of primitive symplectic varieties. Examples of such finite quotients are much easier to construct than Enriques manifolds; see for instance Example \ref{ex-hk}. Theorem \ref{thm-ccforexplicitvars} also generalizes the results of \cite{MQ} to non-free group actions on abelian varieties; see Example \ref{ex-jac} for an application. Finally, Theorem \ref{thm-ccforexplicitvars} applies to diagonal group actions in the spirit of \cite[Theorem 2.23]{OS01}; see Example \ref{ex-ueh}.


\medskip

Theorem \ref{thm-main} and the slightly more general Theorem \ref{thm-maintechnical}
generalize several known results of Galois descent. In dimension $2$, Bright--Logan--van Lujik \cite{BvL} use Galois descent to establish the cone conjecture for K3 surfaces over fields of characteristics other than 2, and Xu \cite[Section 2.7, Lemma 23]{Fulin} uses mixed equivariant descent to establish the cone conjecture for Enriques surfaces over fields of characteristic zero. In the continuity of these papers, Theorem \ref{thm-main} contributes to \cite{Brand} in establishing the cone conjecture for Enriques surfaces over fields of odd characteristic and perfect fields of characteristic $2$. As mentioned earlier, Corollary \ref{cor-boundedness} settles the case of abelian varieties over perfect fields. 
Theorem \ref{thm-main} can also be applied to primitive symplectic varieties, in order to descend the cone conjecture from an algebraically closed field to suitable subfields. This recovers the following known results: Over fields characteristic zero, the cone conjecture is proven for irreducible holomorphic symplectic manifolds by Takamatsu \cite[Theorem 1.0.5]{Taka}, and for primitive symplectic varieties by Fu--Li--Takamatsu--Zou \cite[Theorem 3.5]{FLTZ}. Very recently, Faucher \cite[Theorem B]{Faucher} reproved and used this last result to settle the relative cone conjecture for fibrations in primitive symplectic varieties, assuming the existence of good minimal models.

\subsection{Limitations and future directions}

We conclude this introduction by mentioning some limitations to our approach. In principle, Theorem \ref{thm-main} may descend the movable cone conjecture under finite quotients for varieties $X$ beyond the scope of Theorem \ref{thm-ccforexplicitvars}, for instance for the blow-ups of $\P^3$ at $8$ very general points discussed in \cite{SX23} (see Example \ref{ex-listperfclipped}). This example is actually uninteresting, because it has no finite quotients. In principle, it may deform to special smooth threefolds that have interesting finite quotients, which could be studied by the recent work of Lutz \cite{Lutz} and Theorem \ref{thm-main}. In any case, we hope that this remark can inspire a discussion on the cone conjecture for finite quotients of smooth Calabi--Yau threefolds, varieties, and pairs in general. 

In light of the Beauville--Bogomolov decomposition theorem \cite{Beauv84}, a landmark result would be to extend the conclusions of Theorem \ref{thm-ccforexplicitvars} to varieties $X$ of the form
$$X = A \times \prod_{i=1}^r Y_i \times \prod_{j=1}^s Z_j,$$
where $A$ is an abelian variety, each $Y_i$ is a primitive symplectic variety, and each $Z_j$ is a smooth Calabi--Yau variety. However, the movable cone of a smooth Calabi--Yau threefold is not always well-clipped, see Example \ref{ex-schoen}.
Motivated by the singular Beauville--Bogomolov decomposition theorem \cite{GGK,Druel,HP19}, we can ask the same question for Calabi--Yau varieties with canonical singularities in the sense of \cite{GGK}.

A last unfortunate aspect of Definition \ref{def-nicelycutout} is that most non-simplicial rational polyhedral cones are not even well-clipped, see Example \ref{ex-ratpolwellclipped}. However, if $X$ is a projective variety with a rational polyhedral movable cone, then the movable cone conjecture clearly descends under any finite group action on $X$: It is a bit disappointing that this straight-forward case is not directly covered by Theorem \ref{thm-main}.In order to handle this case, and more generally movable cones cut out by a finite $\PsAut(\overline{X},\overline{\Delta})$-invariant collection of hyperplanes inside an appropriate well-clipped cone, we provide a slight technical improvement of Theorem \ref{thm-main} in the form of Theorem \ref{thm-maintechnical}. 


\medskip

\noindent {\bf Acknowledgements.} I want to thank A. Höring, H.-Y. Lin, G. Pacienza and L. Wang for discussions in Fall 2022 that convinced me that descending the Morrison--Kawamata cone conjecture under finite quotients was interesting. Although they do not appear in the final paper, examples and arguments from B. Skauli's Master thesis \cite{Skauli} were quite formative to my understanding of this question. I am also thankful to  M. Monti, K. Oguiso, A. Quedo and A. Sarti for their interest in this problem and for their questions on earlier communications, and in particular to G. Pacienza for discussions in Mai 2023 in Nancy, to C. Lehn for his comments on a first draft, to N. Tsakanikas for useful remarks on the symplectic setting during a visit to Lausanne in March 2025, and to S. Filipazzi for sharing an early draft of \cite{Boundedness}. I also thank J. Schneider for pointing me to the reference \cite{Lamy} regarding the N\'eron--Severi representation of absolute Galois groups of perfect fields. Finally, my sincere gratitude goes to the anonymous referee, both for catching a mistake in a first version of Lemma \ref{lem-surface-nicely} and for their many thoughtful comments. I feel that this paper improved greatly thanks to their feedback. I am partially supported by the projects GAG ANR-24-CE40-3526-01 and the DFG funded TRR 195 - Projektnumber@ 286237555.

\section{Preliminaries}

All cones considered are non-empty and convex in an ambient real finite-dimensional vector space. A cone is called {\it non-degenerate} if it contains no line. For a cone $\mathcal{C}$ in a real vector space $V$ with a rational structure $V_{\Q}\subset V$, we set
$$\mathcal{C}^+ := \mathrm{Conv}_{\R}(\overline{\mathcal{C}}\cap V_{\Q}),$$
where $\overline{\mathcal{C}}$ stands for the closure of $\mathcal{C}$ in $V$ in Euclidean topology. A cone is called {\it rational polyhedral} if it is spanned by a finite set of points of $V_{\Q}$.

For $\mathcal{C}$ a cone in $V=V_{\Z}\otimes \R$, we denote by $\Aut(\mathcal{C})$ the group of linear automorphisms of $V$ that preserve the cone $\mathcal{C}$, by $\Aut(\mathcal{C},V_{\Z})$ its subgroup that also preserves the lattice $V_{\Z}$, by $\Aut^{\circ}(\mathcal{C})$ the connected component of the identity in $\Aut(\mathcal{C})$ with respect to the Euclidean topology. If $\mathcal{A}$ is another cone in $V$, we denote by $\Aut(\mathcal{C},\mathcal{A},V_{\Z})$ the group of linear automorphisms of $V$ that preserve both cones $\mathcal{C}$ and $\mathcal{A}$, and the lattice $V_{\Z}$. 

If $G$ is a subgroup of $\mathrm{GL}(V)$, we denote by 
$$V^G:=\{v\in V\mid \forall g\in G,\, g(v)=v\}$$ 
the {\it invariant subspace}. For a cone $\mathcal{C}$ in $V$ that is preserved by $G$, we denote by $\mathcal{C}^G:=\mathcal{C}\cap V^G$ the {\it invariant cone}.

For a normal $\Q$-factorial complex projective variety $X$, we denote by $N^1(X)$ the finite dimensional real vector space of numerical equivalence classes of $\R$-divisors on $X$. On $X$, we say that a Cartier divisor $D$ is {\it movable} if its linear system $|D|$ is non-empty and has base locus of codimension at least $2$. In this space, we denote by $\Mov(X)$ the closure of the convex cone spanned by classes of movable Cartier divisors on $X$, and call it the {\it movable cone} of $X$. Unless otherwise stated, our lattice of choice in the Néron--Severi space $N^1(X)$ is the lattice of integral Weil divisors, which is preserved by the action of $\PsAut(X)$ and denoted by $N^1(X)_{\Z\mathrm{-Weil}}$. For any subgroup $H <\PsAut(X)$, we use the superscript $H^*$ to denote the image of $H$ under the representation
$$g\in \PsAut(X) \to g^*\in \mathrm{GL}(N^1(X)),$$
given by pullback.

Primitive symplectic varieties are introduced in \cite[Definition 3.1]{BL18}. Recall that a primitive symplectic variety is a normal compact Kähler variety $X$ with $h^1(\mathcal{O}_X) = 0$ and $H^0(X_{\rm reg},\Omega^2_{X_{\rm reg}})=\C\cdot \sigma$, where $\sigma$ is a symplectic form on $X_{\rm reg}$ that extends holomorphically (not necessarily as a symplectic form) to any resolution of singularities of $X$.

For a group $\Gamma$ and a subgroup $G<\Gamma$, we denote by
$$N_{\Gamma}(G):=\{\gamma\in\Gamma\mid \gamma G = G\gamma\},
\quad 
C_{\Gamma}(G):=\{\gamma\in\Gamma\mid \forall g\in G,\, \gamma g = g\gamma\}$$
the {\it normalizer} of $G$ in $\Gamma$ and the {\it centralizer} of $G$ in $\Gamma$, respectively.

\subsection{Self-dual homogeneous cones}

Throughout this section, a reference is \cite{FK94}.

\begin{definition}\label{def-selfdual}
Let $\mathcal{A}$ be an open non-degenerate cone in a finite dimensional real vector space $V$. 
We say that $\mathcal{A}$ is a {\it self-dual cone} if there exists a positive definite quadratic form ${\rm tr}:V\otimes V\to \R$ that induces an identification of $\mathcal{A}$ with its dual cone. 
We say that $\mathcal{A}$ is a {\it homogeneous cone} if the group of linear automorphisms of $V$ that preserve $\mathcal{A}$ acts transitively on $\mathcal{A}$.
We say that $\mathcal{A}$ is $\R$-{\it indecomposable} if, for any decomposition $V=V_1\oplus V_2$ into $\R$-linear subspaces with both $\mathcal{A}_i:=\mathcal{A}\cap V_i\neq \emptyset$, we have a strict inclusion
$\mathcal{A} \supsetneq \mathcal{A}_1 + \mathcal{A}_2.$
\end{definition}

\begin{example}\label{ex-hyp}
Let $n\ge 3$. Let $q$ be the standard quadratic form of signature $(1,n-1)$ on $\R^{n}$, that is
$$q:(x_1,\ldots,x_n)\in\R^n\mapsto x_1^2 - \sum_{i=2}^n x_i^2,$$
and set $h= (1,0,\ldots, 0)\in\R^n$. The non-degenerate open cone
$$\mathcal{H}_n:=\{v\in \R^{n}\mid q(v) > 0,\, q(h,v) > 0\}$$
is an $\R$-indecomposable self-dual homogeneous cone.

While the self-duality of $\mathcal{H}_n$ with respect to $q$ is well-known, Definition \ref{def-selfdual} requires us to prove that $\mathcal{H}_n$ is also self-dual with respect to a positive definite quadratic form. In fact, we show that the standard $L^2$-norm
$$\mathrm{tr}:(x_1,\ldots,x_n)\in\R^n\mapsto \sum_{i=1}^n x_i^2$$
is suitable. We follow the argument of \cite[Chapter I., Paragraph 5. (C)]{Koecher}. Note that the linear transformation
$\varphi:(x_1,x_2,\ldots,x_n)\mapsto (x_1,-x_2,\ldots,-x_n)$
fixes $h$, is involutive, $q$-orthogonal, and thus preserves the cone $\mathcal{H}_n$. It also satisfies 
$q(x,y)=\mathrm{tr}(x,\varphi(y))$ for all $x,y\in\R^n$. Therefore,
\begin{align*}
(\mathcal{H}_n)^{\vee_{\mathrm{tr}} }
&= \{v\in\R^n\mid \mathrm{tr}(v,w) > 0\; \forall\, w\in \mathcal{H}_n\} \\
&= \{v\in\R^n\mid q(v,x) > 0\; \forall\, x\in \mathcal{H}_n\} \\
&= (\mathcal{H}_n)^{\vee_{q}} \\
&=  \mathcal{H}_n.
\end{align*}

A cone that identifies, up to linear isomorphism, with $\mathcal{H}_n$ for some $n\ge 3$ is said to be {\it of hyperbolic type}, or {\it hyperbolic} for short.
\end{example}

A classification of $\R$-indecomposable self-dual homogeneous cones was achieved by the 1934 result of Jordan--von Neumann--Wigner \cite{JvNW} on the classification of formally real Jordan algebras, and by the 1970ies Koecher--Vinberg theorem stating an equivalence of categories between formally real Jordan algebras and self-dual homogeneous cones (see e.g. \cite{Koecher}, \cite[Theorem III.2.1 and Section III.3]{FK94}, and Theorem \ref{thm-koechervinberg} for a precise statement). Here we present the classification.

\begin{theorem}\label{thm-classification}
Let $\mathcal{A}$ be a self-dual homogeneous cone. Then there is a unique decomposition
$$\mathcal{A}=\bigoplus_{j=1}^r \mathcal{A}_j,$$
where each $\mathcal{A}_j$ is an $\R$-indecomposable self-dual homogeneous cone. Moreover, the $\R$-indecomposable self-dual homogeneous cones are classified as:
\begin{enumerate}
\item the positive halfline in $\R$;
\item the cone of positive definite, real symmetric $n$ by $n$ matrices for some $n\ge 3$;
\item the cone of positive definite, complex Hermitian $n$ by $n$ matrices for some $n\ge 3$;
\item the cone of positive definite, quaternionic Hermitian $n$ by $n$ matrices for some $n\ge 3$;
\item the hyperbolic cone $\mathcal{H}_n$ for some $n\ge 3$;
\item the cone of positive definite octonionic Hermitian $3$ by $3$ matrices.
\end{enumerate}
\end{theorem}

We also mention the following result, usually attributed to Koecher and Vinberg.

\begin{proposition}
Let $\mathcal{A}$ be a self-dual homogeneous cone. Then the Lie group $\Aut(\mathcal{A})$ has finitely many components, and its identity component $\Aut^{\circ}(\mathcal{A})$ is a connected Lie group realised as the real points of an algebraic group.
\end{proposition}

We often work with a preferred lattice, or rational structure on the real vector space $V$. The following definition ensures compatibility of a self-dual homogeneous cone with a given rational structure; We follow \cite[Sections 3, 4]{AMRT}.

\begin{definition}\label{def-compatiblerat}
Let $\mathcal{A}$ be a self-dual homogeneous cone in a finite dimensional real vector space $V$. Let $V_{\Q}$ be a rational structure in $V$. We say that $\mathcal{A}$ is {\it compatible with $V_{\Q}$} if the algebraic group $\Aut^{\circ}(\mathcal{A})$ is defined over $\Q$. For short, we may write that $\mathcal{A}$ be a self-dual homogeneous cone in $V=V_{\Q}\otimes \R$ to indicate compatibility between $\mathcal{A}$ and $V_{\Q}$.
\end{definition}

\begin{remark}\label{rem-rho2}
In $V=\R^2$ endowed with its usual rational structure, any non-degenerate open cone can be written as the convex hull of two positive halflines:
$$\mathcal{A}=\mathrm{Conv}_{\R}(R_1,R_2).$$
Such a cone is always self-dual and homogeneous. However, it is compatible with the rational structure if and only if one of the following two situations occurs:
\begin{itemize}
\item $R_1$ and $R_2$ are both rational halflines,
\item or $R_1$ and $R_2$ are both defined over a quadratic extension $K$ of $\Q$, and exchanged by the Galois conjugation of $K/\Q$.
\end{itemize}
\end{remark}

Consider a self-dual homogeneous cone $\mathcal{A}$ compatible with a rational structure $V_{\Q}$. Remark \ref{rem-rho2} shows that the $\R$-indecomposable summands of $\mathcal{A}$ do not themselves need to be compatible with $V_{\Q}$. Nevertheless, we now show that adequate partial sums of the $\R$-indecomposable summands of $\mathcal{A}$ are compatible with $V_{\Q}$.

\begin{remark}\label{rem-ex34} In a real finite-dimensional vector space $V_{\R}$, consider a self-dual homogeneous cone $\mathcal{A}$ of the form
$$\mathcal{A}=\mathcal{B}^{\oplus n}\oplus \mathcal{D},$$
for an $\R$-indecomposable self-dual homogeneous cone $\mathcal{B}$, an integer $n\ge 1$, and a self-dual homogeneous cone $\mathcal{D}$ that does not admit any summand isomorphic to $\mathcal{B}$.
If $\mathcal{A}$ is compatible with some rational structure $V_{\Q}$ on $V_{\R}$, then, by uniqueness of the decomposition of $\mathcal{A}$, the summands $\mathcal{B}^{\oplus n}$ and $\mathcal{D}$ are both defined over $V_{\Q}$.

Of course, we may further decompose $\mathcal{D}$ and iterate the argument. 
In particular, if $\mathcal{A}$ is compatible with $V_{\Q}$, the sum of its non-hyperbolic summands is also compatible with $V_{\Q}$.
\end{remark}

The compatibility of $\mathcal{A}$ with the rational structure $V_{\Q}$ will be important when we construct \textbf{rational} polyhedral fundamental domains later on, e.g., in Theorem \ref{thm-amrt} and Propositions \ref{prop-characterize-RPFD}, \ref{prop-descentwellclipped}.

The next lemma is written with a lattice $V_{\Z}$; the corresponding rational structure is of course $V_{\Q}:=V_{\Z}\otimes \Q$.

\begin{lemma}\label{lem-quadratic}
Let $\mathcal{A}$ be a self-dual homogeneous cone in $V=V_{\Z}\otimes \R$. There exists a positive definite $V_{\Z}$-integral quadratic form $\mathrm{tr}$ on $V$ that is $\Aut(\mathcal{A},V_{\Z})$-invariant, and with respect to which $\mathcal{A}$ is self-dual. 
\end{lemma}

\begin{proof}
The existence of a positive definite form ${\rm tr}$ that is $V_{\Q}$-rational, on which $\Aut(\mathcal{A})$ acts by scaling, and with respect to which $\mathcal{A}$ is self-dual is due to standard facts of Lie theory, see for instance \cite[3.1, Paragraph 2]{AMRT} for an explanation. A large enough multiple of it is $V_{\Z}$-integral and $\Aut(\mathcal{A},V_{\Z})$-invariant.
\end{proof}

\subsection{Round and simplicial parts}

We introduce a few {\it ad hoc} definitions.

\begin{definition}\label{def-rdpart}
Let $\mathcal{A}$ be a self-dual homogeneous cone. 
We define the {\it round part} of $\mathcal{A}$, denoted by $\rd \mathcal{A}$, to be the cone obtained by removing all halfline summands of the decomposition of $\mathcal{A}$ given by Theorem \ref{thm-classification}. 
We define the {\it simplicial part} of $\mathcal{A}$ as the sum of all halfline summands of $\mathcal{A}$ in that same decomposition.
\end{definition}

\begin{remark}
The round part and the simplicial part of a self-dual homogeneous cone both are self-dual homogeneous cones in their own linear span.
\end{remark}

The next definition defines a preorder on the set of simplicial cones in a vector space with a preferred lattice. Crucially, we do not assume our simplicial cones to be full-dimensional in this definition.

\begin{definition}\label{def-simplicialpart-preorder}
Let $\Sigma$ and $\Xi$ be two simplicial cones in $V=V_{\Z}\otimes \R$. We say that $\Sigma$ {\it rules} $\Xi$ if the following two assumptions hold
\begin{enumerate}
\item the linear spans of $\Sigma$ and $\Xi$ in $V$ coincide;
\item there is a finite index subgroup in $\Aut(\Xi,V_{\Z})$ that preserves the cone $\Sigma$.
\end{enumerate}
\end{definition}

\begin{remark}\label{rem-rules} Let us state a few properties of this relation.
\begin{itemize}
\item Any simplicial cone clearly rules itself.
\item The relation that a simplicial cone rules another simplicial cone is transitive.
\item A rational simplicial cone is ruled by any simplicial cone with the same linear span. Indeed, if $\Xi$ is a rational simplicial cone with $k$ extremal rays, the group
$\Aut(\Xi,V_{\Z})$ acts faithfully by permutation on the primitive elements spanning the $k$ extremal rays, thus its order divides the order of the $k$-th symmetric group: This shows that $\Aut(\Xi,V_{\Z})$ is finite, and thus that the trivial group has finite index in $\Aut(\Xi,V_{\Z})$.
\item The relation that a full-dimensional simplicial cone rules another full-dimensional simplicial cone is stable under direct sum; see Lemma \ref{lem-dirsumrules}.
\item Let $\Sigma$ be a simplicial cone in $\R^2=\Z^2\otimes \R$ spanned by two irrational rays. Let $\Xi$ be a cone spanned by one of the extremal rays $R$ of $\Sigma$ and by a rational ray $R_e$.
Then $\Sigma$ rules $\Xi$. Indeed, 
consider an element $\tau\in \Aut(\Xi,\Z^2)$. Note that $\tau$ cannot exchange $R_e$ with $R$, thus $\tau$ preserves $R_e$ and acts trivially on it, and preserves $R$ and acts by scaling by a scalar $\lambda > 0$ on it. But $\tau$ preserves the lattice $\Z^2$, hence has determinant $\pm 1$. This implies $\lambda= 1$, hence $\tau$ is the identity. So the group $\Aut(\Xi,\Z^2)$ is trivial, and $\Sigma$ indeed rules $\Xi$.
\end{itemize}
\end{remark}

\begin{lemma}\label{lem-dirsumrules}
Let $V$ be a finite-dimensional vector space with a rational decomposition $V_{\Q}=V_{1,\Q}\oplus V_{2,\Q}$. For $i=1,2$, let $\Sigma_i$ and $\Xi_i$ be two full-dimensional simplicial cones in $V_{i}$ and assume that $\Sigma_i$ rules $\Xi_i$. Then, the direct sum $\Sigma_1\oplus \Sigma_2$ rules the direct sum $\Xi_1\oplus \Xi_2$ in $V$.
\end{lemma}

\begin{proof}
Item (1) of Definition \ref{def-simplicialpart-preorder} clearly holds. Item (2) for the direct sums will follow from Item (2) for the individual factors, once we prove that the inclusion
\begin{equation}\label{noname}
\Aut(\Xi_1,V_{1,\Z})\times \Aut(\Xi_2,V_{2,\Z})  < \Aut(\Xi_1\oplus \Xi_2,V_{1,\Z}\oplus V_{2,\Z})
\end{equation}
is of finite index. 

Consider the set $\mathcal{R}$ of extremal rays of the cone $\Xi_1\oplus \Xi_2$, on which the group $\Aut(\Xi_1\oplus \Xi_2,V_{\Z})$ acts by permutation. Since $\mathcal{R}$ is finite, the kernel of this action, that is the subgroup
$$\Gamma:= \{\phi\in \Aut(\Xi_1\oplus \Xi_2,V_{\Z})\mid \phi(R)=R\;\;\forall R\in\mathcal{R}\}$$
is of finite index in $\Aut(\Xi_1\oplus \Xi_2,V_{\Z})$. We also note that $\Gamma$ preserves any cone spanned by a subset of the extremal rays of $\Xi_1\oplus \Xi_2$, and in particular preserves both cones $\Xi_i$ for $i=1,2$. This realizes $\Gamma$ as a subgroup of 
\begin{align*}
\Aut(\Xi_1,\Xi_2,V_{\Z}) &= \Aut(\Xi_1,V_{\Z}\cap\mathrm{Span}_{\R}\Xi_1)\times \Aut(\Xi_2,V_{\Z}\cap\mathrm{Span}_{\R}\Xi_2) \\
&= \Aut(\Xi_1,V_{1,\Z})\times \Aut(\Xi_2,V_{2,\Z})
\end{align*} and shows that the inclusion \eqref{noname} is indeed of finite index.
\end{proof}

These notions will later appear in Proposition \ref{prop-descentwellclipped} and Lemma \ref{lem-rdpartfunddom}.

\subsection{Formally real Jordan algebras}

As it will be used later, we also state the Koecher--Vinberg theorem, a key ingredient in the classification stated in Theorem \ref{thm-classification}, which provides an equivalence of categories between self-dual homogeneous cones and a certain class of algebraic objects. We loosely follow \cite{Koecher}.

\begin{definition}
A {\it Jordan algebra} is a finite dimensional $\R$-vector space $V$, an element $e\in V$, and a bilinear operation $\circ: V\times V\to V$ such that
\begin{itemize}
\item $\circ$ is commutative;
\item $e$ is a neutral element for $\circ$;
\item for all $x,y\in V$, the Jordan identity $(x\circ x)\circ (x\circ y) = x\circ ((x\circ x) \circ y)$ holds.
\end{itemize}

\noindent A Jordan algebra $(V,e,\circ)$ is called {\it formally real} if for all $x_1,\ldots,x_k\in V$,
we have $$x_1\circ x_1 + \ldots + x_k\circ x_k = 0 \Longrightarrow x_1 = \ldots = x_k = 0.$$

\noindent A Jordan algebra $(V,e,\circ)$ is called {\it simple} if any linear subspace of $V$ that is preserved by multiplication by $V$ is equal to $\{0\}$ or $V$.
\end{definition}

\begin{remark}
The notion of a formally real Jordan algebra is closed under isomorphisms, subalgebras, and direct products.
\end{remark}

The next result is due to Koecher and Vinberg, see for instance \cite[Theorem 15]{Koecher}.

\begin{theorem}\label{thm-koechervinberg}
There is an isomorphism of categories between formally real Jordan algebras and pointed self-dual homogeneous cones, given on objects by
$$F:(V,e,\circ) \mapsto (\mathcal{A},e) \quad \mbox{where }\mathcal{A}:=\{x\circ x\mid x\in V\}^{\circ},$$
and on morphisms by
$F: \phi \mapsto \phi.$ 
\end{theorem}

\begin{remark}\label{rem-koechervin}
We mention some aspects this correspondence, as they will be used later on in Lemma \ref{lem-transitivecentralizer}.

Let $(V,e,\circ)$ be a formally real Jordan algebra.
We consider the following closed subset of $V$:
$$\overline{\mathcal{A}} := \{x\circ x\mid x\in V\}.$$
This subset contains $e$ and is clearly stable under multiplication by non-negative (real) scalars. The fact that $V$ is formally real shows that $\overline{\mathcal{A}}$ contains no line. 
One can show that $\overline{\mathcal{A}}$ is a full-dimensional cone in $V$. The crux of the argument is the proof that $\overline{\mathcal{A}}$ is closed under addition, which follows from a spectral theorem for Jordan algebras (\cite[Theorem III.1.1]{FK94}). Let us sketch the argument.
For every $v\in V$, we denote the corresponding \textit{left-multiplication operator} by $L(v):V\to V$. The fact that $V$ is formally real is equivalent to the fact that the symmetric bilinear form
$$\mathrm{tr}(x,y) := \mathrm{tr}(L(x\circ y))$$
is positive definite (\cite[Theorem 12]{Koecher}) on $V$. With respect to this form, every left-multiplication operator $L(v)$ is self-adjoint. It is proven in \cite[Theorem III.2.1]{FK94} that 
$$\overline{\mathcal{A}}=\{v\in V\mid L(v)\mbox{ positive semi-definite}\}.$$
The fact that the set of positive semi-definite symmetric real matrices is closed under addition thus implies that $\overline{\mathcal{A}}$ is closed under addition.

The interior $\mathcal{A}$ of $\overline{\mathcal{A}}$ has a nice description:
$$\mathcal{A} = \{v\in V\mid L(v)\mbox{ positive definite}\}  = \{ b\circ b\mid b\in V,\, b\mbox{ invertible}\},$$
where an element $b\in V$ is \textit{invertible} if there exists $y\in\R[b]$ such that $b\circ y = e$ (see \cite[Page 30 before Proposition II.2.3 and Page 48 last line of Proof of Theorem III.2.1]{FK94}). This cone is self-dual with respect to the positive definite form $\mathrm{tr}$. It is also homogeneous: Indeed, for any $a\in\mathcal{A}$, letting $b$ be an invertible element such that $a = b\circ b$, the \textit{quadratic operator} 
$$Q(b):=2L(b)^2-L(b\circ b)$$
is an invertible linear endomorphism of $V$ (\cite[Proposition II.3.1]{FK94}), preserves the cone $\mathcal{A}$ (\cite[Proposition III.2.2]{FK94}), and sends $e$ to $a$.
\end{remark}

\begin{remark}
If we decide to not keep track of the neutral element of Jordan algebras, and want to define a functor $\overline{F}$ from the category of formally real Jordan algebras into the category of self-dual homogeneous cones, it remains an equivalence of categories, but fails to be injective on objects.
\end{remark}

\subsection{Useful facts from birational geometry}

This subsection gathers a few facts from birational geometry, which we will mostly use in the proof of Theorem \ref{thm-ccforexplicitvars}. Our first lemma is very close to a result of Xu for the cone $\Mov^{\rm e}$ \cite[Theorem 24, Remark on Page 26]{Fulin} and Gachet-Lin-Stenger-Wang for the cone $\Nef^{\rm e}$ \cite[Proposition 1.6]{GLSW24}. We state it and prove it here for the cones $\Mov^+$ and $\Nef^+$, under somewhat weaker assumptions than in \cite[Theorem 24]{Fulin}.

\begin{lemma}\label{lem-descent}
Let $f:(X,\Delta)\to (Y,\Delta_Y)$ be a crepant birational morphism of klt Calabi--Yau pairs. Assume that the pair $(X,\Delta)$ satisfies the movable cone conjecture stated as in Conjecture \ref{conj-cc}. Then the pair $(Y,\Delta_Y)$ satisfies the movable cone conjecture.
\end{lemma}

\begin{proof}
Note that $F:=f^*\Mov^+(Y)$ is a face of the cone $\Mov^+(X)$. By \cite[Proposition 3.6]{GLSW24}, since $(X,\Delta)$ satisfies the movable cone conjecture, the subgroup $\mathrm{Stab}(F)$ of $\PsAut^*(X,\Delta)$ that stabilizes the face $F$ acts on $F$ with a rational polyhedral fundamental domain.
To conclude, it suffices to check that $\mathrm{Stab}(F)$ is a subgroup of $f^{-1}\circ \PsAut^*(Y,\Delta_Y)\circ f$. 

For $g\in \mathrm{Stab}(F)$, note that there exists $\alpha:X\dashrightarrow X'$ a small $\Q$-factorial modification such that the relative interiors of $g^*f^*\Nef(Y)$ and $f^*\alpha^*\Nef(Y)$ intersect by \cite[Theorem 1.2]{BCHM} (see also \cite[Proposition 4.10]{GLSW24}). By \cite[Lemma 4.2]{GLSW24}, this means that $f\circ g = \alpha\circ f$, thus 
$g\in f^{-1}\circ \PsAut^*(Y,\Delta_Y)\circ f$ as wished.
\end{proof}

\begin{lemma}\label{lem-descentbis}
Let $f:(X,\Delta)\to (Y,\Delta_Y)$ be a crepant birational morphism of klt Calabi--Yau pairs. Assume that finitely many pairs arise as small $\Q$-factorial modifications of $(X,\Delta)$, and that each of them satisfies the nef cone conjecture stated as in Conjecture \ref{conj-cc}. Then finitely many pairs arise as small $\Q$-factorial modifications of $(Y,\Delta_Y)$, and each of them satisfies the nef cone conjecture.
\end{lemma}

\begin{proof}
Apply \cite[Theorem 1.5 (i), (3) $\Rightarrow$ (4)]{GLSW24} to the pair $(X,\Delta)$, and restrict to the Mori chambers in $N^1(X)$ associated to small $\Q$-factorial modifications of $Y$ (pre-composed with $f$).
\end{proof}

We also prove a lemma about movable cones of products.

\begin{lemma}\label{lem-mov-product}
Let $X=X_1\times X_2$ be a product of normal projective varieties. Assume that $N^1(X) = p_1^*N^1(X_1) + p_2^*N^1(X_2)$, where $p_1,p_2$ denote the natural projections. Then every movable Cartier divisor on $X$ is the sum of the pullbacks of movable Cartier divisors on $X_1$ and $X_2$. Therefore, we have
$$\Mov(X) = p_1^*\Mov(X_1) + p_2^*\Mov(X_2).$$
\end{lemma}

\begin{proof}
Let $D$ be a movable Cartier divisor on $X$. We write $D=p_1^*D_1+p_2^*D_2$.
Since $D$ is movable, we can pick two sections $s,s'\in|D|$ that have no common divisorial component. For a component $Z$ of the intersection $s\cap s'$, the restriction of the projection $p_2:Z\to  X_2$ has general fiber that is empty, or of dimension at most $\dim X_1 -2$.
In particular, for a general point $u\in X_2$, the locus $s\cap s'$ intersects the fiber $X_1\times \{u\}$ along a closed subscheme of codimension at least $2$. Thus, $s|_{X_1\times\{u\}}$ and $s'|_{X_1\times\{u\}}$ are well-defined sections of the linear system $|D_1|$ on $X_1$, and have no common divisorial component. This shows that $D_1$ is movable, and the same argument works symmetrically for $D_2$. 

Since the movable cone is the closed convex cone spanned by classes of movable divisors, this concludes the proof.
\end{proof}

The next lemma is inspired by \cite[Lemma 4.6]{Druel}, and is therein also credited to C. Casagrande.

\begin{lemma}\label{lem-psautpreserves}
Let $X= X_1\times X_2$ be a product of projective varieties with klt $\Q$-factorial singularities, and with $h^1(X_2)= 0$.
Consider a small $\Q$-factorial modification $\alpha: X\dashrightarrow Y$. Then there exist $\alpha_i:X_i\dashrightarrow Y_i$ small $\Q$-factorial modifications for $i=1,2$ such that $Y=Y_1\times Y_2$ and $\alpha=(\alpha_1,\alpha_2)$.
\end{lemma}

\begin{proof}
By \cite[Exercise III.12.6]{HarBook}, since $h^1(X_2)=0$, we can decompose $N^1(X)$ as $p_1^*N^1(X_1)\oplus p_2^*N^1(X_2)$. Thus and by Lemma \ref{lem-mov-product}, we have a decomposition of the movable cone:
$$\Mov(X)=p_1^*\Mov(X_1)\oplus p_2^*\Mov(X_2).$$
Let $H$ be a very ample Cartier divisor on $Y$, then 
$$\alpha^* H = p_1^*M_1+p_2^*M_2,$$ 
where $M_i$ is a movable, big Cartier divisor on $X_i$. 

The linear system $|M_i|$ defines a birational contraction $\alpha_i:X_i\dashrightarrow Y_i$, with an ample divisor $H_i$ on $Y_i$ such that $M_i=\alpha_i^*H_i$. In the Néron-Severi space of $X$, we have a non-empty intersection
$$\alpha^*\mathrm{Amp}(Y)\cap (\alpha_1,\alpha_2)^*\mathrm{Amp}(Y_1\times Y_2) \neq\emptyset,$$
so by \cite[Lemma 4.2]{GLSW24} (see also \cite[Lemma 1.5]{Kaw97} in dimension 3), the birational map
$$(\alpha_1,\alpha_2)\circ \alpha^{-1} : Y\to Y_1\times Y_2$$
is indeed a biregular isomorphism.
\end{proof}

\begin{corollary}\label{cor-psautpreserves}
Let $X=\prod_{i=1}^r X_i$ be a product of projective varieties with klt $\Q$-factorial singularities, with $h^1(X_i)= 0$ for $i\ge 2$.
Then $\PsAut(X)$ preserves the decomposition of $X$, up to permuting factors that are isomorphic in codimension $1$.
\end{corollary}

\begin{proof}
Note that applying \cite[Lemma 4.6]{Druel} to the identity from $X$ to itself, any other decomposition of $X$ as a product differs by a mere permutation.
Thus, by \cite[Exercise III.12.6]{HarBook} and by Lemma \ref{lem-psautpreserves}, for a pseudoautomorphism $\alpha$ of $X$, there is a permutation $\sigma$ of the integers between $1$ and $r$ and, for every index $i$, a small $\Q$-factorial modification $\alpha_i:X_i\dashrightarrow X_{\sigma(i)}$, such that
$$\alpha = \prod_{i=1}^r \alpha_i,$$
which indeed, preserves the decomposition of $X$ up to permutation of factors that are isomorphic in codimension $1$.
\end{proof}

We conclude with a lemma inspired by \cite{MY15}, \cite[Section 7, Page 31]{LMP24}, and \cite[Proof of Theorem 6.1 (iii) (3a), Page 26-27]{GLSW24}. Here, one may think of the $(V_w)_{w\in W}$ as wall hyperplanes corresponding to flops.

\begin{lemma}\label{lem-finitesqm}
Let $(X,\Delta)$ be a klt Calabi--Yau pair that satisfies the movable cone conjecture. Suppose that there exist hyperplanes $(V_w)_{w\in W}$ in $N^1(X)$ such that
$$\mathrm{Mov}^{\circ}(X) \setminus \bigcup_{w\in W} V_w = \bigsqcup_{\substack{\alpha: X\dashrightarrow X'\\ \mathrm{SQM}}}\alpha^*\mathrm{Amp}(X')$$
and such that for any rational polyhedral cone $\Pi\subset \Mov^+(X)$, only finitely many of the $V_w$ are intersecting $\Pi^{\circ}$.
Then there are finitely many isomorphism classes of pairs $(Y,\Delta_Y)$ obtained by small $\Q$-factorial modifications of $(X,\Delta)$, and the nef cone conjecture holds for each of them.
\end{lemma}

\begin{proof}
Let $\Pi$ be a rational polyhedral fundamental domain for $\PsAut^*(X,\Delta)$ acting on the cone $\Mov^+(X)$. By assumption, we can take finitely many small $\Q$-factorial modifications
$\alpha_i:X\dashrightarrow X_i$ with $1\le i\le r$ to cover
$$\Pi
\quad\subset\quad
\bigcup_{i=1}^r \alpha_i^*\Nef(X_i).$$
Since $\Pi$ is a fundamental domain for $\Mov^+(X)$, this already shows that any small $\Q$-factorial modification of $X$ identifies up to $\PsAut(X,\Delta)$ with one of the $\alpha_i$. Thus, a pair $(Y,\Delta_Y)$ obtained by a small $\Q$-factorial modification from $(X,\Delta)$ must be one of the $(X_i,{\alpha_i}_*\Delta)$. 

As a corollary, for a fixed index $i$, the chambers of the form $\beta^*\alpha_i^*\mathrm{Amp}(X_i)$ with $\beta\in\PsAut(X,\Delta)$ that intersect $\Pi^{\circ}$ are finitely many. We denote them by $\beta_j^*\alpha_i^*\mathrm{Amp}(X_i)$ with $1\le j\le s_i$.

Fix $i$. We now prove the nef cone conjecture for $(X_i,{\alpha_i}_*\Delta)$. 
Denote by $\Sigma_i$ the sum of the cones of the form ${\beta_j^{-1}}^*\Pi\cap \Nef(X_i)$ for $1\le j\le s_i$. Note that $\Sigma_i$ is itself rational polyhedral: Indeed, by assumption on the wall hyperplanes $(V_w)$, any intersection of the form $\alpha_k^*\Nef(X_k) \cap\Pi$ for $1\le k\le r$ is a rational polyhedral cone.
If the translates of $\Sigma_i$ by $\Aut(X_i,{\alpha_i}_*\Delta)$ cover the cone $\mathrm{Amp}(X_i)$, then \cite[Proposition 4.1]{Loo14} concludes. Let $D\in\mathrm{Amp}(X_i)$. There exists $\beta\in\PsAut(X,\Delta)$ such that $\beta^*\alpha_i^*D\in \Pi$, in particular it belongs to a chamber $\beta_j^*\alpha_i^*\mathrm{Amp}(X_i)$ for some $j$. By \cite[Lemma 4.2]{GLSW24}, this forces 
$$h:=\alpha_i\beta\beta_j^{-1}\alpha_i^{-1}$$ 
to be a biregular automorphism of $X_i$, and in fact $h\in \Aut(X_i,{\alpha_i}_*\Delta)$. It is now immediate that $h^*D$ belongs to $\Sigma_i$.
\end{proof}

\section{Well-clipped, neatly clipped, perfectly clipped cones}

\subsection{Definition and examples of well-clipped cones}

This definition is the focal point of the whole section.

\begin{notation}
When $V$ is a finite dimensional real vector space, $H$ a hyperplane in $V$, we denote by $H_-$ and $H_+$ the two connected components of the complement $V\setminus H$. We call them the \textit{negative}, respectively \textit{positive side} of $H$.
\end{notation}

\begin{definition}\label{def-nicelycutout}
Let $\mathcal{A}$ be a self-dual homogeneous cone in $V=V_{\Z}\otimes \R$. A full-dimensional convex cone $\mathcal{C}$ is {\it well-clipped} in $\mathcal{A}$ if there exists an at most countable set of hyperplanes $({H_i})_{i\in I}$ of $V$ such that
$$\overset{\circ}{\mathcal{C}}=\mathcal{A}\cap \bigcap_{i\in I} H_{i,+},$$
and the following three assumptions hold
\begin{enumerate}
\item[(i)] Decomposing $\mathcal{A} = \bigoplus_{1\le j\le r} \mathcal{A}_j$ into its finitely many $\R$-indecomposable summands, every hyperplane $H_i$ is of the form
$$H_i = 
H_i\cap {\rm Span}_{\R}\, \mathcal{A}_{j(i)}
\oplus
\bigoplus_{\substack{1\le k\le r \\k\neq j(i)}}{\rm Span}_{\R}\, \mathcal{A}_k,$$
where the remaining cone $\mathcal{A}_{j(i)}$ is of hyperbolic type and defined over $V_{\Q}$.
\end{enumerate}
Fixing an $\mathrm{Aut}(\mathcal{A},V_{\Z})$-invariant and $V_{\Z}$-integral quadratic form $q$ that is a direct sum of hyperbolic forms on the linear spans of the $(\mathcal{A}_{j(i)})_{i\in I}$ and of a positive definite quadratic form as in Lemma \ref{lem-quadratic} on the other summands' spans, 
\begin{enumerate}
\item[(ii)] For every $i\in I$, the $q$-orthogonal reflection
$\sigma_i$ fixing the hyperplane $H_i$ preserves the lattice $V_{\Z}$.
\item[(iii)] For every $i,k\in I$ such that $H_i\neq H_k$, any elements $e_i,e_k$ perpendicular to $H_i,H_k$ and on their negative sides satisfy $q(e_i,e_k) \ge 0$.
\end{enumerate}

We call a cone {\it well-clipped} if there exists a self-dual homogeneous cone $\mathcal{A}$ in $V=V_{\Z}\otimes \R$ in which it is well-clipped.
\end{definition}

\begin{remark}\label{rem-minimal}
Unless otherwise stated, when we describe a well-clipped cone $\mathcal{C}$ in a self-dual homogeneous cone $\mathcal{A}$, we pick a list of hyperplanes $(H_i)_{i\in I}$ that has no redundancies (that is, $i\neq i'\Rightarrow H_i\neq H_{i'}$), and such that each halfspace $H_{i,-}$ intersects $\mathcal{A}$ (that is, the halfspace $H_{i,+}$ cuts out a proper subset in $\mathcal{A}$).
\end{remark}

\begin{remark}\label{rem-root} We make a few remarks on the role of the rational structure $V_{\Q}$ in Definition \ref{def-nicelycutout}.
\begin{itemize}
\item To state Assumptions (ii) and (iii), we use an integral quadratic form $q$ on $V_{\Z}$ that restricts to a hyperbolic quadratic form on the span $V_{j(i)}$ of each summand $\mathcal{A}_{j(i)}$. Since the cone $\mathcal{A}_{j(i)}$ is defined over $V_{\Q}$ by Assumption (i), the intersection $V_{\Z}\cap V_{j(i)}$ is a lattice in $V_{j(i)}$: This allows to view the restriction $q|_{V_{j(i)}}$ as a hyperbolic and integral quadratic form on $V_{j(i)}$.

\item For any $i\in I$, note that, by Assumption (ii),
$$H_i^{\perp}=\ker(\sigma_i +\id).$$
The linear map $\sigma_i +\id$ preserves the lattice $V_{\Z}$, hence restricts to a linear endomorphism of $V_{\Q}$. Thus, its kernel is defined over $V_{\Q}$. Since $\sigma_i$ already fixes a hyperplane, this shows that $H_i^{\perp}$ is a line defined over $V_{\Q}$, and so we have an isomorphism
$${H_i}^{\perp}\cap V_{\Z} \simeq \Z.$$
\end{itemize}
\end{remark}

Remark \ref{rem-root} permits the following notation.

\begin{notation}\label{note} In the set-up of Definition \ref{def-nicelycutout}, for any $i\in I$, we denote by $e_i$ the generator of the monoid 
$${H_i}^{\perp}\cap H_{i,-}\cap V_{\Z}\simeq \N.$$
\end{notation}

\begin{lemma}\label{lem-roots}
In the set-up of Definition \ref{def-nicelycutout}, Remark \ref{rem-minimal}, and Notation \ref{note}, the following properties hold:
\begin{enumerate}
\item For all $a\in \mathcal{A}$, we have
$a\in\mathcal{C}^{\circ}$ if and only if $q(e_i,a)> 0$ for all $i\in I$.
\item For all $i\in I$, $e_i$ is in the linear span of the hyperbolic summand $\mathcal{A}_{j(i)}$ of $\mathcal{A}$, that is defined over $V_{\Q}$.
\item For all $i\in I$, $q(e_i,e_i)$ is negative. 
\item The $q$-orthogonal reflection given by
$$\sigma_i:v\in V\mapsto v -\frac{2q(e_i,v)}{q(e_i,e_i)} e_i$$
preserves the lattice $V_{\Z}$.
\item We have $q(e_i,e_k)\ge 0$ for all $i\neq k\in I$.
\end{enumerate}
Moreover, the data of hyperplanes $(H_i)_{i\in I}$ for Definition \ref{def-nicelycutout} is equivalent to the data of elements $(e_i)_{i\in I}$ in $V_{\Z}$ satisfying Properties (1)--(5).
\end{lemma}

\begin{proof}
Item (1) is immediate.
We prove (2). Fix $i\in I$. Denoting by $V_{j(i)}$ the linear span of $\mathcal{A}_{j(i)}$ and by $q_{j(i)}$ the restriction $q|_{V_{j(i)}}$, we have
$$e_i\in H_i^{{\perp}_q} = (H_{i}\cap V_{j(i)})^{{\perp}_{q_{j(i)}}} \subset V_{j(i)}.$$

We now prove (3). Fix $i\in I$. By (2), we have $q(e_i,e_i)=q_{j(i)}(e_i,e_i)$. By Remark \ref{rem-minimal}, the hyperplane $H_{i}\cap V_{j(i)}$ intersects the hyperbolic cone $\mathcal{A}_{j(i)}$ non-trivially. Therefore, the restriction of $q_{j(i)}$ to $H_{i}\cap V_{j(i)}$ has signature $(1,\dim V_{j(i)} - 2)$, and the fact that $e_i$ is orthogonal to that hyperplane yields
$q_{j(i)}(e_i,e_i) < 0,$ as wished.

Note that (4) and (5) are equivalent to Assumption (ii) and (iii) respectively.

For the ``Moreover'' part, let $(e_i)_{i\in I}$ be elements of $V_{\Z}$ satisfying Properties (1)-(5). We define 
$$H_i=(e_i)^{{\perp}_q}\mbox{ and }H_{i,+}=\{x\in V\mid q(x,e_i) > 0\}.$$
By Property (1), we have
$$\mathcal{C}^{\circ}=\mathcal{A}\cap\bigcap_{i\in I} H_{i,+}.$$
By Property (2), we have $e_i\in V_{j(i)}$ for some hyperbolic summand $\mathcal{A}_{j(i)}$ of $\mathcal{A}$. Since $q$ splits as a hyperbolic form on $V_{j(i)}$ and a residual quadratic form on the span $W$ of all other summands of $\mathcal{A}$, the space $H_i$ is the direct sum of the $q_{j(i)}$-orthogonal of $e_i$ in $V_{j(i)}$ and of $W$. Property (3) then concludes that $H_i$ is a hyperplane; it is clearly defined over $V_{\Q}$. This proves Assumption (i). As mentioned previously, Properties (4) and (5) ensure Assumptions (ii) and (iii). 
\end{proof}

\begin{remark}\label{rem-hypAss}
Let $\mathcal{A}$ be a self-dual homogeneous cone of hyperbolic type in $V=V_{\Z}\otimes \R$ and let $q$ denote the corresponding hyperbolic form. Consider a collection $(H_i)_{i\in I}$ of hyperplanes satisfying Remark \ref{rem-minimal}. It clearly satisfies Assumption (i) of Definition \ref{def-nicelycutout}. It satisfies Assumption (ii) if and only if, for all $i\in I$, the corresponding element $e_i\in V_{\Z}$ satisfies:
\begin{center}
$q(e_i,e_i)$ divides every element of $2 q(e_i,V_{\Z}) \subset \Z$.
\end{center}
In general, Assumption (ii) could be rephrased, in a lattice-theoretic context, by saying that the $(e_i)_{i\in I}$ are {\it roots} of the lattice $V_{\Z}$, see \cite[Section 0.1]{Borcherds}. Yet we prefer to avoid this terminology, as it is often expected of roots to have norm $-2$, whilst our $(e_i)_{i\in I}$ can have any negative square.
\end{remark}

We now prove that the data of the cones $\mathcal{C}$ and $\mathcal{A}$ and of the quadratic form $q$ determines the set of hyperplanes $(H_i)_{i\in I}$ used to cut out $\mathcal{C}$ (after removing superfluous ones in the sense of Remark \ref{rem-minimal}). 

\begin{proposition}\label{prop-roots}
Let $\mathcal{A}$ be a self-dual homogeneous cone in $V=V_{\Z}\otimes \R$. Let $\mathcal{C}$ be a well-clipped cone in $\mathcal{A}$ and let $q$ be a quadratic form on $V$. Up to permutation of the indices, there is at most one collection of hyperplanes $(H_i)_{i\in I}$ in $V$ that satisfies Remark \ref{rem-minimal} and cuts out $\mathcal{C}$ with respect to the quadratic form $q$ in the sense of Definition \ref{def-nicelycutout}.
\end{proposition}

\begin{proof}
Assume that a certain collection of hyperplanes $(H_i)_{i\in I}$ satisfies Remark \ref{rem-minimal} and cuts out $\mathcal{C}$ with respect to our fixed quadratic form $q$, in the sense of Definition \ref{def-nicelycutout}. To each $H_i$, we associate the element $e_i\in V$ given by Notation \ref{note}.

We will now characterize $\{e_i \mid i\in I\}$ as the set of primitive generators of extremal rays of the dual cone $\mathcal{C}^{\vee_q}$ that are not contained in the cone $\overline{\mathcal{A}}$. Since $(e_i)_{i\in I}$ determines $(H_i)_{i\in I}$ by Lemma \ref{lem-roots}, this characterization is enough to conclude.

Since $\mathcal{A}$ is self-dual with respect to the quadratic form $q$, the dual cone of $\mathcal{C}$ is
$$\mathcal{C}^{\vee_q} = \overline{\mathrm{Cone}\,(\mathcal{A},(e_i)_{i\in I})}.$$
Since $\mathcal{C}$ is full-dimensional, $\mathcal{C}^{\vee_q}$ is non-degenerate. Moreover, we claim that
\begin{equation}\label{eq-closed}
\mathcal{C}^{\vee_q} = \mathrm{Cone}\,(\overline{\mathcal{A}},(e_i)_{i\in I}).
\end{equation}

We prove this claim. Clearly, the inclusion
$$\mathrm{Cone}\,(\overline{\mathcal{A}},(e_i)_{i\in I}) \subset  \overline{\mathrm{Cone}\,(\mathcal{A},(e_i)_{i\in I})}$$
holds. We are left to prove the other inclusion, i.e., that the set $\mathrm{Cone}\,(\overline{\mathcal{A}},(e_i)_{i\in I})$ is closed. By sorting the $(e_i)_{i\in I}$ according to the index $j(i)$, we obtain a finite decomposition
$$\mathrm{Cone}\,(\overline{\mathcal{A}},(e_i)_{i\in I})  = \bigoplus_{k\in j(I)} \mathrm{Cone}\,(\overline{\mathcal{A}_k},(e_i)_{i\in j^{-1}(k)}) \oplus \bigoplus_{k\in J\setminus j(I)} \overline{\mathcal{A}_k}.$$
Since a finite sum of closed cones is closed, we have
reduced to the case when $\mathcal{A}$ is an irreducible hyperbolic cone. In that case, $q$ is the hyperbolic quadratic form defining $\mathcal{A}$; see Example \ref{ex-hyp}. 

If $I$ is finite, the set $\mathrm{Cone}\,(\overline{\mathcal{A}},(e_i)_{i\in I})$ is clearly closed, and we are done. Otherwise, we fix an identification $I\simeq \N$. It is enough to consider a converging sequence $(v_n)_{n\in\N}$ in $V$ whose elements write
$$v_n = \sum_{i=1}^n \lambda_i^n e_i$$
with coefficients $\lambda_i^n\ge 0$ for all $1\le i\le n$ and all $n\in\N$, and prove that its limit $v$ belongs to $\mathrm{Cone}\,(\overline{\mathcal{A}},(e_i)_{i\in I})$.
Pick an element $c\in\mathcal{C}^{\circ}\cap V_{\Z}$. Since $q(e_i,c)\ge 1$ for all $i\in\N$, the existence of the limit
$$q(v,c)=\lim_{n\to\infty} \sum_{i=1}^n \lambda_i^n q(e_i,c)$$
forces the limit $\lim_{n\to\infty} \sum_{i=1}^n \lambda_i^n$ to exist as well. Thus, every sequence of the form $(\lambda_i^n)_{n\in\N}$ is bounded from above, and for all but finitely many values of $i$, it even converges to zero. Extracting an appropriate subsequence of $(v_n)_{n\in\N}$, we can assume that, for every $i\in\N$, the sequence $(\lambda_i^n)_{n\in\N}$ converges to a limit $\lambda_i\in\R_{\ge 0}$, and set a threshold $r$ such that $\lambda_i = 0$ for all $i\ge r+1$. We write
$$v = \sum_{i=1}^r \lambda_i e_i + a,$$ 
for some $a\in V$. It is worth noting that
$$a = \lim_{n\to\infty} \sum_{i=r+1}^n \lambda_i^{n} e_i.$$
If we can prove that $a$ lies in the cone $\overline{\mathcal{A}}$, we are done. Recall that $\mathcal{A}$ is the hyperbolic cone associated to the form $q$. Since, for the element $c\in\mathcal{A}$ picked previously, it holds
$$q(a,c) = \lim_{n\to\infty} \sum_{i= r+1}^n \lambda_i^n q(e_i,c) \ge 0,$$
it suffices to check that $q(a)\ge 0$ to conclude. Since the coefficients $\lambda_i^n$ are non-negative, by continuity and linearity of $q$ in the second argument, it is enough to check that $q(a,e_i)\ge 0$ for all $i\ge r+1$. We fix an index $i\ge r+1$. Since
$$q(a,e_i) = \lim_{n\to\infty} \lambda_i^n q(e_i,e_i) + \left(\sum_{\substack{j=r+1\\ j\neq i}}^n \lambda_j^n q(e_j,e_i)\right),$$
since $\lambda_i^n q(e_i,e_i)$ tends to zero when $n$ tends to $\infty$, 
and since each remaining summand $\lambda_j^n q(e_j,e_i)$ is non-negative, we obtain $q(a,e_i)\ge 0$, as wished. This proves Identity \eqref{eq-closed}.

We can now prove our characterization of $\{e_i \mid i\in I\}$ as the set of primitive generators of extremal rays of the dual cone $\mathcal{C}^{\vee_q}$ that are not contained in the cone $\overline{\mathcal{A}}$. By Identity \eqref{eq-closed}, every extremal ray of $\mathcal{C}^{\vee_q}$ is either contained in $\overline{\mathcal{A}}$, or contains one of the elements $e_i$. Conversely, let us prove that the primitive element $e_i$ spans an extremal ray of the cone $\mathcal{C}^{\vee_q}$. If we write
$$e_i = a + \sum_{k\in K} \lambda_k e_k,$$
with $K\subset I$ finite such that $i\notin K$, $a\in \overline{\mathcal{A}}$, and $\lambda_k > 0$, applying the reflection $\sigma_i$ yields
\begin{equation}\label{eq-ei}
 - e_i = \sigma_i(a) + \sum_{k\in K} \lambda_k e_k - \lambda_k\frac{2q(e_i,e_k)}{q(e_i)}e_i.
 \end{equation}
By Assumption (iii), we have
$$-\lambda_k\frac{2q(e_i,e_k)}{q(e_i)} \ge 0,$$ and since $\sigma_i$ preserves the cone $\mathcal{A}$, the right-handside of Identity \eqref{eq-ei} is in the cone $\mathcal{C}^{\vee_q}$. This contradicts the non-degeneracy of $\mathcal{C}^{\vee_q}$. Thus, the set $(e_i)_{i\in I}$ is indeed determined by the cones $\mathcal{C}$ and $\mathcal{A}$ and the quadratic form $q$.
\end{proof}

The next lemma is a clear consequence of the definition.

\begin{lemma}\label{lem-directsum}
A direct sum of well-clipped cones is itself well-clipped.
\end{lemma}

\begin{proof}
Consider a cone $\mathcal{C}=\mathcal{C}_1\oplus\mathcal{C}_2$, where $\mathcal{C}_u$ is well-clipped in a self-dual homogeneous cone $\mathcal{A}_u$ in $V_u=V_{u,\Q}\otimes \R$ for $u=1,2$. Note that $\mathcal{A}:=\mathcal{A}_1\oplus\mathcal{A}_2$ is a self-dual homogeneous cone in $V:=V_1\oplus V_2$ with respect to the rational structure $V_{\Q}=V_{1,\Q}\oplus V_{2,\Q}$. Clearly, the pullbacks of the hyperplanes cutting out $\mathcal{C}_1$ in $\mathcal{A}_1$ and $\mathcal{C}_2$ in $\mathcal{A}_2$ still cut out $\mathcal{C}$ in $\mathcal{A}$.
Checking the assumptions of Definition \ref{def-nicelycutout} is immediate as long as we take the split quadratic form $q=q_1\oplus q_2$ for Assumptions (ii) and (iii).
\end{proof}

We provide examples of well-clipped cones, with an algebro-geometric reader in mind.

\begin{example}\label{ex-list} This list of examples is by far not exhaustive.
\begin{enumerate}
\item A simplicial cone in $\R^n$ is always well-clipped. Indeed, it decomposes as a direct sum of half-lines over $\R$. In particular, any non-degenerate convex cone in $\R^2$ is well-clipped.
\item The nef cone of a smooth del Pezzo surface $\Sigma$ is rational polyhedral, and well-clipped. For $\rho(\Sigma)=1$ or $2$, it follows from Item (1). Otherwise, we can take the quadratic form to be the intersection form on $N^1(\Sigma)$, and the $(e_i)_{i\in I}$ in Notation \ref{note} and Lemma \ref{lem-roots} to be the finitely many classes of $(-1)$-curves.
\item The nef cone of a smooth projective K3 surface is well-clipped by \cite{Sterk85}.
\item The nef cone of a smooth projective surface underlying a klt Calabi--Yau pair can be described as the intersection of a well-clipped cone, namely the dual cone
$$(\overline{\mathrm{NE}}(S)_{K_S\le 0})^{\vee}$$
with a finite collection of rational halfspaces that is invariant under the action of $\Aut(S)$, see Lemma \ref{lem-surface-nicely} below. The nef cone is not necessarily well-clipped itself, see Example \ref{ex-rationalsurface}.
\item The movable cone of a smooth projective hyperkähler manifold is well-clipped by \cite[Lemma 6.22]{Markman11} and \cite[Theorem 1.1]{Markman13}.
\item The movable cone of a projective primitive symplectic variety with terminal $\Q$-factorial singularities is well-clipped \cite[Theorem 3.10, Lemma 4.6]{LMP24}; \cite[Definition 3.1, Theorem 3.11]{KMPP19}.
\item The movable cone of an abelian variety is self-dual homogeneous, thus well-clipped \cite[Theorem 4.3 and the two paragraphs thereafter]{PS12a}.
\item The nef cone of $\P^2$ blown up at $9$ general points and the movable cone of $\P^3$ blown up at 8 very general points are both well-clipped, by \cite{Nagata60} and \cite{SX23} respectively.
\item The movable cone of a product of normal projective varieties
$$X=\prod_{1\le i\le k} X_i,$$
with $h^1(\mathcal{O}_{X_i}) = 0$ for $2\le i\le k$
is well-clipped if and only if the movable cone of each factor is well-clipped. This follows from \cite[Exercise III.12.6]{HarBook}, Lemmas \ref{lem-mov-product} and \ref{lem-directsum}.
\end{enumerate}
\end{example}

We prove a result backing up Item (4) in Example \ref{ex-list} above.

\begin{lemma}\label{lem-surface-nicely}
Let $S$ be a smooth projective surface underlying a klt Calabi--Yau pair $(S,\Delta)$. The cone $\mathcal{C}_S := (\overline{\mathrm{NE}}(S)_{K_S\le 0})^{\vee}$ is well-clipped, and the cone $\Nef(S)$ can be described as the intersection of $\mathcal{C}_S$ with a finite collection of rational halfspaces that is determined by the isomorphism class of the surface $S$, and in particular invariant under the action of $\Aut(S)$.
\end{lemma}

\begin{proof}
We use the intersection form $q_S$ of $S$ to define the hyperbolic cone $\mathcal{H}(q_S)$ (that contains all ample divisor classes) in the Néron--Severi space $N^1(S)$. By duality, to show that the cone $\mathcal{C}_S$ is well-clipped in $\mathcal{H}(q_S)$, it suffices to prove that any reduced irreducible curve $E$ on $S$ with $-K_S\cdot E \ge 0$ and $E^2 < 0$ induces an integral orthogonal reflection. Showing that $E^2 \in \{-1,-2\}$ would conclude by Remark \ref{rem-hypAss}. Since $S$ is smooth, $E^2$ is an integer. By the genus formula \cite[I.15]{Beauv84}, we have 
$$E^2\ge E^2 + K_S\cdot E = 2h^1(E,\mathcal{O}_E) - 2 \ge -2,$$
as wished.  
This shows that $\mathcal{C}_S$ is a well-clipped cone.

Since $-K_S\equiv \Delta$ is pseudoeffective, we can consider its Zariski decomposition:
$-K_S = P+N,$
where $P$ and $N$ are $\Q$-divisors, with $P$ nef and $N$ effective, and such that, denoting by $(N_i)_{1\le i \le r}$ the irreducible components of $N$, the intersection matrix $(N_i\cdot N_j)_{1\le i,j\le r}$ is negative definite. The Mori cone of $S$ is spanned by its $K_S$-non-positive part and by $\{N_i\}_{1\le i \le r}$. Dualizing, we obtain
$$\Nef(S) = \mathcal{C}_S\cap \bigcap_{i=1}^r \{D\in N^1(X)\mid D\cdot N_i\ge 0\}.$$  By uniqueness of the Zariski decomposition of $-K_S$, the collection of halfspaces derived from the $\{N_i\}_{1\le i\le r}$ is preserved by isomorphisms of surfaces, and in particular invariant under the action of $\Aut(S)$.
\end{proof}

\subsection{Non-examples and questions on well-clipped cones}

In this subsection, we describe various cones that cannot be well-clipped. Although we personally find these examples instructive, they are independent of the main results of this paper, and can be skipped in a first read.

\begin{example}\label{ex-ratpolwellclipped} We construct a rational polyhedral cone that is not well-clipped. 

Consider the circle $x^2+y^2=1$ in $\R^2$, and pick thirteen distinct points $v_1,\ldots ,v_{13}$ on it in that order. We take their convex hull and projectivize, to obtain a polyhedral cone $\mathcal{C}$ with $13$ extremal rays in $\R^3$. 

Arguing by contradiction, assume that $\mathcal{C}$ is well-clipped in some self-dual homogeneous cone $\mathcal{A}$. Since $\mathcal{C}$ is polyhedral in $\R^3$ and not simplicial, $\mathcal{A}$ must be of hyperbolic type. Let $q$ be the hyperbolic quadratic form defining $\mathcal{A}$. Pick an affine hyperplane $H$ in which $q$ restricts to the equation of a circle in $\R^2$. Denote by $\ell_1,\ldots,\ell_{13}$ the lines in $H$ that define the sides of the $13$-gon $\mathcal{C}\cap H$. Note that in $H$, the oriented angle measures between these lines satisfy
$$\cos^2(\ell_i,\ell_j)=\frac{q(e_i,e_j)^2}{q(e_i)q(e_j)},$$
since $\ell_i$ is the polar line to the point $e_i$ with respect to the circle defined by $q$ in $H$.
By Lemma \ref{lem-coxeter} below, this implies that
$$(\ell_i,\ell_j)\in \left\lbrace 0,\frac{\pi}{6},\frac{\pi}{4},\frac{\pi}{3},\frac{\pi}{2},\frac{2\pi}{3},\frac{3\pi}{4},\frac{5\pi}{6}\right\rbrace.$$
Summing the angles of the $13$-gon $\mathcal{C}\cap H$ and writing $\ell_{14}:=\ell_1$, this yields
$$11\pi = \sum_{i=1}^{13}(\ell_i,\ell_{i+1}) \le 13\cdot \frac{5\pi}{6},$$
a contradiction.
\end{example}

\begin{example}\label{ex-schoen} We describe the movable cone of a smooth Calabi--Yau threefold that is not well-clipped, but nonetheless satisfies the Morrison-Kawamata cone conjecture.
Let $X=S_1\times_{\P^1}S_2$ be a Schoen threefold, as in \cite{Sc88,Na91,GM93}. It satisfies the nef cone conjecture by \cite{GM93}, as well as the movable cone conjecture by \cite{GLSW24}. We claim that neither the nef cone, nor the movable cone of $X$ are well-clipped.

Recall that $N^1(X) = p_1^*N^1(S_1)+p_2^*N^1(S_2)$, where $p_i:X\to S_i$ denote the two projections. Recall that there also is a fibration $f:X\to \P^1$ that factors through both $p_i$, and satisfies
$$p_1^*N^1(S_1)\cap p_2^*N^1(S_2) = f^*N^1(\P^1).$$ 
Here, we have $\rho(X)=19$ and $\rho(S_1)=\rho(S_2)=10$. Recall also that
the cones $\Nef(X)$ and $\Mov(X)$ are locally rational polyhedral at a non-zero boundary point $v$ if and only if $v\notin p_1^*N^1(S_1) \cup p_2^*N^1(S_2)$.
Let $\mathcal{C}$ denote the nef or the movable cone of $X$ and denote by $\mathfrak{S}$ the set of boundary points where $\mathcal{C}$ is not locally rational polyhedral, i.e.,
$$\mathfrak{S}=\partial\mathcal{C}\cap (p_1^*N^1(S_1) \cup p_2^*N^1(S_2))\setminus \{0\}.$$

Arguing by contradiction, we assume that $\mathcal{C}$ is well-clipped in a self-dual homogeneous cone $\mathcal{A}$, which decomposes as
$$\mathcal{A} = \bigoplus_{j=1}^r \mathcal{A}_j.$$
If one of the $\mathcal{A}_j$ has dimension $3$ or higher, then Theorem \ref{thm-classification} shows that $\partial \mathcal{A}_j$ is nowhere locally rational polyhedral. If moreover $\mathcal{A}_j$ is not of hyperbolic type, this provides an open set of the boundary $\partial\mathcal{C}$ that is contained in $\mathfrak{S}$, a contradiction.
Hence each $\mathcal{A}_j$ is a halfline or a hyperbolic cone.

Let $v\in\mathfrak{S}$. We decompose $v=\sum_{j=1}^r v_j$ with $v_j\in \overline{\mathcal{A}_j}\cap \mathcal{C}$ and see that there exists $k\in J$ for which $\mathcal{A}_k$ is of hyperbolic type and $q_k(v_k)=0$. We then note that near any point of
$$\R_{>0}\, v_k \oplus \bigoplus_{\substack{1\le j \le r\\ j\neq k}}\overline{\mathcal{A}_j}\cap \mathcal{C},$$ 
the cone $\mathcal{C}$ is not locally rational polyhedral. In particular, this provides a subspace of $\mathfrak{S}$ containing $v_k$ of dimension $19 - \dim \mathcal{A}_k + 1$, thus $\dim \mathcal{A}_k\ge 10$.

We apply this inequality for two points $v$ as above, one in each of the linear subspaces $p_i^*\Nef(S_i)\setminus  f^*\Nef(\P^1)$ contained in the boundary of $\mathcal{C}$. This provides two indices $k_1\neq k_2$ such that $\dim \mathcal{A}_{k_i} \ge 10$. But $10+10 > 19$, a contradiction.
\end{example}

An important fact to keep in mind, mentioned by Totaro in \cite[Section 3]{To10}, is how Assumption (ii) of Definition \ref{def-nicelycutout} tends to fail in presence of canonical surface singularities. We present an example inspired by \cite[Section 3]{To10}.

\begin{example}\label{ex-kltK3}
Let $E,F$ be elliptic curves, let $\Sigma = E\times F /\langle -1\rangle$. Let $S$ be the partial minimal resolution of $\Sigma$ obtained by resolving all $A_1$-singularities except for one, say the image of $0\in E\times F$. The surface $S$ is a K3 surface with canonical singularities in the sense of \cite{GGK}. 

Let $C$ denote the strict transform in $S$ of the image of $E\times \{0\}$ in $\Sigma$. It is is a smooth rational curve passing through a single $A_1$-singular point of $S$, thus $C^2 = -\frac{3}{2}$. Meanwhile, the exceptional divisor $D_p$ over the image of a $2$-torsion point $(p,0)$ satisfies $C\cdot D_p = 1$. The class of $C$ spans an extremal ray of the Mori cone $\overline{\rm NE}(S)$, and therefore defines a necessary hyperplane of the nef cone $\Nef(S)$ (see Remark \ref{rem-minimal}). However, the unique orthogonal reflection with respect to that hyperplane sends the Cartier divisor $D_p$ to $D_p + \frac{4}{3}C$, which is not integral.
\end{example}

We thank the referee for suggesting the following example.

\begin{example}\label{ex-rationalsurface}
Let $j$ denote the first third root of unity and $E_j$ the elliptic curve with complex multiplication by $j$. Consider the surface 
$$S_0=E_j\times E_j/\langle\mathrm{diag}(j,j)\rangle.$$ 
Let $\varepsilon:S\to S_0$ denote its minimal resolution. Since $S_0$ has nine isolated singularities, and since each of them is of type $\frac{1}{3}(1,1)$, the resolution $\varepsilon$ has nine exceptional divisors $C_1,\ldots, C_9$, each of them of square $-3$ in $S$.
It also holds $$K_S+\frac{1}{3}(C_1+\ldots + C_9) = \varepsilon^*K_{S_0} \equiv 0,$$
so the surface $S$ underlies a klt Calabi--Yau pair.

Since $C_1$ spans an extremal ray of $\overline{\mathrm{NE}}(S)$ and has square $-3$, we already see that the cone $\Nef(S)$ is not well-clipped in the hyperbolic cone induced by the intersection form on $S$. We now prove, more generally, that $\Nef(S)$ is not well-clipped in any self-dual homogeneous cone.

By contradiction, assume that $\Nef(S)$ is well-clipped in a self-dual homogeneous cone $\mathcal{A}$. Note that 
\begin{itemize}
\item $\rho(S_0) = 4$ and $\rho(S) = 13$;
\item $\Nef(S_0)\simeq \Nef(E_j\times E_j)$ is an irreducible self-dual homogeneous cone that is not hyperbolic;
\item $\varepsilon^*\Nef(S_0)$ appears as a face of dimension 4 of the cone $\Nef(S)$;
\item for each $1\le i\le 9$, the hyperplane $H_i = \{D\in N^1(S)\mid D\cdot C_i = 0\}$ supports the cone $\Nef(S)$ along a face of codimension $1$.
\end{itemize}
From these facts, an argument along the lines of Example \ref{ex-schoen} forces $\mathcal{A}$ to decompose into the following irreducible self-dual homogeneous summands:
$$\mathcal{A}=\varepsilon^*\Nef(S_0)\oplus \mathcal{H},$$
where $\mathcal{H}$ is a hyperbolic cone lying in a $9$-dimensional subspace of $N^1(S)$, which we denote by $V$. The hyperplanes $H_i$ all intersect the interior of $\mathcal{H}$ in $V$.

Since $\Nef(S)$ is well-clipped in $\mathcal{A}$, by Assumption (i) of Definition \ref{def-nicelycutout}, we must be able to clip it with hyperplanes of the form $\varepsilon^*N^1(S_0)\oplus W$, where $W\subset V$ is itself a hyperplane. 
Set $B$ to be the image of the elliptic curve $E_j\times\{0\}$ in the quotient surface $S_0$ and let $B'$ denote its strict transform in $S$. Since $B'$ is a $(-1)$-curve, it spans an extremal ray of $\overline{\mathrm{NE}}(S)$, and so the hyperplane
$$\{D\in N^1(S)\mid D\cdot B' = 0\}$$
supports the cone $\Nef(S)$ along a face of codimension $1$. However, since $\varepsilon_*B'= B$ is not numerically trivial on $S_0$, this hyperplane does not contain the subspace $\varepsilon^*N^1(S_0)$, a contradiction.
\end{example}

The next question, if it had an affirmative answer, would be a great motivation to try and understand cones of positive divisors arising from algebraic geometry in terms of well-clipped cones. However, there seems to me to be no reason to expect a positive answer to it, even for movable cones in dimension $3$.

\begin{question}
Is the movable cone of a smooth projective Fano variety well-clipped?
\end{question}

\subsection{Fundamental domains and well-clipped cones}

We start with a simple lemma.

\begin{lemma}\label{lem-coxeter}
Let $\mathcal{C}$ be a cone in $V=V_{\Z}\otimes \R$ that is well-clipped in some self-dual homogeneous cone $\mathcal{A}$ with quadratic form $q$ by hyperplanes $(H_i)_{i\in I}$. Then, for every $i,j\in I$ such that $H_i\neq H_j$, we have
$$\frac{q(e_i,e_j)^2}{q(e_i,e_i)q(e_j,e_j)} \in \left\lbrace \cos^2\left(\frac{\pi}{n_{ij}}\right) \mid n=2,3,4,6,\infty \right\rbrace \cup (1,\infty).$$
\end{lemma}

\begin{proof}
By Assumptions (ii) and (iii) of Definition \ref{def-nicelycutout} and Lemma \ref{lem-roots} (3) and (4), we see that
$$\frac{q(e_i,e_j)^2}{q(e_i,e_i)q(e_j,e_j)} \in \left\lbrace 0, \frac{1}{4}, \frac{1}{2}, \frac{3}{4}\right\rbrace \cup (1,\infty).$$
Checking the squares of the appropriate values of cosine concludes the proof.
\end{proof}

\begin{corollary}\label{cor-Cfunddom}
Let $\mathcal{C}$ be a well-clipped cone in a self-dual homogeneous cone $\mathcal{A}$ in $V=V_{\Z}\otimes \R$. Then the cone $\overline{\mathcal{C}}$ is a fundamental domain for the action of the group generated by the orthogonal reflections $(\sigma_i)_{i\in I}$ with respect to the side hyperplanes of $\mathcal{C}$ on the cone $\mathcal{A}^+$.
\end{corollary}

\begin{proof}
We first check the disjoint interior property. Let $w\in\langle \sigma_i\mid i\in I\rangle$ be such that 
$\mathcal{C}^{\circ}\cap w(\mathcal{C}^{\circ}) \neq \emptyset$.
Note that $w$ can be written as the product of finitely many reflections within our set of generators, say $w\in\langle \sigma_1,\ldots \sigma_m\rangle$. We pick a rational polyhedral cone $\Pi\subset \overline{\mathcal{C}}$ such that 
\begin{itemize}
\item The hyperplanes fixed by $\sigma_1,\ldots,\sigma_m$ support the cone $\Pi$ along faces of codimension $1$.
\item We have $\Pi^{\circ}\cap w(\Pi^{\circ})\neq \emptyset$.
\end{itemize}
By Lemma \ref{lem-coxeter}, we can apply \cite[Theorem 1, Proposition 6]{Vin71} to the cone $\Pi$. We obtain that $w=1$, which concludes the proof of the disjoint interior property.

\medskip

We now check the covering property. We fix a point $c\in\mathcal{C}^{\circ}$. For $a\in\mathcal{A}^+$, the line segment joining $a$ and $c$ is compact. Since the reflection group $W=\langle \sigma_i \mid i\in I\rangle$ is discrete, that line segment only crosses finitely many of the hyperplanes of the form $P_j := \ker(r_j-\mathrm{id}_V)$, where $\{r_j\}_{j\in J}$ denotes the at most countable set of all orthogonal hyperplane reflections in $W$. For every $j\in J$, we denote by $u_j$ an element of 
$P_j^{\perp}\cap V_{\Z}.$
We denote by $N_a$ the number of hyperplanes crossed (if $a$ is contained in some of the hyperplanes, we do not count them).
We now fix $a\in\mathcal{A}^+$, and choose an element $m$ in the orbit $W\cdot a$ that minimizes $N_{m}$. 

We claim that $N_{m}=0$. Assume by contradiction that $N_{m}\ge 1$. Then, we can take a crossing hyperplane $P_j$ for the segment $(m,c)$. Note that $q(c,u_j) > 0$ and $q(m,u_j)< 0$. Drawing in the plane containing $c,m,r_j(m)$, we have a triangle in $\R^2$. The crossing hyperplanes restrict to lines in $\R^2$. By Pasch's axiom, each line must satisfy exactly one of the following three propositions:
\begin{itemize}
\item[(i)] it is disjoint from the triangle;
\item[(ii)] it contains a vertex of the triangle;
\item[(iii)] it intersects exactly two sides of the triangle in their relative interiors.
\end{itemize}
We deduce two facts:
\begin{itemize}
\item The hyperplane $P_j$ does not cross the segment $(r_j(m),c)$. 
\item The involution $r_j$ induces a bijection
\begin{align*}
\{P_k\mid k\neq j, \,P_k\mbox{ intersects } 
& (m,c)\mbox{ but not }(r_j(m),c)\} \\
&{\scriptstyle r_j}{\updownarrow}\\
\{P_k\mid k\neq j, \,P_k\mbox{ intersects } 
& (r_j(m),c)\mbox{ but not }(m,c)\}.
\end{align*}
Indeed, let $P_k$ be a hyperplane that does not cross $(m,c)$ but crosses $(r_j(m),c)$. Then $P_k$ crosses $(r_j(m),m)$. By convexity, $P_k$ also crosses any segment of the form $(r_j(m),t)$, for $t$ on the segment $[c,m]$. We apply this remark to the intersection point of the two segments $(c,m)$ and $(r_j(c),r_j(m))$, that is
$$t = q(c,u_j) m - q(m,u_j) c,$$
and obtain that $P_k$ crosses the segment $(r_j(m),t) \subset (r_j(m),r_j(c))$. Hence, the hyperplane $r_j(P_k)$ crosses the segment $(c,m)$. It still crosses the segment $(m,r_j(m))$, thus does not cross the segment $(c,r_j(m))$. So the reflection $r_j$ induces a bijection such as described above.
\end{itemize}
We derive $N_{r_j(m)} = N_m - 1$, a contradiction.
Since $N_{m}=0$, the point $m$ belongs to $\overline{\mathcal{C}}$. This shows the covering property.
\end{proof}

We introduce the notions of a {\it neatly clipped} and a {\it perfectly clipped} cone.

\begin{definition}\label{def-neatlyclipped}
Let $V=V_{\Z}\otimes \R$ be a finite dimensional vector space and let $\mathcal{A}$ be a self-dual homogeneous cone in $V$. We say that a full-dimensional non-degenerate convex cone $\mathcal{C}$ is 
\begin{itemize}
\item {\it neatly clipped in} $\mathcal{A}$ if $\mathcal{C}$ is well-clipped in $\mathcal{A}$, and if there is a subgroup $W$ in $\mathrm{Aut}(\mathcal{A},V_{\Z})$ such that 
\begin{enumerate}
\item[(i)] for any $w\in W\setminus\{1\}$, 
$w(\mathcal{C}^{\circ})\cap \mathcal{C}^{\circ} = \emptyset$;
\item[(ii)] the group $W$ is preserved by conjugation under a subgroup 
$$\Gamma_{\mathcal{C}} < \Aut(\mathcal{C}^{\circ},\mathcal{A},V_{\Z});$$
\item[(iii)] the subgroup $\Gamma = W \rtimes  \Gamma_{\mathcal{C}}$
is an arithmetic subgroup of $\mathrm{Aut}(\mathcal{A})$;
\item[(iv)] the subgroup $\Gamma$ in (iii) contains the reflection group 
$\langle \sigma_i\mid i\in I\rangle$.
\end{enumerate}
\item {\it perfectly clipped in} $\mathcal{A}$ if $\mathcal{C}$ is well-clipped in $\mathcal{A}$,
and the orthogonal reflections induced by the hyperplanes delimiting $\mathcal{C}$ are such that
$$\langle \sigma_i \mid i\in I\rangle \rtimes \Aut(\mathcal{C}^{\circ},\mathcal{A},V_{\Z})$$
is an arithmetic subgroup of $\Aut(\mathcal{A})$.
\end{itemize}
\end{definition}

\begin{remark}\label{rem-perfneat}
Note that a perfectly clipped cone is always neatly clipped with $W= \langle\sigma_i\mid i\in I\rangle$ and $\Gamma_{\mathcal{C}} = \Aut(\mathcal{C}^{\circ},\mathcal{A},V_{\Z})$. Indeed, the disjoint interior property follows from Corollary \ref{cor-Cfunddom}, and the semidirect product structure in the definition of perfectly clipped cone is naturally induced by the fact that $\Aut(\mathcal{C}^{\circ},\mathcal{A},V_{\Z})$ preserves the set of hyperplanes $(H_i)_{i\in I}$.
\end{remark}

Let us take the list of Example \ref{ex-list}, and note that all of those well-clipped cones are in fact perfectly clipped.

\begin{example}\label{ex-listperfclipped} There are many instances of perfectly clipped cones.
\begin{enumerate}
\item[(1+7+8)] A self-dual homogeneous cone is always perfectly clipped. In particular, any simplicial cone is perfectly clipped, and any non-degenerate convex cone in $\R^2$ is perfectly clipped. For the same reason, the movable cones of abelian varieties are perfectly clipped.
\item[(2+3)] The nef cone of a smooth del Pezzo surface is perfectly clipped by Proposition \ref{prop-characterize-RPFD} below and by the cone theorem. The nef cone of a smooth projective K3 is also perfectly clipped: It follows from Proposition \ref{prop-characterize-RPFD} and the Torelli theorem \cite[Paragraph 6, Theorem 1]{PSSTorelli}.
\item[(4)] Let $S$ be a surface underlying a klt Calabi--Yau pair. The cone $(\overline{\mathrm{NE}}(S)_{K_S\le 0})^{\vee}$ (appearing in Example \ref{ex-list} and Lemma \ref{lem-surface-nicely}) is perfectly clipped, see Lemma \ref{lem-surfperfectly} below. This is a consequence of the cone conjecture in dimension $2$ proven by Totaro \cite{To10}.
\item[(5)] The movable cone of a smooth projective hyperkähler manifold is perfectly clipped by the work of Markman \cite[Theorem 6.18 (4+5), Lemma 6.23]{Markman11}, notably thanks to the Hodge theoretic Torelli theorem \cite[Theorem 1.3]{Markman11}. Please note the words of caution in \cite[Paragraph following Definition 1.1, Caution 6.19]{Markman11} and the counterexamples \cite{Deb,Na02,Markman10}, that convinced us to work on descending the movable (rather than the nef) cone conjecture and to take arithmetic subgroups in Definition \ref{def-neatlyclipped} rather than the whole group $\Aut(\mathcal{A},V_{\Z})$ (compare with \cite{OS01}).
\item[(6)] The movable cone of a projective primitive symplectic variety $X$ with terminal $\Q$-factorial singularities is perfectly clipped. It is clear for $b_2(X) \le 4$ by Item (1) above, and follows for $b_2(X) \ge 5$ from the work of Bakker-Lehn \cite[Theorem 8.2]{BL18} and Lehn-Mongardi-Pacienza \cite[Theorem 5.12]{LMP24}, using notably a Torelli theorem \cite[Theorem 4.9]{LMP24}.
\item[(9)] The nef cone of $\P^2$ blown up at $9$ general points and the movable cone of $\P^3$ blown up at 8 very general points are both perfectly clipped, by Proposition \ref{prop-characterize-RPFD} below and by \cite{Nagata60} and \cite{SX23} respectively. 
\item[(10)] A direct sum of perfectly clipped cones remains perfectly clipped. The proof is essentially the same as that of Lemma \ref{lem-directsum}. 
\end{enumerate}
\end{example}

Our motivation to introduce these notions of neatly and perfectly clipped cones is the following proposition, which is inspired by the work of Ash \cite[Chapter II]{AMRT}, Sterk \cite{Sterk85}, Markman \cite{Markman11}, Looijenga \cite[Example 4.8]{Loo14}, and Lehn--Mongardi--Pacienza \cite{LMP24}. One can view it as a characterization of those well-clipped cones which admit rational polyhedral fundamental domains under some group action.

\begin{proposition}\label{prop-characterize-RPFD}
Let $\mathcal{C}$ be a well-clipped cone in a self-dual homogeneous cone $\mathcal{A}$ in $V=V_{\Z}\otimes \R$. Then the following statements are equivalent
\begin{enumerate}
\item[(i)] $\mathcal{C}$ is perfectly clipped in $\mathcal{A}$;
\item[(ii)] $\mathcal{C}$ is neatly clipped in $\mathcal{A}$;
\item[(iii)] there is a rational polyhedral fundamental domain for the action of $\Aut(\mathcal{C}^{\circ},\mathcal{A},V_{\Z})$ on the cone $\mathcal{C}^+$.
\end{enumerate}
\end{proposition}

A key result involved in the proof of Proposition \ref{prop-characterize-RPFD} is the following theorem, originally attributed to Vinberg.

\begin{theorem}\label{thm-amrt}
Let $\mathcal{A}$ be a self-dual homogeneous cone in $V=V_{\Z}\otimes \R$. For any arithmetic subgroup $\Gamma$ of $\Aut(\mathcal{A})$, the action of $\Gamma$ on $\mathcal{A}^+$ admits a rational polyhedral fundamental domain.
\end{theorem}

\begin{proof}
This follows from the work of Vinberg on reduction theory and Siegel sets; see \cite[Chapter II, Theorem 4.1]{AMRT}, \cite[Propositions 4.2, 4.6]{Loo14}. Note that the compatibility of $\mathcal{A}$ with the rational structure $V_{\Q}$ (\textit{c.f.} Definition \ref{def-compatiblerat}) is used here to apply \cite[Proposition 4.2]{Loo14}.
\end{proof}

In order to prove Proposition \ref{prop-characterize-RPFD}, we also need a technical lemma.

\begin{lemma}\label{lem-AMRTstyle}
Let $\mathcal{C}$ be a well-clipped cone in a self-dual homogeneous cone $\mathcal{A}$ in $V=V_{\Z}\otimes \R$. Then for any arithmetic group $\Gamma<\Aut(\mathcal{A},V_{\Z})$ that contains the group spanned by the orthogonal reflections $\langle \sigma_i\mid i\in I\rangle$, there exists a rational polyhedral cone $\Pi\subset\mathcal{C}^+$ such that
\begin{enumerate}
\item[(1)] We have $$\mathcal{C}^+ =
\bigcup_{\gamma\in \mathfrak{S}}
\gamma(\Pi),$$
where $\mathfrak{S}$ denotes the set $\{\gamma\in \Gamma\mid \gamma(\mathcal{C}^{\circ})\cap \mathcal{C}^{\circ}\neq\emptyset\}$.
\item[(2)] For $\gamma\neq \gamma'\in\mathfrak{S}$, it holds that
$\gamma(\Pi^{\circ})\cap \gamma'(\Pi^{\circ})=\emptyset.$
\end{enumerate}
If moreover $\mathcal{C}$ is neatly clipped in $\mathcal{A}$, then
\begin{enumerate}
\item[(3)] There is a rational polyhedral fundamental domain for the action of the group $\Aut(\mathcal{C}^{\circ},\mathcal{A},V_{\Z})$ on $\mathcal{C}^+$. Additionally, any group $\Gamma_{\mathcal{C}}$ that may appear in Definition \ref{def-neatlyclipped} for the cone $\mathcal{C}$ is of finite index in $\Aut(\mathcal{C}^{\circ},\mathcal{A},V_{\Z})$.
\end{enumerate}
\end{lemma}

\begin{proof}[Proof of Lemma \ref{lem-AMRTstyle}]
We start with the self-dual homogeneous cone $\mathcal{A}$. Let $\Gamma$ be an arithmetic subgroup of $\Aut(\mathcal{A})$ that contains all the orthogonal reflections $(\sigma_i)_{i\in I}$ with respect to the hyperplanes $(H_i)_{i\in I}$ cutting out $\mathcal{C}$. By Theorem \ref{thm-amrt}, there is a rational polyhedral fundamental domain for the action of $\Gamma$ on the cone $\mathcal{A}^+$. By \cite[Application 4.14]{Loo14}, we can perform the following construction {\it à la} Dirichlet--Voronoi: If we choose a point $a\in \mathcal{A}\cap V_{\Q}$ with trivial stabilizer in $\Gamma$, the cone
$$\Pi := \{x \in\mathcal{A}^+\mid \forall \gamma\in\Gamma,\; q(x,\gamma(a)) \ge q(x,a)\}$$
is a rational polyhedral fundamental domain for the action of $\Gamma$ on $\mathcal{A}^+$. From here on, we fix $a$ to be a point in $\mathcal{C}^{\circ}\cap V_{\Q}$ that has trivial stabilizer in $\Gamma$, and whose $\Gamma$-orbit avoids the hyperplanes $(H_i)_{i\in I}$ used to cut out $\mathcal{C}$. (Its existence essentially follows from \cite[Theorem 3.8]{Loo14}, see \cite[Page 8, Proof of Proposition 2.3]{GLW22} and \cite[Page 13, Proof of Lemma 3.5, Last Paragraph]{LiZhao} for more details.)

\medskip

We claim that $\Pi\subset \mathcal{C}^+$. 
Fix $x\in\Pi$. We have
$$q(x,\sigma_i(a)) - q(x,a) = \frac{-2q(a,e_i)q(x,e_i)}{q(e_i,e_i)} \ge 0,$$
which implies, since $a\in\mathcal{C}^{\circ}$, that $q(x,e_i) \ge 0$. Therefore, we have that $x\in \overline{\mathcal{C}}$. This shows that $\Pi\subset \overline{\mathcal{C}}$, and since $\Pi$ is rational polyhedral, it proves our claim.

The same argument shows, more generally, that for any $\gamma\in \Gamma$ satisfying $\gamma(a)\in\mathcal{C}^{\circ}$, we have 
$\gamma(\Pi)\subset \mathcal{C}^+.$ We also note that, for any $\gamma\in \Gamma$ such that $\gamma(a)\notin\mathcal{C}^{\circ}$, by our assumptions on the point $a$, there exists $i\in I$ such that
$q(\gamma(a),e_i) < 0$, and thus $\gamma(\Pi)$ is contained in the closed halfspace $\overline{H_{i,-}}$. In particular, the intersection $\gamma(\Pi)\cap \mathcal{C}^{\circ}$ is empty. This is an interesting dichotomic behavior between translates of $\Pi$ contained in $\mathcal{C}^+$, and translates of $\Pi$ disjoint from $\mathcal{C}^{\circ}$.

To conclude this proof, we need to understand whether/how translates of $\Pi$ cover the cone $\mathcal{C}^+$.
Let $x\in \mathcal{C}^{\circ}$. Since $\Pi$ is a fundamental domain for the action of $\Gamma$ on $\mathcal{A}^+$, we can find an element $\gamma\in\Gamma$ such that $\gamma(x)\in \Pi$. This yields that
$\gamma^{-1}(\Pi)\cap \mathcal{C}^{\circ}\neq \emptyset$, hence by the dichotomic behavior explained above, we must have $\gamma^{-1}(\Pi)\subset \mathcal{C}^+$, and $\gamma^{-1}(a)\in\mathcal{C}^{\circ}$. This shows that
$$\mathcal{C}^+ =
\bigcup_{\gamma\in \mathfrak{S}}
\gamma(\Pi),$$
where $\mathfrak{S}$ denotes the set $\{\gamma\in\Gamma\mid \gamma(a)\in\mathcal{C}^{\circ}\}$. This shows Point 1. Point 2 is clear by definition of $\Pi$, since $a$ has trivial $\Gamma$-stabilizer.

\medskip

We finally prove Point 3. Assume that $\mathcal{C}$ is neatly clipped, and take $\Gamma$ to be an arithmetic subgroup that also writes as in Definition \ref{def-neatlyclipped}
$$\Gamma = W \rtimes \Gamma_{\mathcal{C}},$$
for some subgroup $\Gamma_{\mathcal{C}} < \Aut(\mathcal{C}^{\circ},\mathcal{A},V_{\Z})$, and for a group $W$ such that non-trivial $W$-translates of $\mathcal{C}^{\circ}$ avoid $\mathcal{C}^{\circ}$.
Note that for $\gamma \in\Gamma$, we have 
$\gamma(\mathcal{C}^{\circ})\cap \mathcal{C}^{\circ}\neq\emptyset$ if and only if $\gamma \in \Gamma_{\mathcal{C}}.$
So $\mathfrak{S} = \Gamma_{\mathcal{C}}$ is a subgroup of $\Aut(\mathcal{A},V_{\Z})$, and $\Pi$ transparently is a rational polyhedral fundamental domain for its action. Since $\Gamma_{\mathcal{C}}$ is contained in $\Aut(\mathcal{C}^{\circ},\mathcal{A},V_{\Z})$, there also is a rational polyhedral fundamental domain for the action of the larger group $\Aut(\mathcal{C}^{\circ},\mathcal{A},V_{\Z})$ on $\mathcal{C}^+$, see for instance \cite[Proposition 4.1, Application 4.14]{Loo14}. The fact that $\Gamma_{\mathcal{C}}$ has finite index in $\Aut(\mathcal{C}^{\circ},\mathcal{A},V_{\Z})$ is then a consequence of the general fact \cite[Proposition 4.6]{Loo14}.
\end{proof}

We can now prove Proposition \ref{prop-characterize-RPFD}.

\begin{proof}[Proof of Proposition \ref{prop-characterize-RPFD}]
By Remark \ref{rem-perfneat}, Item (i) implies Item (ii). By Lemma \ref{lem-AMRTstyle}, Item (ii) implies Item (iii). 

We now assume Item (iii) and prove Item (i).
By Item (iii), there is a rational polyhedral fundamental domain $\Pi$ for the action of $\Aut(\mathcal{C}^{\circ},\mathcal{A},V_{\Z})$ on $\mathcal{C}^+$, thus by Corollary \ref{cor-Cfunddom}, $\Pi$ also is a rational polyhedral fundamental domain for the action of the group
$$\Gamma := \langle \sigma_i \mid i\in I\rangle \rtimes \Aut(\mathcal{C}^{\circ},\mathcal{A},V_{\Z})$$
on the cone $\mathcal{A}^+$. By Theorem \ref{thm-amrt}, there exists a rational polyhedral fundamental domain for the action of the group $\Aut(\mathcal{A},V_{\Z})$ on the cone $\mathcal{A}^+$ as well. By \cite[Proposition 4.6]{Loo14}, this shows that $\Gamma$ is commensurable to $\Aut(\mathcal{A},V_{\Z})$, i.e., an arithmetic subgroup of $\Aut(\mathcal{A})$.
\end{proof}

We conclude this section with the following result, promised in Example \ref{ex-listperfclipped} and furthering Lemma \ref{lem-surface-nicely}.

\begin{lemma}\label{lem-surfperfectly}
Let $S$ be a smooth projective surface underlying a klt Calabi--Yau pair $(S,\Delta)$ and denote by $q_S$ its intersection form. The cone $\mathcal{C}_S := (\overline{\mathrm{NE}}(S)_{K_S\le 0})^{\vee}$ is perfectly clipped in the hyperbolic cone $\mathcal{H}(q_S)$ associated to $q_S$. Moreover, the group $\Aut^*(S,\Delta)$ is of finite index in the group of linear automorphisms of $N^1(S)$ preserving $\mathcal{C}_S$, $\mathcal{H}(q_S)$, and the lattice of Cartier divisors on $S$.
\end{lemma}

\begin{proof} We consider the contragradient representation of the group $\Aut^*(S,\Delta)$ on the dual space $N^1(S)^{{\vee}}\simeq N_1(S)$.
It follows from the cone conjecture for the pair $(S,\Delta)$ (\textit{c.f.} \cite[Theorem 4.1, Lemma 4.2]{To10}) and from \cite[Proposition 4.1 (i) $\Leftrightarrow$ (i*)]{Loo14} that the action of the group $\Aut^*(S,\Delta)$ on the cone $\overline{\mathrm{NE}}(S)^+$ admits a rational polyhedral fundamental domain $\Pi$. 

The cone
$\Pi':= \Pi_{K_S\le 0}$ is also rational polyhedral, and forms a fundamental domain for the action of $\Aut^*(S,\Delta)$ on the subcone $(\overline{\mathrm{NE}}(S)_{K_S\le 0})^+$. By \cite[Proposition 4.1 (i) $\Leftrightarrow$ (i*)]{Loo14} again, we may dualize and obtain that the action of $\Aut^*(S,\Delta)$ on $\mathcal{C}_S^+$ admits a rational polyhedral fundamental domain. \textit{A fortiori}, the action of the larger group $\Aut(\mathcal{C}_S^{\circ},\mathcal{H}(q_S),N^1(S)_{\Z})$ on that cone admits a rational polyhedral fundamental domain too; see for instance \cite[Proposition 4.1, Application 4.14]{Loo14}. 
The implication $(iii)\Rightarrow(i)$ of Proposition \ref{prop-characterize-RPFD} shows that the cone $\mathcal{C}_S$ is perfectly clipped. 
The fact that $\Aut^*(S,\Delta)$ has finite index in $\Aut(\mathcal{C}_S^{\circ},\mathcal{H}(q_S),N^1(S)_{\Z})$ follows from \cite[Proposition 4.6]{Loo14}.
\end{proof}

Motivated by this geometric situation, we prove the following result.

\begin{lemma}\label{lem-convexfurtherslice}
Let $V = V_{\Z}\otimes \R$ be a finite dimensional vector space. Let $\mathcal{C}$ be a cone in $V$ and assume that a subgroup $\Gamma$ of $\Aut(\mathcal{C}^{\circ},V_{\Z})$ acts on the cone $\mathcal{C}^+$ with a rational polyhedral fundamental domain. Then, for any finite $\Gamma$-invariant collection of rational halfspaces $(W_i)_{1\le i\le r}$ in $V$, the action of $\Gamma$ on the cone
$$\left(\mathcal{C}\cap \bigcap_{i=1}^r W_i\right)^+$$
admits a rational polyhedral fundamental domain.
\end{lemma}

\begin{proof}
Since $\Gamma$ preserves the finite collection of the $(W_i)_{1\le i\le r}$, it also preserves the (possibly degenerate) convex cone
$$W= \bigcap_{i=1}^r W_i.$$
Let $\Pi$ be a rational polyhedral fundamental domain for the action of $\Gamma$ on $\mathcal{C^+}$. We claim that the cone
$\Pi_W:= \Pi\cap W$ is a rational polyhedral fundamental domain for the action of $\Gamma$ on ${\mathcal{C}_W}^+$, where $\mathcal{C}_W:=\mathcal{C}\cap W$.
The following two facts are immediate: $\Pi_W$ is contained in $\mathcal{C}_W$, and any two distinct $\Gamma$-translates of $\Pi_W$ have disjoint interiors.
Since $\Pi_W$ is the intersection of a rational polyhedral cone with a finite collection of rational halfspaces, it is a rational polyhedral cone. Finally, let $x\in{\mathcal{C}_W}^+$. Since $x\in\mathcal{C}$, there is an element $\gamma\in\Gamma$ such that $\gamma(x)\in \Pi$. Since $x\in W$ and $\gamma(W)=W$, we get that $\gamma(x)\in\Pi_W$. Hence, the rational polyhedral cone $\Pi_W$ is indeed a fundamental domain for the action of $\Gamma$ on ${\mathcal{C}_W}^+$.
\end{proof}

\section{Descent under finite quotients}

\subsection{Descent of self-dual homogeneous cones}

This result is inspired by the work of Monti--Quedo \cite{MQ}. It is surprisingly hard to write down, for a self-dual homogeneous cone $\mathcal{A}$ and a finite subgroup $G$ of $\Aut(\mathcal{A})$, an explicit group that acts transitively on the cone $\mathcal{A}^G$; That is without the powerful machinery of formally real Jordan algebras and of Theorem \ref{thm-koechervinberg}. Our proofs in this subsection highlight this fact. We make the choice to completely forego the use of the classification presented in Theorem \ref{thm-classification}.

\begin{lemma}\label{lem-quotient-selfdualhom}
Let $\mathcal{A}$ be a self-dual homogeneous cone in $V$. Let $G$ be a subgroup of $\Aut(\mathcal{A})$ acting with finite orbits. 
The following results hold:
\begin{enumerate}
\item The invariant cone $\mathcal{A}^G \subset V^{G}$ is self-dual homogeneous.
\item If $\mathcal{A}$ is compatible with a rational structure $V_{\Q}$ on $V$, then $\mathcal{A}^G$ is compatible with the rational structure $V_{\Q}^G$.
\item If $\mathcal{A}$ is of hyperbolic type, then $\mathcal{A}^G$ is a halfline, the direct sum of two halflines, or a cone of hyperbolic type.
\item Suppose that $\mathcal{A} = \mathcal{H}^{\oplus n}$, for some $n\ge 1$ and $\mathcal{H}$ a hyperbolic cone. Then $G$ acts by permutation on the set of summands. Suppose further that this action on the set of summands is transitive and denote by $G_0 \triangleleft G$ the subgroup preserving each summand. Then, the cone $\mathcal{A}^G$ viewed inside $\mathcal{A}$ is isomorphic to the cone $\mathcal{H}^{G_0}$ embedded diagonally into $\mathcal{H}^{\oplus n}$.
\end{enumerate}
\end{lemma}

\begin{proof}
We first prove (1). Since $\mathcal{A}$ is invariant under sum and preserved by $G$, it contains the finite sum of all elements of a given $G$-orbit. Thus, the invariant cone $\mathcal{A}^G$ is clearly non-empty. By Lemma \ref{lem-quadratic}, we take an $\Aut(\mathcal{A},V_{\Z})$-invariant quadratic form $\mathrm{tr}$ on $V$ with respect to which $\mathcal{A}$ is self-dual. Since it is $G$-invariant, its restriction to $V^G$ remains positive definite and the cone $\mathcal{A}^G$ remains self-dual with respect to it.

Let us take $e\in\mathcal{A}^G$. By Theorem \ref{thm-koechervinberg}, there is a formally real Jordan algebra $(V,e,\circ)$ associated to the pointed cone $(\mathcal{A},e)$. At the level of arrows, the equivalence of categories given by Theorem \ref{thm-koechervinberg} identifies the automorphism group of $(V,e,\circ)$ (as a Jordan algebra) with the automorphism group $\Aut(\mathcal{A},e)$ of our pointed cone. Since $(\mathcal{A},e)$ is $G$-invariant, so is the Jordan algebra $(V,e,\circ)$, that is:
$$g(x\circ y) = g(x)\circ g(y)\quad\forall\, g\in G,\,x,y\in V.$$
Restricting the operation, we obtain the invariant formally real Jordan algebra $(V^G,e,\circ)$ contained in the initial one. By Theorem \ref{thm-koechervinberg} and Remark \ref{rem-koechervin}, we get the self-dual homogeneous cone 
$$\mathcal{B} := \{b\circ b\mid b\in V^G,\, b\mbox{ invertible}\} \subset V^G\cap \mathcal{A} = \mathcal{A}^G.$$
Both $\mathcal{B}$ and $\mathcal{A}^G$ are self-dual with respect to the same quadratic form $\mathrm{tr}|_{V^G}$, thus the equality $\mathcal{A}^G = \mathcal{B}$ holds, proving that $\mathcal{A}^G$ is self-dual and homogeneous.

Note that $\mathcal{A}$ is compatible with the rational structure $V_{\Q}$ if and only if the corresponding Jordan operation $\circ$ (for a choice of $e\in\mathcal{A}\cap V_{\Q}$) restricts to an operation on $V_{\Q}$. This proves Item (2).

Now we prove Item (3). Suppose that $\mathcal{A}$ is of hyperbolic type. There is a hyperbolic quadratic form $q$ on $V$ such that 
$\partial\mathcal{A} \subset \{v\in V\mid q(v) = 0\}$. The boundary of $\mathcal{A}^G$ is also contained in the zero locus of the restriction $q|_{V^G}$. 
If $\dim V^G \le 2$, the cone $\mathcal{A}^G$ is a direct sum of one or two halflines, as wished. 

Assume now that $\dim V^G \ge 3$. Since $e\in V^G$ and $q(e) > 0$, the restricted quadratic form $q|_{V^G}$ is of index $(1,a,b)$ for some integers $a,b\ge 0$ with $a+b= \dim V^G -1$. If $b=0$, then $q|_{V^G}$ is hyperbolic and we are done. Otherwise $b\ge 1$, i.e., there exists $v\in V^G$ be such that $q(v,V^G) = \{0\}$. Since the form $q$ is $G$-invariant, for any $w\in V$, we have 
$$|G\cdot w|\cdot q(v,w)= q\left(v,\sum_{g\in G/{\rm Stab}(w)} g(w)\right) \in q(v,V^G) = \{0\},$$
yielding $q(v,V)=\{0\}$. But the form $q$ is hyperbolic, thus non-degenerate on $V$, a contradiction.

We finally prove Item (4). Since the decomposition of $\mathcal{A}$ into irreducible summands is unique, the group $G$ indeed acts by permutation on the set of the summands. Let $W$ denote the span of the first summand: We have $V = W^{\oplus n}$. Since $G$ acts transitively on the set of summands, if an element of $W^{\oplus n}$ is $G$-invariant, then it belongs to the diagonal $\Delta_W\simeq W$, and its first coordinate in $W$ is further invariant under $G_0$. The converse is easy to check, allowing us to identify
$V^G \subset V$ with the linear subspace $W^{G_0}\hookrightarrow W^{\oplus n},$ embedded diagonally. Intersecting with $\mathcal{A}$ concludes the proof.
\end{proof}

\begin{remark}
The proof of Lemma \ref{lem-descent} involves fixing a $G$-invariant element $e$ in $\mathcal{A}$. This is  essentially the same choice as the one made by Monti--Quedo in \cite[Page 13, Paragraph below Lemma 6.6]{MQ} of a $G$-invariant polarization on the abelian variety with a $G$-action.
\end{remark}

\begin{lemma}\label{lem-transitivecentralizer}
Let $\mathcal{A}$ be a self-dual homogeneous cone in $V$. Let $G$ be a subgroup of $\Aut(\mathcal{A})$ acting with finite orbits. Then, the centralizer $C_{\Aut(\mathcal{A})}(G)$ acts transitively on the cone $\mathcal{A}^G$.
\end{lemma}

\begin{proof}
Let $e,a\in\mathcal{A}^G$. Consider the Jordan algebra $(V,e,\circ)$ associated to the cone $\mathcal{A}$ pointed at $e$ by Theorem \ref{thm-koechervinberg}. By Remark \ref{rem-koechervin} and by the proof of Lemma \ref{lem-quotient-selfdualhom}, we have
$$\mathcal{A}^G = \{b\circ b \mid b\in V^G \mbox{ invertible}\}.$$
Let $b\in V^G$ be an invertible element such that $a=b\circ b$. Then the quadratic operator $Q(b)$ defined in Remark \ref{rem-koechervin} belongs to $\Aut(\mathcal{A})$ and sends $e$ to $a$. The fact that $Q(b)$ commutes with any $g\in G$ follows from the fact that both $b$ and $\circ$ are perserved by the $G$-action.
\end{proof}

\subsection{Descent of well-clipped cones}

The next proposition is inspired by the work of Oguiso--Sakurai \cite{OS01}. Before stating it, we need a definition.

\begin{definition}\label{def-rhoG}
Let $V$ be a finite dimensional $\R$-vector space, and let $G <\mathrm{GL}(V)$ be a subgroup.
The invariant subspace $V^G$ is preserved by the action of the normalizer subgroup $N_{\mathrm{GL}(V)}(G)$. We denote by
\begin{align*}
\rho^G: N_{\mathrm{GL}(V)}(G) 
&\to \mathrm{GL}(V^G) \\
\phi & \mapsto  \phi|_{V^G}
\end{align*}
the representation induced by restriction of automorphisms.
\end{definition}



\begin{proposition}\label{prop-descentwellclipped}
Let $\mathcal{A}$ be a self-dual homogeneous cone in $V=V_{\Z}\otimes \R$. 
Let $\mathcal{C}$ be a well-clipped cone in $\mathcal{A}$, cut out by hyperplanes $(H_i)_{i\in I}$ with corresponding orthogonal reflections $(\sigma_i)_{i\in I}$. Let $G$ be a subgroup of $\Aut(\mathcal{C}^{\circ},\mathcal{A},V_{\Z})$ acting with finite orbits.
Then there exist a subset $I^*\subset I$ and a self-dual homogeneous cone $\mathcal{B}\subset V^{G}=V^G_{\Z}\otimes \R$ with the same round part as $\mathcal{A}^G$, whose simplicial part is ruled by the simplical part of $\mathcal{A}^G$,
such that
\begin{enumerate}
\item The invariant cone $\mathcal{C}^G$ is well-clipped in $\mathcal{B}$, and is cut out by the hyperplanes $h_i = H_i\cap V^G$ indexed by $I^*$.
\item The orthogonal reflection group $\langle \tau_i\mid i\in I^*\rangle$ for the hyperplanes $(h_i)_{i\in I^*}$ is contained in $\rho^G(C_{\langle \sigma_i\mid i\in I\rangle}(G))$.
\end{enumerate}
\end{proposition}

We recall from Section 2 that $C_{\langle \sigma_i\mid i\in I\rangle}(G)$ denotes the centralizer of $G$ in ${\langle \sigma_i\mid i\in I\rangle}$. We also refer to Definitions \ref{def-rdpart} and \ref{def-simplicialpart-preorder} on round and simplicial parts.
In a first read, the reader might want to focus on the conclusion presented as Item (1); It says that being a well-clipped cone descends under finite quotients. The conclusion reached in Item (2) is equally important, but mostly for technical reasons.

The key of the proof is to extract orthogonal reflections on the invariant space $V^G$ out of a large Weyl group of orthogonal reflections on $V$, on which $G$ acts by permuting the (infinitely many) generators. Merely keeping the orthogonal reflections that are fixed by this action by permutation of $G$ is not enough: We need to build more. This difficulty is already mentioned in \cite[1.4]{OS01}, and the way we overcome it is quite similar to \cite[1.6]{OS01}. However, we have to work in the framework of convex cones, which is not as nice as working in the Néron-Severi space of a K3 surface (the set-up of \cite{OS01}) and bans us from using geometric tools like the Riemann--Roch and Noether formulas.

\begin{proof}
By Definition \ref{def-nicelycutout} and Remark \ref{rem-minimal}, we can write
\begin{equation}\label{eq-pfdescent}
\overset{\circ}{\mathcal{C}}=\mathcal{A}\cap \bigcap_{i\in I} {H_{i,+}},
\end{equation}
where $({H_i}^{+})_{i\in I}$ is a minimal set of open halfspaces delimited by rational hyperplanes $H_i$ satisfying Assumptions (i), (ii), (iii). We choose the index set $I$ to be minimal as in Remark \ref{rem-minimal}.
For $i\in I$, we set $e_i$ to be the generator of $H_i^{\perp}\cap H_{i,-}\cap V_{\Z}$, and we define $h_i := H_i\cap V^G$.
Each $h_i$ is a hyperplane in $V^G$. Indeed, picking an element $c_0\in\mathcal{C}^{\circ}$, we construct a $G$-invariant element of $\mathcal{C}^{\circ}$ as
\begin{equation}\label{eq-celement}
c := \sum_{g\in G/\mathrm{Stab}(c_0)} g(c_0).
\end{equation}
Since $q(e_i,c) > 0$, the orthogonal complement $H_i = e_i^{\perp}$ does not contain $V^G$, and thus $h_i$ is indeed a hyperplane in $V^G$.

Note that the finite sum
\begin{equation}\label{eq-eps}
\eps_i := \sum_{g\in G/{\rm Stab}(e_i)} g(e_i)
\end{equation}
is an element of $h_i^{\perp}\subset V^G$, and of the invariant lattice $V_{\Z}^G$.
We decompose $\mathcal{A}$ into its irreducible summands and regroup them into the hyperbolic cones $(\mathcal{H}_j)_{j\in J}$ indexed by a finite set $J$ on one hand, and the sum of all remaining summands $\mathcal{A}_{\mathrm{nh}}$ on the other hand. Since the decomposition into irreducible summands is unique, the action by $G$ preserves the cone $\mathcal{A}_{\mathrm{nh}}$ and permutes the hyperbolic summands. We consider the induced action on the index set $J$. We denote by $J/G$ the set of the $G$-orbits in $J$. For each $\omega\in J/G$, we choose a representative $j_{\omega}\in J$. We then denote by $n_{\omega}$ the cardinality of the orbit $G\cdot j_{\omega}$, and by $G_{\omega}$ the stabilizer of $j_{\omega}$ in $G$: Neither depends on the choice of the representative $j_{\omega}$. With this notation, we have
\begin{equation}\label{eqAG}
\mathcal{A}^G = (\mathcal{A}_{\mathrm{nh}})^G\oplus\bigoplus_{\omega\in J/G}({\mathcal{H}_{j_{\omega}}}^{\oplus n_{\omega}})^G = (\mathcal{A}_{\mathrm{nh}})^G\oplus\bigoplus_{\omega\in J/G}{\mathcal{H}_{j_{\omega}}}^{G_{\omega}},
\end{equation}
where the second equality follows from Lemma \ref{lem-quotient-selfdualhom} (4).

With the notations of Definition \ref{def-nicelycutout} (i), we set
$$J_d := \{j\in J\mid \dim {\mathcal{H}_j}^{G_j} = d\},$$
for $d\in\N$. Note that it is a union of $G$-orbits in $J$. We now define the smaller index set
$$I^* := \{i\in I \mid j(i)\notin J_1\cup J_2\mbox{ and } q(\eps_i) < 0\},$$
and the smaller cone
\begin{equation}\label{eqB}
\mathcal{B} := (\mathcal{A}_{{\rm nh}})^G \oplus \bigoplus_{\omega\in (J\setminus J_{2})/G} {\mathcal{H}_{j_{\omega}}}^{G_{\omega}} \oplus \bigoplus_{\psi\in J_2/G} \mathcal{B}_{\psi},
\end{equation}
where for $\psi\in J_2/G$, we set
$$\mathcal{B}_{\psi}:={\mathcal{H}_{j_{\psi}}}^{G_{\psi}} \cap \bigcap_{\substack{i\in I\\ [j(i)] = \psi\\ q(\eps_i) < 0}} h_{i,+}.$$

For $\psi\in J_2/G$, the cones ${\mathcal{H}_{j_{\psi}}}^{G_{\psi}}$ and $\mathcal{B}_{\psi}$ are open non-degenerate two-dimensional cones, hence self-dual homogeneous cones. Thus and by Lemma \ref{lem-quotient-selfdualhom} (1), $\mathcal{B}$ is a self-dual homogeneous cone in $V^G$. It is contained in $\mathcal{A}^G$.

For $\omega\in (J\setminus J_1\cup J_2)/G$, Lemma \ref{lem-quotient-selfdualhom} (3) show that the invariant cone ${\mathcal{H}_{j_{\omega}}}^{G_{\omega}}$ is hyperbolic, and thus appears as an irreducible factor in the round parts of $\mathcal{B}$ and $\mathcal{A}^G$. In fact, Identities \eqref{eqAG} and \eqref{eqB} imply that the two round parts coincide:
$$\rd\mathcal{B} = \rd (\mathcal{A}_{{\rm nh}})^G \oplus \bigoplus_{\omega\in (J\setminus J_1\cup J_{2})/G} {\mathcal{H}_{j_{\omega}}}^{G_{\omega}} = \rd\mathcal{A}^G.$$
For the simplicial parts, we will later prove that, for each $\psi\in J_2/G$, the simplicial cone ${\mathcal{H}_{j_{\psi}}}^{G_{\psi}}$ rules the simplicial cone $\mathcal{B}_{\psi}$. Then, taking the direct sum over all $\psi\in J_2/G$, adding the simplicial part of $(\mathcal{A}_{{\rm nh}})^G$ and the half-lines of the form ${\mathcal{H}_{j_{\omega}}}^{G_{\omega}}$ for $\omega\in J_1/G$ on both sides, it will follow from Remark \ref{rem-rules} that the simplicial part of $\mathcal{A}^G$ rules the simplicial part of $\mathcal{B}$.

\medskip

We now show that the $(h_i)_{i\in I^*}$ satisfy Assumption (i) of Definition \ref{def-nicelycutout}. Fix $i\in I^*$ and denote by $G_{j(i)}$ the stabilizer of the index $j(i)$ in $G$. By Lemma \ref{lem-quotient-selfdualhom} (3), the cone ${\mathcal{H}_{j(i)}}^{G_{j(i)}}$, which has dimension at least three, is of hyperbolic type. Recall from Equation \eqref{eq-eps} that $h_i$ is the orthogonal hyperplane to $\varepsilon_i\in V^G$. Since $i\in I^*$, we have $q(\varepsilon_i)<0$, so checking that $\varepsilon_i$ belongs to the $\R$-linear span of the cone ${\mathcal{H}_{j(i)}}^{G_{j(i)}}$ is enough to prove Assumption (i) for $h_i$. But, since $\varepsilon_i$ is the sum of the elements in the $G$-orbit of $e_i\in{\rm Span}_{\R}\mathcal{H}_{j(i)}$, it indeed holds
$$\varepsilon_i \in V^G\cap \bigoplus_{j\in G\cdot j(i)}{\rm Span}_{\R}\mathcal{H}_{j}
= {\rm Span}_{\R}{\mathcal{H}_{j(i)}}^{G_{j(i)}}.$$

\medskip

We then show that the $(h_i)_{i\in I^*}$ satisfy Assumption (iii) of Definition \ref{def-nicelycutout}. Fix $i, k\in I^*$ such that $h_i\neq h_k$. Then, by Lemma \ref{lem-roots} (5), it suffices to check that $q(\eps_i,\eps_k)\ge 0$. Proposition \ref{prop-roots} shows that the action of $G$ on $\mathcal{C}$ permutes the $(e_i)_{i\in I}$. Since $h_i\neq h_k$, the two orbits $G\cdot e_i$ and $G\cdot e_k$ are distinct, thus disjoint, and thus 
$$q(\eps_i,\eps_k) = \sum_{\substack{g\in G/{\rm Stab}(e_i)\\ g'\in G/{\rm Stab}(e_k)}} q(g(e_i),g'(e_k)) \ge 0,$$
using the fact that the $(H_i)_{i\in I}$ satisfy Assumption (iii).

\medskip

We now claim that $\mathcal{C}^G$ is cut out by the (possibly non-minimal) set of hyperplanes $(h_i)_{i\in I^*}$ in the cone $\mathcal{B}$. To see this, take Identity \eqref{eq-pfdescent}, intersect with $V^G$, and express $\mathcal{A}^G$ in terms of $\mathcal{B}$ using Identities \eqref{eqAG} and \eqref{eqB} to see that
\begin{equation}\label{eq-pfdescent-Gint}
\overset{\circ}{\mathcal{C}^G}=\mathcal{B}
\cap \bigcap_{\substack{i\in I\\ j(i)\in J_1}} {h_{i,+}}
\cap \bigcap_{\substack{i\in I\\ j(i)\in J_2\\ q(\eps_i)\ge 0}} {h_{i,+}}
\cap \bigcap_{\substack{d\ge 3 \\ i\in I\\ j(i)\in J_{d}}} h_{i,+} .
\end{equation}
Let $i\in I$ such that $j(i)\in J_1$ or $q(\eps_i) \ge 0$. Let $V_{j(i)}$ denote the linear span of the cone ${\mathcal{H}_{j(i)}}^{G_{j(i)}}$. As shown in Lemma \ref{lem-quotient-selfdualhom} (3), the restricted quadratic form $q|_{{V_{j(i)}}}$ is hyperbolic, and takes positive values on the cone ${\mathcal{H}_{j(i)}}^{G_{j(i)}}$. Note that in the case when $j(i)\in J_1$, the form is in fact positive definite, and so $q(\eps_i)\ge 0$. 
We are now in the more general case where $i\in I$ and $q(\eps_i)\ge 0$.
The non-negativity of $q(\eps_i)$ proves that the element $\eps_i$ or its opposite $-\eps_i$ belongs to the closure of the hyperbolic cone $\mathcal{H}_{j(i)}^{G_{j(i)}}$. From \cite[Proposition 6.4]{Iversen}, we derive that
\begin{itemize}
\item[a)] either $q(\eps_i,v)\ge 0$ for every $v\in {\mathcal{H}_{j(i)}}^{G_{j(i)}}$,
\item[b)] or $q(\eps_i,v)\le 0$ for every $v\in {\mathcal{H}_{j(i)}}^{G_{j(i)}}$.
\end{itemize}
By \eqref{eq-pfdescent-Gint}, the linear form $q(\eps_i,\cdot)$ takes positive values on ${\mathcal{C}^{\circ}}^G$, a fortiori on its non-empty subset ${\mathcal{C}^{\circ}}^G\cap V_{j(i)}$ contained in ${\mathcal{H}_{j(i)}}^{G_{j(i)}}$. Thus, Alternative a) holds, that is ${\mathcal{H}_{j(i)}}^{G_{j(i)}}$ is contained in $h_{i,+}$.
This discussion allows to rewrite Identity \eqref{eq-pfdescent-Gint} as 
\begin{equation}\label{eq-pfdescent-G}
\overset{\circ}{\mathcal{C}^G}=\mathcal{B}\cap \bigcap_{i\in I^*} h_{i,+},
\end{equation} 
and indeed $\mathcal{C}^G$ is cut out by the hyperplanes $(h_i)_{i\in I^*}$ in the cone $\mathcal{B}$.

\medskip

We now prove that, for every $i\in I$ such that $q(\eps_i)< 0$, the orthogonal reflection with respect to $h_i$, given by the formula
\begin{equation}\label{eq-taui}
\tau_i:x\in V^G \mapsto  x  -\frac{2q(e_i,x)}{q(e_i,\eps_i)}\eps_i \in V^G,
\end{equation}
is integral with respect to the invariant lattice $V_{\Z}^G$. Using the $G$-invariance of $q$, we can also rewrite
$$\tau_i(x) = x - \frac{2q(\eps_i,x)}{q(\eps_i,\eps_i)}\eps_i.$$

Fix $i\in I$ such that $q(\eps_i)< 0$. Let $s_i:=-q(e_i)\in\Z_{>0}$ and let $d_i$ denote the cardinality of the orbit $G\cdot e_i$. Let $\{g_1,\ldots g_{d_i}\}$ be a set of representatives of the left cosets $G/{\rm Stab}(e_i)$, with $g_1=1$. We take note that 
\begin{equation}\label{eq-negativesum}
0 > q(\eps_i) = d_i\, q(e_i,\eps_i) = d_i \left(-s_i + \sum_{u=2}^{d_i} q(e_i,g_u(e_i))\right).
\end{equation}
For $u\ge 2$, since $g_u(e_i)$ and $e_i$ are distinct, Lemma \ref{lem-roots} (5) shows that
$$q(e_i,g_u(e_i))\in \frac{s_i}{2}\, \Z_{\ge 0}.$$
In particular, there is at most one $u\ge 2$ such that $q(e_i,g_u(e_i))$ is non-zero, and then 
it equals $\frac{s_i}{2}$. Furthermore, $q(e_i,\eps_i)$ takes one of the two values $-s_i,\frac{-s_i}{2}$. This shows that $s_i\Z\subset q(e_i,\eps_i)\Z$. Since moreover $2q(e_i,V_{\Z}^G) \subset 2q(e_i,V_{\Z}) \subset s_i \Z$, we conclude by Formula \eqref{eq-taui} that $\tau_i$ is integral with respect to the invariant lattice $V_{\Z}^G$. 

\medskip

In particular, this shows that the hyperplanes $(h_i)_{i\in I^*}$ satisfy Assumption (ii) of Definition \ref{def-nicelycutout}.

\medskip

Let us now prove that, for every $i\in I$ such that $q(\eps_i)<0$, the reflection $\tau_i$ is an element of the group $\rho^G(C_{\langle \sigma_i\mid i\in I\rangle}(G))$.
To sum up, we already showed using inequality \eqref{eq-negativesum} that for any element $e$ in the orbit $G\cdot e_i$, there is at most one other element $e'$ in that orbit such that $q(e,e') > 0$, and in that case 
\begin{equation}\label{eq-lazy}
s_i = -q(e) = -q(e') = -q(e+e') = 2q(e,e').
\end{equation}
This allows to decompose the orbit $G\cdot e_i$ into finitely many disjoint subsets $(B_v)$ of size $1$ or $2$
$$G\cdot e_i = \bigsqcup_{v=1}^r B_v,$$
with $e,e'\in B_v$ for some $v$ if and only if $q(e,e')\neq 0$.
Since $q$ is $G$-invariant, this decomposition is preserved by the $G$-action on $G\cdot e_i$, which is transitive. In particular, all $B_v$ are of the same size, either 1 or 2. 
To each index $v$, we associate an element $b_v\in \langle \sigma_i\mid i\in I\rangle$ as follows:
\begin{itemize}
\item If $B_v=\{e\}$, we take $b_v = \sigma_e$.
\item If $B_v= \{e,e'\}$, we set $b_v = \sigma_e\sigma_{e'}\sigma_e = \sigma_{e'}\sigma_{e}\sigma_{e'}$, which is given by
$$b_v:x\in V\mapsto x - \frac{2q(e+e',x)}{q(e)}(e+e') \in V.$$
This can be checked by an explicit computation involving Identity \eqref{eq-lazy} or by a Coxeter group argument in the spirit of Lemma \ref{lem-coxeter}. Note that this element $b_v$ is fixed by the transposition exchanging $e$ and $e'$.
\end{itemize}
 Since the $B_v$ are mutually $q$-orthogonal, it is clear that all of the $b_v$ commute with one another. This ensures that the element $b:=\prod_{v=1}^r b_v$ is well-defined independently of the order in which the product is taken. Recall that each $g\in G$ induces a permutation $s_g$ on $\{1,\ldots r\}$ determined by $g(B_v) = B_{s_g(v)}$. It is easy to check from our formula that $gb_vg^{-1} = b_{s_g(v)}$. Taking the product, we see that $b$ belongs to the centralizer 
$C_{\langle \sigma_i\mid i\in I\rangle}(G).$
 
 What we claim is that $\tau_i = \rho^G(b)$. Indeed, it is easy to compute the composition of commuting reflections: 
\begin{itemize}
\item If all $B_v$ are of size $1$, we have $q(e_i,\eps_i)=q(e_i)$, and
 $$b|_{V^G}:x\in V^G \mapsto  x - \frac{2q(e_i,x)}{q(e_i)}\eps_i,$$
where we recall that $\eps_i$ is the sum of the orbit elements in $G\cdot e_i$ (taken once each), and use the fact that $q(e,x) = q(e_i,x)$ for $x\in V^G$ and for any $e\in G\cdot e_i$.
\item If all $B_v$ are of size $2$, we have $q(e_i,\eps_i)=\frac{1}{2}q(e_i)$, and
 $$b|_{V^G}:x\in V^G \mapsto  x - \frac{4q(e_i,x)}{q(e_i)}\eps_i 
 \; = x - \frac{2q(e_i,x)}{q(e_i,\eps_i)}\eps_i.$$
 \end{itemize}
 In either case, we have $\tau_i = \rho^G(b)$ as wished.
 
 \medskip

We finally return to the simplicial parts. We fix $k\in J_2$ and prove, as promised, that the cone $\mathcal{H}_k^{G_k}$ rules the cone $\mathcal{B}_k$. Let $V_k$ denote the linear span of the cone $\mathcal{H}_k^{G_k}$. Since $k\in J_2$, it is of dimension two. We just showed that the orthogonal reflection group
$$W_k := \langle \tau_i\mid i\in I,\, [j(i)] = [k],\, q(\eps_i) < 0\rangle$$
preserves the lattice $V_{\Z}^G$. It also preserves the cone $\mathcal{H}_k^{G_k}$ and the linear space $V_k$ in $V$. Restricting gives rise to a faithful action
$$W_k\hookrightarrow \Aut(\mathcal{H}_k^{G_k},q|_{V_k},V_k\cap V_{\Z}^G) \subset {\rm O}^+(q|_{V_k}).$$
Since the irreducible summands of $\mathcal{A}$ are compatible with the rational structure on $V$, the intersection of $V_{\Z}^G$ with $V_k$ remains a lattice in $V_k$. This shows that $W_k$ is a discrete subgroup of ${\rm O}^+(q|_{V_k})$. Choosing the two extremal rays of the cone $\mathcal{H}_k^{G_k}$ as a (real, not necessarily rational) basis for $V_k$ further identifies a subgroup of index at most two of $W_k$ with a discrete subgroup of the diagonal matrices 
$$\left\lbrace \mathrm{diag}\left(t,\frac{1}{t}\right)\mid t\in\R_{> 0}\right\rbrace.$$
Via the isomorphism of topological groups
$$\theta\in (\R,+)\mapsto \mathrm{diag}\left(e^{\theta},e^{-\theta}\right),$$
we realize our subgroup of index at most two of $W_k$ as a discrete subgroup of $(\R,+)$, that is a trivial or infinite cyclic group. Thus, $W_k$ is isomorphic to a subgroup of $\Z\rtimes\Z/2\Z$.
\begin{itemize}
\item If $W_k$ is trivial, then $\mathcal{B}_k = \mathcal{H}_k^{G_k}$.
\item The case $W_k\simeq \Z$ cannot occur, since $W_k$ is generated by involutions.
\item If $W_k\simeq \Z/2\Z$, then $\mathcal{B}_k$ is cut out by one rational hyperplane inside $\mathcal{H}_k^{G_k}$, thus ruled by $\mathcal{H}_k^{G_k}$ by the last item of Remark \ref{rem-rules}.
\item Finally, if $W_k\simeq \Z\rtimes\Z/2\Z$, let $r$ denote the generator of the cyclic subgroup of index $2$. Fix $i_0\in I$ such that $[j(i_0)]=[k]$ and $q(\eps_{i_0})<0$. We have the following description:
$$\{h_i\mid i\in I, [j(i)]=[k], q(\eps_i)<0\} = \{r^n(h_{i_0})\mid n\in\Z\}.$$
The cone $\mathcal{B}_k$ is cut out by the halfspaces corresponding to these hyperplanes in $\mathcal{H}_k^{G_k}$. Thus, it is a connected component of 
$$\mathcal{H}_k^{G_k}\setminus \bigcup_{n\in\Z} r^n(h_{i_0}).$$
So $\mathcal{B}_k$ is spanned by two rays of the form $r^n(h_{i_0})$ and $r^{n+1}(h_{i_0})$ for some $n\in\Z$, that are both rational. Thus, by the third item of Remark \ref{rem-rules}, $\mathcal{B}_k$ is ruled by $\mathcal{H}_k^{G_k}$.
\end{itemize}

\end{proof}

\subsection{Centralizer and symmetries of the invariant cone}

\begin{lemma}\label{lem-normalize-commens}
Let $\Gamma_1$ and $\Gamma_2$ be two commensurable subgroups of a group $\Gamma$, and let $G$ be a subgroup of $\Gamma$. Then the centralizers $C_{\Gamma_1}(G)$ and $C_{\Gamma_2}(G)$ are commensurable.
\end{lemma}

\begin{proof}
It suffices to check that for any finite index subgroup $\Gamma_0<\Gamma_i$, the inclusion of centralizers $C_{\Gamma_0}(G) < C_{\Gamma_i}(G)$ is of finite index too. This follows from the identity
$$C_{\Gamma_0}(G) = \Gamma_0\cap C_{\Gamma_i}(G),$$
by \cite[(3.13)(i)]{SuzukiGroup}.
\end{proof}

\begin{lemma}\label{lem-normalizesemidirectprod}
Let $W$ and $Q$ be two groups. Let $\Gamma$ denote the semi-direct product $W\rtimes Q$ and let $G$ be a subgroup of $Q$. Then the centralizers satisfy 
$$C_{\Gamma}(G) = C_W(G)\rtimes C_Q(G).$$
\end{lemma}

\begin{proof}
Let $\gamma\in \Gamma$. We write $\gamma = (w,q)$ with $w\in W$ and $q\in Q$. Take $g\in G$, and view it as $(1,g)\in \Gamma$.
We compute
$$ \gamma g = (w,q)(1,g) = (w,qg)$$
and
$$g \gamma = (1,g)(w,q) = (g\cdot w, gq) = (gwg^{-1}, gq).$$
Hence, $\gamma$ belongs to $C_{\Gamma}(G)$ if and only if $q\in C_Q(G)$ and $w\in C_W(G)$, as wished.
\end{proof}

The next result is the main point of this section. It compares the centralizer subgroup of a group $G$ acting with finite orbits on a homogeneous cone $\mathcal{A}$ with the symmetries of the invariant cone $\mathcal{A}^G$. Recall that the representation $\rho^G$ was defined in Definition \ref{def-rhoG}.

\begin{proposition}\label{prop-centralizer-vs-autinv}
Let $\mathcal{A}$ be a self-dual homogeneous cone in $V=V_{\Z}\otimes \R$. Let $G$ be a subgroup of $\Aut(\mathcal{A},V_{\Z})$ acting with finite orbits.
Then the inclusion of groups
$$
\mathrm{Im}(\rho^G: C_{\Aut(\mathcal{A},V_{\Z})}(G)\to \mathrm{GL}(V^G) )
\quad <\quad
\Aut(\mathcal{A}^G,{V_{\Z}}^G)
$$
is of finite index.
\end{proposition}

\begin{proof}
We consider the linear algebraic group $C_{\Aut(\mathcal{A})}(G)$ with its natural embedding in $\GL(V)$. Let $P$ denote the intersection of the identity component of $C_{\Aut(\mathcal{A})}(G)$ with the preimage of the special linear subgroup $\SL(V^G)$ by the representation $\rho^G$: It is a connected linear algebraic group. Let $P_{\Z}$ denote the intersection of $P$ with the lattice $\GL(V_{\Z})$. Clearly, we have an inclusion
\begin{equation}\label{eq-anotherinclusion}
\rho^G(P_{\Z}) \quad < \quad \Aut(\mathcal{A}^G,V_{\Z}^G).
\end{equation}
We will show that it has finite index.

Let us fix a quadratic form $\mathrm{tr}$ on $V$ that is positive definite, $V_{\Z}$-integral, $\Aut(\mathcal{A},V_{\Z})$-invariant, and with respect to which the cone $\mathcal{A}$ is self-dual. It exists by Lemma \ref{lem-quadratic}. We consider the characteristic function of the cone $\mathcal{A}^G$: For $a\in\mathcal{A}^G$, it is given by
 $$\phi(a):= \int_{\mathcal{A}} e^{-{\rm tr}(a,x)}{\rm d}x \quad \in \R_{ >0},$$
 where ${\rm d}x$ denotes the Euclidean volume form on $V^G$; see \cite[Page 10]{FK94}. By \cite[Prop. I.3.1]{FK94}, the action of the group $\Aut(\mathcal{A}^G)\cap \SL(V^G)$ preserves each level set of the function $\phi$. We record the following consequence for later use.
 
 \medskip

 \noindent {\bf Fact.} The subgroup $\rho^G(P)$ acts on the level set $\phi^{-1}(t)$ for each $t\in\R_{> 0}$.
 
 \medskip
 
 Consider the set $S$ of extremal points of the convex hull of $\mathcal{A}^G\cap V_{\Z}^G$ in $V^G$. It is preserved by both $\rho^G(P_{\Z})$ and $\Aut(\mathcal{A}^G,V_{\Z}^G)$. By Theorem \ref{thm-amrt} and \cite[Proposition 4.1]{Loo14}, it consists of finitely many orbits of $\Aut(\mathcal{A}^G,V_{\Z}^G)$. Hence, the function $\phi$ takes finitely many values $t_1,\ldots,t_m$ on $S$. If we prove that each intersection of the form $S\cap \phi^{-1}(t_i)$ consists of finitely many orbits of the group $\rho^G(P_{\Z})$, then $S$ itself consists of finitely many $\rho^G(P_{\Z})$-orbits, and thus Inclusion \eqref{eq-anotherinclusion} is indeed of finite index, as wished.
 
Let $t = t_i > 0$. We claim that the larger intersection $V_{\Z}^G\cap \phi^{-1}(t)$ consists of finitely many $\rho^G(P_{\Z})$-orbits. To prove that, we want to apply \cite[Theorem 6.9]{BCH} to the representation $\rho^G:P\to \GL(V^G)$. We need to check the following things:
\begin{itemize}
\item that $P$ is a reductive subgroup of $\GL(V)$;
\item that the closed subset $\phi^{-1}(t)$ is an orbit of $\rho^G(P)$.
\end{itemize}

We check the reductivity first. We already know that $P$ is connected. Transposition with respect to the form $\mathrm{tr}$ defines an involution $\theta$ on $\GL(V)$. To prove that $P$ is reductive, it suffices to show that $\theta$ preserves $P$. Note that $\theta$ preserves the group $\Aut(\mathcal{A})$ (since $\mathcal{A}$ is self-dual with respect to $\mathrm{tr}$) and the group $G$ (since $\mathrm{tr}$ is $G$-invariant). Hence, $\theta$ preserves the centralizer subgroup $C_{\Aut(\mathcal{A})}(G)$. Since $\theta$ fixes the identity matrix, the identity component of $C_{\Aut(\mathcal{A})}(G)$ is preserved as well. An elementary computation shows that the representation $\rho^G: C_{\Aut(\mathcal{A})}(G)\to \GL(V^G)$ is $\theta$-equivariant, the induced involution on $\GL(V^G)$ being given by transposition with respect to $\mathrm{tr}|_{V^G}$. Since transposition preserves the determinant, it preserves the special linear group $\SL(V^G)$ thus the group $P$ is preserved by $\theta$.

Next, we prove that the afore-mentioned action of $\rho^G(P)$ on the closed level set $\phi^{-1}(t)$ is transitive. By Lemma \ref{lem-transitivecentralizer}, the centralizer group $C_{\Aut(\mathcal{A})}(G)$ acts transitively via $\rho^G$ on the open cone $\mathcal{A}^G$. So does its identity component. For any two elements $e,a\in \phi^{-1}(t)$, we can choose $h$ in the identity component of $C_{\Aut(\mathcal{A})}(G)$ such that
$\rho^G(h)(e) = a$. By \cite[Prop. I.3.1]{FK94}, we then have
$$t = \phi(a)  = |\det \rho^G(h)|\phi(e) = |\det \rho^G(h)| t,$$
so $\det \rho^G(h) = 1$, i.e., $h\in P$. So $\phi^{-1}(t)$ is indeed a closed orbit of the group $\rho^G(P)$.
So \cite[Theorem 6.9]{BCH} applies and proves our claim.
\end{proof}

\subsection{Descent of perfectly clipped cones and more}

Before proving a descent result for perfectly clipped cones, we have to address the apparition of an auxiliary self-dual homogeneous cone, {\it a priori} smaller than the invariant cone, in the descent result for well-clipped cones given by Proposition \ref{prop-descentwellclipped}. This is what the discussion of round and simplicial parts prepared us to, in particular the following lemma.

\begin{lemma}\label{lem-rdpartfunddom}
Let $\mathcal{C}$ be a well-clipped cone in a self-dual homogeneous cone $\mathcal{B}$ in $V=V_{\Z}\otimes \R$. Let $\mathcal{A}$ be a self-dual homogeneous cone in $V=V_{\Z}\otimes \R$ with the same round part as $\mathcal{B}$, and whose simplicial part rules the simplicial part of $\mathcal{B}$.
Then
$\Aut(\mathcal{C}^{\circ},\mathcal{A})$ is contained in
$\Aut(\mathcal{C}^{\circ},\mathcal{B})$. 
Moreover, $\Aut(\mathcal{C}^{\circ},\mathcal{A},V_{\Z})$ is of finite index in $\Aut(\mathcal{C}^{\circ},\mathcal{B},V_{\Z})$, and $\Aut(\mathcal{A},\mathcal{B},V_{\Z})$ is of finite index in $\Aut(\mathcal{B},V_{\Z})$.
\end{lemma}

\begin{proof}
Denote by $\mathcal{R}$ the round part of both our self-dual homogeneous cones, by $R$ its linear span, by $k$ the codimension of $R$ in $V$. 
Note that 
$$\mathcal{A}=\mathcal{R}\oplus\Sigma_k
\quad\mbox{and}\quad
\mathcal{B}=\mathcal{R}\oplus\Xi_k$$
for simplicial cones $\Sigma_k$ and $\Xi_k$ of dimension $k$, with the same linear span $S$ in $V$. Since $\mathcal{C}$ is well-clipped in $\mathcal{B}$, we also have 
$\mathcal{C}= (\mathcal{C}\cap R) \oplus \Xi_k$.

Let us introduce a notation: For a not necessarily fully dimensional cone $\mathcal{H}$ in $V$ of linear span $H\subset V$, we denote by
$\Aut_H(\mathcal{H})$ the set of linear transformations of $H$ that preserve the cone $\mathcal{H}$.
By uniqueness of the decomposition in Theorem \ref{thm-classification}, we can decompose the groups
$$\Aut(\mathcal{A}) = \mathrm{Aut}_R(\mathcal{R}) \times 
\mathrm{Aut}_S(\Sigma_k)
\quad\mbox{and}\quad
\Aut(\mathcal{B}) = \mathrm{Aut}_R(\mathcal{R}) \times 
\mathrm{Aut}_S(\Xi_k)$$
with respect to the linear decomposition $V=R\oplus S$.
In particular, every linear automorphism of $V$ that preserves either of the cones $\mathcal{A}$ or $\mathcal{B}$ preserves the linear decomposition $V=R\oplus S$. Since $\mathcal{C}\cap S = \Xi_k$, we obtain the following descriptions
\begin{equation}\label{eq-expressions}
\begin{aligned}
\Aut(\mathcal{C}^{\circ},\mathcal{B}) 
& = \mathrm{Aut}_R(\mathcal{C}^{\circ}\cap R, \mathcal{R}) \times 
\mathrm{Aut}_S(\Xi_k),\\
\Aut(\mathcal{C}^{\circ},\mathcal{A}) 
& = \mathrm{Aut}_R(\mathcal{C}^{\circ}\cap R, \mathcal{R}) 
\times \mathrm{Aut}_S(\Xi_k,\Sigma_k),\\
\Aut(\mathcal{A},\mathcal{B}) 
& = \mathrm{Aut}_R(\mathcal{R}) \times 
\mathrm{Aut}_S(\Xi_k,\Sigma_k).
\end{aligned}
\end{equation}
The obvious inclusion 
$\Aut_S(\Xi_k, \Sigma_k)
< \Aut_S(\Xi_k)$ thus induces an inclusion
$$\Aut(\mathcal{C}^{\circ},\mathcal{A})
< \Aut(\mathcal{C}^{\circ},\mathcal{B}).$$

As the cones $\mathcal{A}$ and $\mathcal{B}$ are compatible with the rational structure on $V$, their decompositions into round and simplicial parts is defined over $\Q$. Hence, the intersections $S_{\Z}:=V_{\Z}\cap S$ and $R_{\Z}:=V_{\Z}\cap R$ remain lattices in $S$ and $R$ respectively, and the groups $\GL(V_{\Z})$ and $\GL(S_{\Z}\oplus R_{\Z})$ are commensurable.
Hence, we are left to show that the inclusions
\begin{align*}
\Aut(\mathcal{C}^{\circ},\mathcal{A},R_{\Z}\oplus S_{\Z}) 
& < \Aut(\mathcal{C}^{\circ},\mathcal{B},R_{\Z}\oplus S_{\Z}) ,\\
\Aut(\mathcal{A},\mathcal{B},R_{\Z}\oplus S_{\Z}) 
& < \Aut(\mathcal{B},R_{\Z}\oplus S_{\Z})
\end{align*}
are of finite index.
But, since $\Sigma_k$ rules $\Xi_k$, the inclusion
\begin{equation}
\Aut_S(\Xi_k,\Sigma_k,S_{\Z}) < \Aut_S(\Xi_k,S_{\Z}) 
\end{equation}
is of finite index.
We conclude by intersecting Identities \eqref{eq-expressions} and the description of $\Aut(\mathcal{B})$ with the group $\GL(R_{\Z}\oplus S_{\Z})$.
\end{proof}

The next proposition is the most crucial ingredient in the proof of Theorem \ref{thm-convexmain}. 

\begin{proposition}\label{prop-descent}
Let $\mathcal{C}$ be a perfectly clipped cone in a self-dual homogeneous cone $\mathcal{A}$ in $V=V_{\Z}\otimes \R$. Let $G < \Aut(\mathcal{C}^{\circ},\mathcal{A},V_{\Z})$ be a subgroup acting with finite orbits. Then the invariant cone $\mathcal{C}^G$ is perfectly clipped in a self-dual homogeneous cone $\mathcal{B}$, and we have an inclusion of finite index 
$$\rho^G(C_{\Aut(\mathcal{C}^{\circ},\mathcal{A},V_{\Z})}(G)) 
\quad < \quad \Aut({\mathcal{C}^G}^{\circ},\mathcal{B},V_{\Z}^G).$$
\end{proposition}

\begin{proof}
Let $\Gamma$ denote the arithmetic subgroup given by Definition \ref{def-neatlyclipped}, which writes
$$\Gamma = \langle \sigma_i \mid i\in I\rangle \rtimes \Aut(\mathcal{C}^{\circ},\mathcal{A},V_{\Z}),$$
and is of finite index in $\Aut(\mathcal{A},V_{\Z})$.
By Lemma \ref{lem-normalizesemidirectprod}, we have
\begin{equation}\label{eq-centralize-equality}
C_{\Gamma}(G) = C_{\langle \sigma_i \mid i\in I\rangle}(G)\rtimes C_{\Aut(\mathcal{C}^{\circ},\mathcal{A},V_{\Z})}(G).
\end{equation}
By Lemma \ref{lem-normalize-commens} and Proposition \ref{prop-centralizer-vs-autinv}, we derive an inclusion of finite index
\begin{equation}\label{eq-finindex-pfmain}
\rho^G(C_{\langle \sigma_i \mid i\in I\rangle}(G)) 
\rtimes \rho^G(C_{\Aut(\mathcal{C}^{\circ},\mathcal{A},V_{\Z})}(G))
\quad < \quad
\Aut(\mathcal{A}^G,V_{\Z}^G).
\end{equation}

\medskip

By Item (1) of Proposition \ref{prop-descentwellclipped}, the cone $\mathcal{C}^G$ is well-clipped in a self-dual homogeneous cone $\mathcal{B}$ which has the same round part as $\mathcal{A}^G$, and whose simplicial part is ruled by the simplicial part of $\mathcal{A}^G$. Also note that
$\rho^G(C_{\Aut(\mathcal{C}^{\circ},\mathcal{A},V_{\Z})}(G))$ is contained in $\Aut({\mathcal{C}^G}^{\circ},\mathcal{A}^G)$, thus preserves the cone $\mathcal{B}$ by Lemma \ref{lem-rdpartfunddom}.
In particular, by Lemma \ref{lem-rdpartfunddom} and Inclusion \eqref{eq-finindex-pfmain},
there is an inclusion of finite index
\begin{equation}\label{eq-arithemticinB-pfmain}
W
\rtimes \rho^G(C_{\Aut(\mathcal{C}^{\circ},\mathcal{A},V_{\Z})}(G))
\quad < \quad
\Aut(\mathcal{B},V_{\Z}^G),
\end{equation}
where $W$ denotes $\rho^G(C_{\langle \sigma_i \mid i\in I\rangle}(G)) \cap \Aut(\mathcal{B})$.
By Item (2) of Proposition \ref{prop-descentwellclipped}, we know that the orthogonal reflection group given by the hyperplane sides of $\mathcal{C}^G$ satisfies
$\langle \tau_i\mid i\in I^*\rangle 
<
W.$
These facts and an elementary verification of Assumption (i) of Definition \ref{def-neatlyclipped} show that $\mathcal{C}^G$ is in fact neatly clipped in $\mathcal{B}$. By Proposition \ref{prop-characterize-RPFD}, the cone $\mathcal{C}^G$ is also perfectly clipped in $\mathcal{B}$.

\medskip

Therefore, we can apply Lemma \ref{lem-AMRTstyle}. It provides is a rational polyhedral fundamental domain for the action of $\Aut({\mathcal{C}^G}^{\circ},\mathcal{B},V_{\Z}^G)$ on $(\mathcal{C}^G)^+$, and shows that the inclusion
$$\rho^G(C_{\Aut(\mathcal{C}^{\circ},\mathcal{A},V_{\Z})}(G))
\quad <\quad
\Aut({\mathcal{C}^G}^{\circ},\mathcal{B},V_{\Z}^G)
$$
is of finite index, as wished.
\end{proof}

\section{Proof of the main results}

We can now prove the theorems stated in the introduction.

\begin{proof}[Proof of Theorem \ref{thm-convexmain}]
The fact that well-clipped cones descend to well-clipped cones is Proposition \ref{prop-descentwellclipped}. 
Assume that there exists a rational polyhedral fundamental domain for the action of  $\Gamma := \Aut(\mathcal{C}^{\circ},\mathcal{A},V_{\Z})$ on the well-clipped cone $\mathcal{C}$. Then by Proposition \ref{prop-characterize-RPFD}, the cone $\mathcal{C}$ is perfectly clipped. Applying Proposition \ref{prop-descent}, \cite[Proposition 4.6]{Loo14}, and Proposition \ref{prop-characterize-RPFD}, we conclude that the cone ${\mathcal{C}^G}^+$ admits a rational polyhedral fundamental domain for the action of $\rho^G(C_{\Gamma}(G))$.
\end{proof}

We will prove a slight generalization of Theorem \ref{thm-main}.

\begin{theorem}\label{thm-maintechnical}
Let $(X,\Delta)$ be a pair defined over a perfect field $k$. Let $(\overline{X},\overline{\Delta})$ denote their base change to the algebraic closure $\overline{k}$ of $k$. 
\begin{enumerate}
\item Suppose that the movable cone $\Mov(\overline{X})$ is well-clipped inside a self-dual homogeneous $\PsAut(\overline{X},\overline{\Delta})$-invariant cone $\mathcal{A}$. Suppose also that the pair $(\overline{X},\overline{\Delta})$ satisfies the movable cone conjecture. Then the movable cone $\Mov(\overline{X})$ is perfectly clipped in $\mathcal{A}$ 
and the inclusion of groups 
$$\PsAut^*(\overline{X},\overline{\Delta}) < \Aut(\Mov(\overline{X})^{\circ}, \mathcal{A}, N^1(\overline{X})_{\Z\mathrm{-Weil}})$$
is of finite index.
\item Suppose that there exist a perfectly clipped $\PsAut(\overline{X},\overline{\Delta})$-invariant cone $\mathcal{C}$ inside a self-dual homogeneous $\PsAut(\overline{X},\overline{\Delta})$-invariant cone $\mathcal{A}$ in $N^1(\overline{X})$ such that 
\begin{itemize}
\item The closed movable cone $\overline{\Mov}(\overline{X})$ writes as the intersection of the closed cone $\overline{\mathcal{C}}$ with a finite collection of halfspaces that is invariant under both $\PsAut(\overline{X},\overline{\Delta})$ and the Galois group of $k$.
\item The group $\PsAut^*(\overline{X},\overline{\Delta})$ is of finite index of $\Aut(\mathcal{C}^{\circ}, \mathcal{A}, N^1(\overline{X})_{\Z\mathrm{-Weil}})$.
\end{itemize}
Then the movable cone conjecture holds for every quotient pair $(X/G,\Delta_G)$ of $(X,\Delta)$ by a finite subgroup $G$ of $\Aut_k(X,\Delta)$. 
\end{enumerate}
\end{theorem}

\begin{proof}
Let $G$ be a finite subgroup of $\Aut(X,\Delta)$ and let $\overline{G}$ be the group $G\times {\rm Gal}(\overline{k}/k)$ acting on the base change to the algebraic closure $(\overline{X},\overline{\Delta})$. Note that the induced action of $\overline{G}$ on the N\'eron--Severi space $N^1(\overline{X})$ has finite orbits by \cite[Paragraphs 16.1, 16.2]{Lamy}. Denoting by $p:\overline{X}\to X/G$ the quotient map by the action of $\overline{G}$, note that $p^*$ is injective and
$$
N^1(\overline{X})^{\overline{G}} = p^*N^1(X/G), \quad
\Mov(\overline{X}) \cap N^1(\overline{X})^{\overline{G}} = \quad p^*\Mov(X/G) ,
$$
$$\rho^{\overline{G}}(C_{\PsAut^*(\overline{X},\overline{\Delta})}({\overline{G}})) < \quad \PsAut^*(X/G,\Delta_G),$$
where $(X/G,\Delta_G)$ is the quotient pair of $(X,\Delta)$ by $G$ and $\rho^{\overline{G}}$ is obtained by corestricting  the representation of the centralizer
$$\rho\colon C_{\PsAut^*(\overline{X},\overline{\Delta})}({\overline{G}}) \to \GL(N^1(\overline{X}))$$
to the invariant subspace $N^1(\overline{X})^{\overline{G}}$. See Definition \ref{def-rhoG}.

We first prove (1). Recall that $\PsAut^*(\overline{X},\overline{\Delta})$ is contained in 
$$\Gamma := \Aut(\Mov(\overline{X})^{\circ},\mathcal{A},N^1(\overline{X})_{\Z-{\rm Weil}}),$$ where $\mathcal{A}$ denotes the self-dual homogeneous cone given by assumption.
Since $(\overline{X},\overline{\Delta})$ satisfies the movable cone conjecture, there is a rational polyhedral fundamental domain for the action of $\PsAut^*(\overline{X},\overline{\Delta})$ --- and a fortiori for the action of the larger group $\Gamma$ by \cite[Proposition 4.1, Application 4.14]{Loo14} --- on $\Mov^+(\overline{X})$. This implies two things: First, by \cite[Proposition 4.6]{Loo14}, the group $\PsAut^*(\overline{X},\overline{\Delta})$ is of finite index in $\Gamma$. Second, by Proposition \ref{prop-characterize-RPFD}, the cone $\Mov(X)$ is perfectly clipped in $\mathcal{A}$.

We now prove (2). Let us set $\mathcal{C}$ and $\mathcal{A}$ to be the perfectly clipped and the self-dual homogeneous cones given by assumption. By assumption, we can write
\begin{equation}\label{eq-firstinproof}
\overline{\Mov}(\overline{X}) = \overline{\mathcal{C}}\cap\bigcap_{i=1}^m W_i
\end{equation}
for a finite collection $(W_i)$ of rational halfspaces that is invariant under $\PsAut(\overline{X},\overline{\Delta})$ and $\mathrm{Gal}(\overline{k}/k)$. By relabelling, we enforce that the halfspaces $W_i$ that do not contain the invariant subspace $N^1(\overline{X})^{\overline{G}}$ are precisely those indexed by $1\le i\le r$: Then for $1\le i\le r$, the subset $W_i^{\overline{G}}$ is a rational halfspace in $N^1(\overline{X})^{\overline{G}}$, and this collection of halfspaces is clearly invariant under the action of the centralizer $\rho^{\overline{G}}(C_{\PsAut^*(\overline{X},\overline{\Delta})}({\overline{G}}))$.
From Identity \eqref{eq-firstinproof}, we derive the identity
\begin{equation}\label{eq-inproof}
\overline{\Mov}(X/G) = \overline{\mathcal{C}}^{\overline{G}}\cap\bigcap_{i=1}^r W_i^{\overline{G}}.
\end{equation}
To prove the movable cone conjecture for the pair $(X/G,\Delta_G)$, it suffices to show that the smaller group $\rho^{\overline{G}}(C_{\PsAut^*(\overline{X},\overline{\Delta})}({\overline{G}}))$ acts on the cone $\Mov^+(X/G)$ with a rational polyhedral fundamental domain. By Lemma \ref{lem-convexfurtherslice}, we only have to prove that the same centralizer group acts on the cone $(\mathcal{C}^{\overline{G}})^+$ with a rational polyhedral fundamental domain.

We now denote
$$\Gamma := \Aut(\mathcal{C}^{\circ},\mathcal{A},N^1(\overline{X})_{\Z{\rm -Weil}}).$$
Since $\mathcal{C}$ is perfectly clipped in $\mathcal{A}$, by Proposition \ref{prop-characterize-RPFD}, there is a rational polyhedral fundamental domain for the action of $\Gamma$ on $\mathcal{C}^+$. By Theorem \ref{thm-convexmain}, there is a rational polyhedral fundamental domain for the action of $\rho^G(C_{\Gamma}(G))$ on ${\mathcal{C}^G}^+$ as well. Since by assumption, the group $\PsAut^*(\overline{X},\overline{\Delta})$ is of finite index in $\Gamma$, we deduce from Lemma \ref{lem-normalize-commens} and \cite[Proposition 4.6]{Loo14} that the cone ${\mathcal{C}^{\overline{G}}}^+$ also admits a rational polyhedral fundamental domain for the action of $\rho^{\overline{G}}(C_{\PsAut^*(\overline{X},\overline{\Delta})}({\overline{G}}))$, as wished.
\end{proof}

\begin{remark}
In principle, one could use that $\PsAut^*(X/G,\Delta_G)$ contains the larger normalizer subgroup $\rho^G(N_{\PsAut^*(\overline{X},\overline{\Delta})}(G))$. However, working with the centralizer subgroup is somewhat simpler.
\end{remark}

\begin{proof}[Proof of Theorem \ref{thm-main}]
By Theorem \ref{thm-maintechnical} (1), we may set $\mathcal{C}:= \Mov(\overline{X})$ and be in the situation of Theorem \ref{thm-maintechnical} (2). Applying Theorem \ref{thm-maintechnical} (2) now concludes.
\end{proof}

\begin{proof}[Proof of Corollary \ref{cor-boundedness}]
Let $X$ be a torsor over an abelian variety over a perfect field $k$. The base change $\overline{X}$ to the algebraic closure of $k$ is an abelian variety, thus by the work of Prendergast--Smith in \cite{PS12a} has a self-dual homogeneous movable cone and satisfies the movable cone conjecture. Theorem \ref{thm-main} concludes.
\end{proof}

\begin{proof}[Proof of Theorem \ref{thm-ccforexplicitvars}]
Recall that $X = A \times Y \times \prod_{j=1}^s S_j,$
where $A$ is an abelian variety, $Y$ is a product of primitive symplectic varieties with canonical singularities, and each $S_j$ is a smooth rational surface underlying a klt Calabi--Yau pair $(S_j,\Delta_j)$. Let $G$ be a finite subgroup of $\Aut(X,\Delta)$. By \cite[Lemma 4.6]{Druel} the $G$-action preserves the decomposition, up to permutation of potential isomorphic factors.

Regrouping isomorphic factors, we rewrite 
$$Y = \prod_{k=1}^m {Y_k}^{n_k}.$$
By \cite[Corollary 4.4.4]{Pro21}, each $Y_k$ admits a $G\cap\Aut(Y_k)^{n_k}$-equivariant $\Q$-factorial terminalization $\hat{Y_k}$. In particular, 
$$\hat{Y}= \prod_{k=1}^m {\hat{Y_k}}^{n_k}$$ is a $G$-equivariant $\Q$-factorial terminalization of $Y$. We define $\hat{X}:= A\times\hat{Y}\times \prod_{j=1}^s S_j.$
By Lemmas \ref{lem-descent} and \ref{lem-descentbis}, it suffices to show the movable cone conjecture, the finiteness of SQMs, and the nef cone conjecture for all SQMs of the pair $(\hat{X}/G,\hat{\Delta}_G)$. 

To simplify notations, we write $X$ for $\hat{X}$ and $Y$ for $\hat{Y}$ from here on. All primitive symplectic varieties involved have terminal $\Q$-factorial singularities, and $b_2\ge 5$ or dimension 2.
We will now prove the movable cone conjecture for any quotient pair of $(X,\Delta)$ using Theorem \ref{thm-maintechnical} (2).

We first construct the cones $\mathcal{C}$ and $\mathcal{A}$. Denoting by $q_i$ the Beauville--Bogomolov--Fujiki quadratic form for the primitive symplectic variety $Y_i$ and by $f_j$ the intersection form for the surface $S_j$, we set
$$\mathcal{A} := \Mov^{\circ}(A)\oplus \bigoplus_{i=1}^r \mathcal{H}(q_i)\oplus \bigoplus_{j=1}^s \mathcal{H}(f_j).$$
It is a direct sum of self-dual homogeneous cones, thus is self-dual homogeneous itself; see Example \ref{ex-list} Item (7) and Theorem  \ref{thm-koechervinberg}. 
 We now set
 $$\mathcal{C} := \Mov^{\circ}(A)\oplus \bigoplus_{i=1}^r \Mov^{\circ}(Y_i)\oplus \bigoplus_{j=1}^s (\overline{\mathrm{NE}}(S_j)_{K_{S_j}\le 0})^{\vee}.$$
 We checked in Example \ref{ex-list}, see in particular Items (1), (4), (6), (7), and in Lemmas \ref{lem-directsum} and \ref{lem-surface-nicely}, that the cone $\mathcal{C}$ is well-clipped in $\mathcal{A}$.
By Corollary \ref{cor-psautpreserves}, both cones $\mathcal{C}$ and $\mathcal{A}$ are preserved by the action of $\PsAut(X,\Delta)$. 

The known cases of the movable cone conjecture \cite{To10,PS12a,LMP24}, Lemma \ref{lem-surfperfectly} for the surfaces, Proposition \ref{prop-characterize-RPFD}, and Example \ref{ex-listperfclipped} Item (10) show two things: First, that the cone $\mathcal{C}$ is in fact perfectly clipped in $\mathcal{A}$. Second, using also \cite[Proposition 4.6]{Loo14}, that the group
$$\Aut^*(A)\times \prod_{i=1}^r \PsAut^*(Y_i)\times \prod_{j=1}^s \Aut^*(S_j,\Delta_j)$$
is of finite index in the group $\Aut(\mathcal{C}^{\circ},\mathcal{A},N^1(X)_{\Z\mbox{\footnotesize{-Weil}}})$. A fortiori, the larger group $\PsAut^*(X,\Delta)$ is of finite index in $\Aut(\mathcal{C}^{\circ},\mathcal{A},N^1(X)_{\Z\mbox{\footnotesize{-Weil}}})$ too.

We finally relate the movable cone of $X$ to the cone $\mathcal{C}$. Lemma \ref{lem-mov-product} and \cite[Exercise III.12.6]{HarBook} show that the movable cone of $X$ is the direct sum of the movable cones of its factors. Regrouping isomorphic factors, we rewrite 
$$\prod_{j=1}^s S_j = \prod_{k=1}^m {S_k}^{n_k}.$$ 
Lemma \ref{lem-surface-nicely} allows us to describe the movable cone of $X$ as follows:
$$\overline{\Mov}(X) = \overline{\mathcal{C}}\cap \bigcap_{k=1}^m \bigcap_{d=1}^{n_k} \mathrm{pr}_{d,S_k}^*\bigcap_{u\in U_k} W_u,$$
where $(W_u)_{u\in U_k}$ denotes the finite collection of rational halfspaces given by Lemma \ref{lem-surface-nicely} for the surface $S_k$ and $\mathrm{pr}_{d,S_k}$ denotes the projection from $X$ to the $d$-th factor isomorphic to $S_k$. In $N^1(X)$, the collection of halfspaces of the form $\mathrm{pr}_{d,S_k}^*W_u$, indexed by $1\le k\le m$, $1\le  d\le n_k$ and $u\in U_k$, is in fact invariant under the action of the group $\PsAut(X)$ by Corollary \ref{cor-psautpreserves}.

We have verified the assumptions of Theorem \ref{thm-maintechnical} (2), which establishes the movable cone conjecture for any quotient pair of $(X,\Delta)$.

\medskip

We now focus on the finiteness of SQMs and the nef cone conjecture for them. Note that
$$\mathrm{Mov}^{\circ}(X) \setminus \bigcup_{w\in W} w^{\perp} = \bigsqcup_{\substack{\alpha: X\dashrightarrow X'\\ \mathrm{SQM}}}\alpha^*\mathrm{Amp}(X'),$$
where $W$ is the set of pullbacks of primitive wall divisors of the primitive symplectic terminal $\Q$-factorial factors (see \cite[Definition 7.1]{LMP24}). By \cite[Proposition 7.7]{LMP24}, the set of squares $\{q(w)\mid w\in W\}$ is bounded and contained in $\Z_{< 0}$. 
Restricting to the $G$-invariant Néron-Severi subspace, we still have
$$\mathrm{Mov}^{\circ}(X/G) \setminus \bigcup_{w\in W_G} w^{\perp} = \bigsqcup_{\substack{\alpha: X/G\dashrightarrow Y\\ \mathrm{SQM}}}\alpha^*\mathrm{Amp}(Y),$$
where we define 
$$W_G
:=
\left\lbrace\sum_{g\in G} gw \mid w\in W,\, q(w,\sum_{g\in G} gw) < 0\right\rbrace.$$
The elements of $W_G$ clearly remain of bounded negative squares, thus by \cite[Proposition 3.4]{MY15}, for any rational polyhedral cone $\Pi$, only finitely many of their orthogonal hyperplanes intersect $\Pi^{\circ}$. This checks the assumption of Lemma \ref{lem-finitesqm}, and applying it to the pair $(X/G,\Delta_G)$ concludes the proof.
\end{proof}

\section{Examples}\label{sec-examples}
We conclude with a few examples of varieties to which Theorem \ref{thm-ccforexplicitvars} applies.

\medskip

Our first example features a rational surface $\Sigma$ underlying a Calabi--Yau pair. The cone conjecture is already known for this pair by the general results of Totaro \cite{To10}, so this result is not novel. We still include it as one of the simplest illustration of Theorem \ref{thm-ccforexplicitvars} in the presence of a group action with ramification in codimension $1$.

\begin{example}
Let $\left(S,\frac{1}{3}(C_1+\ldots+C_9)\right)$ be the klt Calabi--Yau pair constructed in Example \ref{ex-rationalsurface}. Recall that it is obtained as the minimal resolution of the quotient $S_0=E_j\times E_j/\langle\mathrm{diag}(j,j)\rangle$, where $j$ denotes the first third root of unity. 
The action of the matrix $\mathrm{diag}(1,j)$ on the product $E_j\times E_j$ descends to the quotient $S_0$ and lifts to an automorphism $g_3$ of the minimal resolution $S$. Let $T_j$ denote the set of fixed points of the multiplication by $j$ on $E_j$ and consider the six rational curves obtained by taking the image of
$$E_j\times T_j \;\cup\; T_j\times E_j$$
in the quotient $S_0$. Their strict transform consists of six rational curves $F_1,\ldots,F_6$ on $S$, which are pointwise fixed by $g_3$. In fact, this is precisely the ramification locus of the quotient map 
$$p\colon S\to S/\langle g_3\rangle.$$ 
Let $\Sigma$ denote the quotient $S/\langle g_3\rangle$ and consider the quotient pair $(\Sigma,\Delta)$ with
$$\Delta= \frac{1}{3}\left(p(C_1)+\ldots+p(C_9)\right) + \frac{2}{3}\left(p(F_1)+\ldots+p(F_6)\right).$$
The cone conjecture holds for $(\Sigma,\Delta)$ by \cite{To10}.

We can use the proofs of Theorems \ref{thm-ccforexplicitvars} and \ref{thm-maintechnical} to say slightly more: the cone $\Nef(\Sigma)$ is rational polyhedral. Indeed, we prove that the action of the centralizer subgroup
$C_{\Aut(S)}(\langle g_3\rangle)$ on the cone $\Nef(\Sigma)$ admits a rational polyhedral fundamental domain.
Note however that 
\begin{align*}
C_{\Aut(S)}(\langle g_3\rangle)
&\cong C_{\Aut(S_0)}(\langle\mathrm{diag}(1,j)\rangle)\\
&\cong C_{\Aut(E_j\times E_j)}(\langle\mathrm{diag}(j,j),\mathrm{diag}(1,j)\rangle)\\
&\cong (T_j \times T_j) \rtimes C_{\GL_2(\Z[j])}(\langle\mathrm{diag}(1,j)\rangle)\\
&\simeq (T_j \times T_j) \rtimes (\Z[j]^{\times}\times \Z[j]^{\times})
\end{align*}
where $T_j\simeq \Z/3\Z$ is the set of fixed points of the multiplication by $j$ on $E_j$ as above and $\Z[j]^{\times}\simeq \Z/6\Z$ is the group of units in the ring $\Z[j]$. This group is finite, so the cone $\Nef(\Sigma)$ is rational polyhedral.
\end{example}

The next example is due to Oguiso (\cite[Example 2]{Ogu96}). The cone conjecture was not previously known in this case: we derive it from Theorem \ref{thm-ccforexplicitvars}. However, we can also describe the nef and movable cones by hand. In our opinion, Theorem \ref{thm-ccforexplicitvars} is not strictly essential here; it just allows to skip some tedious computations when finding a covering rational polyhedral domain, after the nef cone and the automorphism group are described.  

\begin{example}\label{ex-jac}
Let $C$ denote the Klein quartic curve, that is the vanishing locus of the equation $x^3y+y^3z+z^3x$ in $\mathbb{P}^2$. Denote by $J(C)$ its Jacobian, which is an abelian threefold. The automorphism of $C$ obtained by restricting
$$g_7\colon [x:y:z]\in\P^2 \mapsto [\zeta_7x:{\zeta_7}^5y:z]\in\P^2$$
induces an automorphism of $J(C)$, which we denote by $h_7$. It has order $7$ and fixes exactly $7$ points of $J(C)$. 

We consider the quotient $\overline{X_7}=J(C)/\langle h_7\rangle$. It is a Calabi--Yau threefold with canonical singularities, it has Picard number $3$ (\cite[Claim 3.2 and Key Claim (2)]{Ogu96}), and it admits a unique crepant resolution $X_7$, which is a rigid Calabi--Yau threefold of Picard number $24$ (\cite[Example 2]{Ogu96}.) 

While Theorem \ref{thm-ccforexplicitvars} proves the movable and nef cone conjectures for $\overline{X_7}$, the cone conjecture remains open for $X_7$. Alternatively, one can explicitly compute the nef and movable cones, as well as the automorphism group of $\overline{X_7}$. These cones and groups are simple enough to describe that one could write down an explicit rational polyhedral cone $\Pi\subset\Nef^+(\overline{X_7})$ whose $\mathrm{Aut}(\overline{X_7})$-translates cover the cone $\Nef^+(\overline{X_7})$, thereby foregoing the use of Theorem \ref{thm-ccforexplicitvars}. To make this point clear, we will now describe the nef and movable cones of $\overline{X_7}$, as well as a finite index subgroup of $\mathrm{Aut}(\overline{X_7})$.

\medskip

Choosing a fixed point of $h_7$ as the origin of $J(C)$ allows us to view $h_7$ as a linear automorphism of the tangent space to $J(C)$, that is the dual space to $H^0(K_C)$. Using the Poincaré residue formula, an elementary computation shows that
$h_7$ acts with the diagonal matrix $\mathrm{diag}(\zeta_7,{\zeta_7}^2,{\zeta_7}^4)$ on the tangent space of $J(C)$. One can lift this action to the universal cover $\C^3$ of $J(C)$, where the eigenspace decomposition gives rise to a natural choice of three linear coordinates $z_1,z_2,z_3$ on $\C^3$. The three closed $(1,1)$-forms
$$i\de z_1\wedge\de \overline{z_1}
,\quad
i\de z_2\wedge\de \overline{z_2}
,\quad
i\de z_3\wedge\de \overline{z_3}$$
are clearly invariant under translation and under the action of $\mathrm{diag}(\zeta_7,{\zeta_7}^2,{\zeta_7}^4)$, thus descend to the quotient $\overline{X_7}$, where we denote their cohomology classes by $D_1,D_2,D_3$. These are clearly nef; see for instance \cite[Lemma 1.5]{DELV}. They generate the N\'eron--Severi space of $\overline{X_7}$ by \cite[Claim 3.2 and Key Claim (2)]{Ogu96}.
Since we have $(D_i)^2 \equiv 0$ in $H^4(\overline{X_7};\Q)$ for $i=1,2,3$, the nef and pseudoeffective cones of $\overline{X_7}$ coincide and are spanned by the classes $D_1 ,D_2, D_3.$ In particular, the nef and movable cones of $\overline{X_7}$ also coincide. They are both clearly polyhedral.

We however claim that there is no non-zero rational point on the boundary of $\Nef(\overline{X_7})$. Before proving that claim, we describe the automorphism group of $\overline{X_7}$. The endomorphisms
$$\id + h_7,\quad \id+{h_7}^2,\quad \id+{h_7}^4$$
of the abelian variety $J(C)$ act with determinant
$(1+\zeta_7)(1+{\zeta_7}^2)(1+{\zeta_7}^4) = 1,$ and are thus automorphisms of $J(C)$. Since they commute with $h_7$ and with each other, they descend to automorphisms of $\overline{X_7}$, and generate a free abelian group of rank $3$ in $\Aut(\overline{X_7})$. It is easy to show that this inclusion of groups has finite index.

Note that the three classes $D_1,D_2,D_3$ are eigenvectors of $\id + {h_7}$, with eigenvalues 
$$|1+\zeta_7|^2, \quad |1+{\zeta_7}^2|^2,\quad |1+{\zeta_7}^4|^2.$$
Since these eigenvalues are distinct from $\pm 1$ and since $\id + h_7$ preserves the ray spanned by $D_i$ as well as the lattice of integral Weil divisor classes, there is no integral Weil divisor class on the ray spanned by $D_i$. This shows that $D_i$ spans an irrational ray for $i=1,2,3$.
More generally, the face of the nef cone spanned by $D_i$ and $D_j$, for $i\neq j$, satisfies one of the following:
\begin{itemize}
\item[a)] it contains no non-zero rational point, 
\item[b)] or it is a rational $2$-plane,
\item[c)] or it contains exactly one rational line.
\end{itemize}
If Option c) occurs, then $\id +h_7$ preserves the one rational line, which is thus an eigenline of $\id +h_7$. This contradicts the irrationality of the rays spanned by $D_i$ and $D_j$. If Option b) occurs, then the characteristic polynomial of $(\id+h_7)|_{\R D_i + \R D_j}$ has integral coefficients. But its trace is clearly not an integer; this is a contradiction. So we are left with Option a), that is, the nef cone of $\overline{X_7}$ indeed has no non-zero rational boundary point.
\end{example}

The next family of examples is inspired by Borcea \cite{Bor97}, Voisin \cite{Voi93}, and Uehara \cite[Section 4]{Ueh}. Theorem \ref{thm-ccforexplicitvars} provides the first proof of the cone conjecture for this family of examples.

\begin{example}\label{ex-ueh}
Let $S$ be a normal complex projective surface with canonical singularities, with $h^1(\mathcal{O}_S)=0$, and with trivial canonical divisor $K_S=0$. Suppose that $S$ admits an automorphism $g$ of order $d\in\{2,3,4,6\}$ acting non-trivially on the one-dimensional space of sections $h^0(K_S)$. It acts by multiplication by a root of unity, say $\zeta$.
Let $E$ be an elliptic curve, which, if $\zeta\neq -1$, we further assume to have complex multiplication by $\zeta$. We define a diagonal action of the cyclic group $G=\Z/d\Z$ on $S\times E$ by $g$ on $S$ and multiplication by $\zeta^{-1}$ on $E$. The quotient threefold
$$\overline{X} = (S\times E)/G,$$
has canonical singularities and trivial canonical divisor. Since $S$ is a primitive symplectic surface with canonical singularities and $E$ is an elliptic curve, Theorem \ref{thm-ccforexplicitvars} yields the movable cone conjecture for $\overline{X}$. Note that $\rho(\overline{X})=\rho(S/G) + 1$.

Note that by \cite[Section 4]{Ueh}, the threefold $\overline{X}$ admits a crepant resolution $X$. In the case when $S$ is a smooth K3 surface and $d=2$, the Hodge numbers of $X$ are computed and its deformations are studied in \cite[Section 1]{Voi93}. A partial mirror symmetry phenomenon is even established for $X$ in \cite[Théorème 2.17; see also 2.21 and 2.22]{Voi93}. The movable cone conjecture is still not known for $X$ in general. 

\medskip

To conclude this example, we present a few interesting choices for the K3 surface $S$, the curve $E$ and the group $G$. In \cite[Theorem 7.10]{Lutz}, Lutz studies the case when $S$ is a double cover of $\P^2$ branched over a smooth general sextic, $G$ is generated by the corresponding deck involution, and $E$ is arbitrary. He proves the cone conjecture for the resolution $X$ and its deformations. Note that $\rho(X)=6$ here.

Another interesting choice is due to Uehara in \cite[Example 4.3 (ii)]{Ueh}, and we now recall it as presented in the reference. Consider a Kummer surface of the form $\mathrm{Kum}(E\times F)$, for $E,F$ general non-isogeneous elliptic curves. We denote by $E[2]$ and $F[2]$ their respective four $2$-torsion points. In this Kummer surface, we consider the strict transform of the image of the locus 
$$E\times F[2] \quad \cup\quad  E[2]\times F$$
under the quotient map $E\times F\to E\times F/\langle -\id_{E\times F}\rangle$. It consists of eight disjoint $(-2)$-curves. Let $S$ be the surface obtained by contracting these eight curves: it is a singular K3 surface with eight singular points of type $A_1$ and with Picard number $10$.
The involution $(\id_{E_1},-\mathrm{id}_{E_2})$ on $E_1\times E_2$ lifts to $\mathrm{Kum}(E_1\times E_2)$, then descends to an antisymplectic involution $g$ on the singular K3 surface $S$. The quotient threefold $\overline{X}$ associated to $S$, an arbitrary elliptic curve $C$, and the cyclic action by $G=\langle(g,-\id_C)\rangle$ fits our setting. Note that $\rho(\overline{X})=11$. Initially, Uehara constructed this threefold because it admits a crepant resolution $X$ with infinitely many small contractions \cite[Proposition 4.1 b)]{Ueh}. The cone conjecture is still open for this resolution.
\end{example}

Our last example is inspired by Beauville (\cite[3.1, 3.5]{Beauv11}). It is a quotient of a hyperkähler fourfold by an antisymplectic involution. In general, it can be difficult to control the fixed locus of an antisymplectic involution on a hyperkähler manifold; see \cite{Beauv11,FMGS}. Convienently though, Theorem \ref{thm-ccforexplicitvars} applies to any finite quotient of a hyperkähler manifold, regardless of the geometry of the fixed locus. To our knowledge, the cone conjecture was not previously known in this case.

\begin{example}\label{ex-hk}
Let $S$ be a smooth quartic surface without lines in $\P^3$. On its Hilbert square $S^{[2]}$, consider the so-called \textit{Beauville involution} $\iota$ that sends a subscheme $Z$ of length $2$ to the residual intersection of the unique line through $Z$ with the quartic surface $S$. The quotient fourfold 
$$X = S^{[2]}/\langle \iota\rangle$$
has canonical singularities. Note that $\rho(X)=1$, thus $X$ automatically satisfies the cone conjecture.

\medskip

Let us make a few changes to construct an example with higher Picard number --- and thus, hopefully, a more interesting geometry. Let $S$ be the smooth quartic surface given in $\P^3$ by the equation
$${x_0}^4+{x_1}^4+{x_2}^4+{x_3}^4 + 12x_0x_1x_2x_3 = 0,$$
see \cite{Muk88,BS20} for more context on this K3 surface, notably on the action of the Mathieu group $M_{20}$ on it. This surface contains no lines (\cite[Proof of Theorem 1.1]{Deg22}) and has Picard rank 20 (\cite[Proposition 4.4]{BS20}).
Let $\sigma$ denote the involution induced by restriction of
$$[x_0:x_1:x_2:x_3]\in\P^3\mapsto [-x_0:-x_1:x_2:x_3]\in\P^3.$$
It is symplectic and has eight fixed points on $S$. Since $\sigma$ is linear, the induced involution $\sigma^{[2]}$ on the Hilbert square $S^{[2]}$ commutes with the Beauville involution $\iota$.

The N\'eron--Severi space of $S^{[2]}$ decomposes as
$$N^1(S^{[2]})_{\R} \;\cong\; N^1(S)_{\R}\oplus \R\cdot [E],$$
where $E$ denotes the exceptional divisor of the blow-up $S^{[2]}\to \mathrm{Sym}^2 S$. Hence, the fourfold $S^{[2]}$ has Picard rank $21$. Pullback by the Beauville involution $\iota$ fixes the multiples of the class $L:= c_1(\mathcal{O}_{\P^3}(1)|_S) - \frac{1}{2}E$, and acts with eigenvalue $-1$ on a hyperplane in $N^1(S^{[2]})_{\R}$. Clearly, pullback by the involution $\sigma^{[2]}$ also fixes the class $L$. Moreover, the Lefschetz trace formula shows that $\sigma$ acts with eigenvalues $1$ and $-1$ on a $12$-dimensional and on an $8$-dimensional eigenspace in $N^1(S)_{\R}$ respectively.
From there, an easy computation shows that the antisymplectic involution $\tau=\iota\circ\sigma^{[2]}$ preserves a $9$-dimensional subspace of $N^1(S^{[2]})_{\R}$.
The quotient fourfold 
$$Y = S^{[2]}/\langle \tau\rangle$$
thus has Picard number $9$. 

Since $S^{[2]}$ is a primitive symplectic variety, Theorem \ref{thm-ccforexplicitvars} proves the movable cone conjecture, the nef cone conjecture, and the finiteness of isomorphism classes of minimal models for $Y$.

\medskip

Note that describing the nef or the movable cone of $Y$ seems extremely challenging: to begin with, it is a $9$-dimensional slice of the $21$-dimensional movable cone of $S^{[2]}$. Furthermore, since $S$ is a Kummer surface associated with the product of two elliptic curves by \cite[Proposition 4.5]{BS20}, it has infinitely many automorphisms \cite{KK01}, hence the Hilbert square $S^{[2]}$ has infinitely many automorphisms. This shows that the nef and movable cones of $S^{[2]}$ are not rational polyhedral. To our knowledge, these cones have not been explicitly described. 

Even though the cones associated to $S^{[2]}$ are quite complicated, one may expect $9$-dimensional slices of them to be much simpler. This motivates a daring question: Is the nef cone of $Y$ rational polyhedral? Theorem \ref{thm-convexmain} allows to formulate an equivalent, yet apparently simpler question: Is the centralizer of the involution $\tau$ in $\Aut(S^{[2]})$ a finite group? A first step on this path would be to ascertain whether the centralizer of $\tau$ in the subgroup $\Aut(S)$ is finite. 
\end{example}

\bibliographystyle{plain}
\bibliography{biblioCCEnriques}

\end{document}